\documentclass[11pt,oneside,english]{amsart}
\usepackage[T1]{fontenc}
\usepackage[utf8]{inputenc}
\usepackage{geometry}
\usepackage{units}
\usepackage{amstext}
\usepackage{amsthm}
\usepackage{amssymb}
\usepackage{stmaryrd}
\usepackage[all]{xy}
\usepackage{todonotes}

\makeatletter

\providecommand{\tabularnewline}{\\}

\numberwithin{equation}{section}
\numberwithin{figure}{section}
\theoremstyle{plain}
\newtheorem{thm}{\protect\theoremname}[section]
\theoremstyle{definition}
\newtheorem{defn}[thm]{\protect\definitionname}
\theoremstyle{remark}
\newtheorem{rem}[thm]{\protect\remarkname}
\theoremstyle{plain}
\newtheorem{lem}[thm]{\protect\lemmaname}
\theoremstyle{plain}
\newtheorem{cor}[thm]{\protect\corollaryname}
\theoremstyle{plain}
\newtheorem{prop}[thm]{\protect\propositionname}
\theoremstyle{definition}
\newtheorem{example}[thm]{\protect\examplename}
\theoremstyle{plain}
\newtheorem{conjecture}[thm]{\protect\conjecturename}

\usepackage{hyperref}

\makeatother

\usepackage{babel}
\providecommand{\conjecturename}{Conjecture}
\providecommand{\corollaryname}{Corollary}
\providecommand{\definitionname}{Definition}
\providecommand{\examplename}{Example}
\providecommand{\lemmaname}{Lemma}
\providecommand{\propositionname}{Proposition}
\providecommand{\remarkname}{Remark}
\providecommand{\theoremname}{Theorem}

\begin{document}
\title{The Batyrev--Manin conjecture for DM stacks}
\author{Ratko Darda and Takehiko Yasuda}
\address{Department of Mathematics, Graduate School of Science, Osaka University
Toyonaka, Osaka 560-0043, JAPAN}
\email{ratko.darda@gmail.com}
\address{Department of Mathematics, Graduate School of Science, Osaka University
Toyonaka, Osaka 560-0043, JAPAN}
\email{yasuda.takehiko.sci@osaka-u.ac.jp}
\begin{abstract}
We define a new height function on rational points of a DM (Deligne--Mumford)
stack over a number field. This generalizes a generalized discriminant
of Ellenberg--Venkatesh, the height function recently introduced
by Ellenberg--Satriano--Zureick-Brown (as far as DM stacks over
number fields are concerned), and the quasi-toric height function
on weighted projective stacks by Darda. Generalizing the Manin conjecture
and the more general Batyrev--Manin conjecture, we formulate a few
conjectures on the asymptotic behavior of the number of rational points
of a DM stack with bounded height. To formulate the Batyrev--Manin
conjecture for DM stacks, we introduce the orbifold versions of the
so-called $a$- and $b$-invariants. When applied to the classifying
stack of a finite group, these conjectures specialize to Malle's
conjecture, except that we remove certain thin subsets from counting.
More precisely, we remove breaking thin subsets, which have been studied
in the case of varieties by people including Hassett, Tschinkel, Tanimoto,
Lehmann and Sengupta, and can be generalized to DM stack thanks to
our generalization of $a$- and $b$-invariants. The breaking thin
subset enables us to reinterpret Klüners' counterexample to Malle's
conjecture. 
\end{abstract}

\maketitle
\global\long\def\bigmid{\mathrel{}\middle|\mathrel{}}%

\global\long\def\AA{\mathbb{A}}%

\global\long\def\CC{\mathbb{C}}%

\global\long\def\FF{\mathbb{F}}%

\global\long\def\GG{\mathbb{G}}%

\global\long\def\LL{\mathbb{L}}%

\global\long\def\MM{\mathbb{M}}%

\global\long\def\NN{\mathbb{N}}%

\global\long\def\PP{\mathbb{P}}%

\global\long\def\QQ{\mathbb{Q}}%

\global\long\def\RR{\mathbb{R}}%

\global\long\def\SS{\mathbb{S}}%

\global\long\def\ZZ{\mathbb{Z}}%

\global\long\def\bA{\mathbf{A}}%

\global\long\def\ba{\mathbf{a}}%

\global\long\def\bb{\mathbf{b}}%

\global\long\def\bd{\mathbf{d}}%

\global\long\def\bf{\mathbf{f}}%

\global\long\def\bg{\mathbf{g}}%

\global\long\def\bh{\mathbf{h}}%

\global\long\def\bj{\mathbf{j}}%

\global\long\def\bm{\mathbf{m}}%

\global\long\def\bp{\mathbf{p}}%

\global\long\def\bq{\mathbf{q}}%

\global\long\def\br{\mathbf{r}}%

\global\long\def\bs{\mathbf{s}}%

\global\long\def\bt{\mathbf{t}}%

\global\long\def\bv{\mathbf{v}}%

\global\long\def\bw{\mathbf{w}}%

\global\long\def\bx{\boldsymbol{x}}%

\global\long\def\by{\boldsymbol{y}}%

\global\long\def\bz{\mathbf{z}}%

\global\long\def\cA{\mathcal{A}}%

\global\long\def\cB{\mathcal{B}}%

\global\long\def\cC{\mathcal{C}}%

\global\long\def\cD{\mathcal{D}}%

\global\long\def\cE{\mathcal{E}}%

\global\long\def\cF{\mathcal{F}}%

\global\long\def\cG{\mathcal{G}}%

\global\long\def\cH{\mathcal{H}}%

\global\long\def\cI{\mathcal{I}}%

\global\long\def\cJ{\mathcal{J}}%

\global\long\def\cK{\mathcal{K}}%

\global\long\def\cL{\mathcal{L}}%

\global\long\def\cM{\mathcal{M}}%

\global\long\def\cN{\mathcal{N}}%

\global\long\def\cO{\mathcal{O}}%

\global\long\def\cP{\mathcal{P}}%

\global\long\def\cQ{\mathcal{Q}}%

\global\long\def\cR{\mathcal{R}}%

\global\long\def\cS{\mathcal{S}}%

\global\long\def\cT{\mathcal{T}}%

\global\long\def\cU{\mathcal{U}}%

\global\long\def\cV{\mathcal{V}}%

\global\long\def\cW{\mathcal{W}}%

\global\long\def\cX{\mathcal{X}}%

\global\long\def\cY{\mathcal{Y}}%

\global\long\def\cZ{\mathcal{Z}}%

\global\long\def\fa{\mathfrak{a}}%

\global\long\def\fb{\mathfrak{b}}%

\global\long\def\fc{\mathfrak{c}}%

\global\long\def\ff{\mathfrak{f}}%

\global\long\def\fj{\mathfrak{j}}%

\global\long\def\fm{\mathfrak{m}}%

\global\long\def\fp{\mathfrak{p}}%

\global\long\def\fs{\mathfrak{s}}%

\global\long\def\ft{\mathfrak{t}}%

\global\long\def\fx{\mathfrak{x}}%

\global\long\def\fv{\mathfrak{v}}%

\global\long\def\fD{\mathfrak{D}}%

\global\long\def\fJ{\mathfrak{J}}%

\global\long\def\fG{\mathfrak{G}}%

\global\long\def\fK{\mathfrak{K}}%

\global\long\def\fM{\mathfrak{M}}%

\global\long\def\fO{\mathfrak{O}}%

\global\long\def\fS{\mathfrak{S}}%

\global\long\def\fV{\mathfrak{V}}%

\global\long\def\fX{\mathfrak{X}}%

\global\long\def\fY{\mathfrak{Y}}%

\global\long\def\ru{\mathrm{u}}%

\global\long\def\rv{\mathbf{\mathrm{v}}}%

\global\long\def\rw{\mathrm{w}}%

\global\long\def\rx{\mathrm{x}}%

\global\long\def\ry{\mathrm{y}}%

\global\long\def\rz{\mathrm{z}}%

\global\long\def\a{\mathrm{a}}%

\global\long\def\AdGp{\mathrm{AdGp}}%

\global\long\def\Aff{\mathbf{Aff}}%

\global\long\def\Alg{\mathbf{Alg}}%

\global\long\def\age{\operatorname{age}}%

\global\long\def\Ann{\mathrm{Ann}}%

\global\long\def\Aut{\operatorname{Aut}}%

\global\long\def\B{\operatorname{\mathrm{B}}}%

\global\long\def\Bl{\mathrm{Bl}}%

\global\long\def\c{\mathrm{c}}%

\global\long\def\C{\operatorname{\mathrm{C}}}%

\global\long\def\calm{\mathrm{calm}}%

\global\long\def\center{\mathrm{center}}%

\global\long\def\characteristic{\operatorname{char}}%

\global\long\def\cjun{c\textrm{-jun}}%

\global\long\def\codim{\operatorname{codim}}%

\global\long\def\Coker{\mathrm{Coker}}%

\global\long\def\Conj{\operatorname{Conj}}%

\global\long\def\D{\mathrm{D}}%

\global\long\def\Df{\mathrm{Df}}%

\global\long\def\diag{\mathrm{diag}}%

\global\long\def\det{\operatorname{det}}%

\global\long\def\discrep#1{\mathrm{discrep}\left(#1\right)}%

\global\long\def\doubleslash{\sslash}%

\global\long\def\E{\operatorname{E}}%

\global\long\def\Emb{\operatorname{Emb}}%

\global\long\def\et{\textrm{ét}}%

\global\long\def\etop{\mathrm{e}_{\mathrm{top}}}%

\global\long\def\el{\mathrm{e}_{l}}%

\global\long\def\Exc{\mathrm{Exc}}%

\global\long\def\FConj{F\textrm{-}\Conj}%

\global\long\def\Fitt{\operatorname{Fitt}}%

\global\long\def\Fr{\mathrm{Fr}}%

\global\long\def\Gal{\operatorname{Gal}}%

\global\long\def\GalGps{\mathrm{GalGps}}%

\global\long\def\GL{\mathrm{GL}}%

\global\long\def\Grass{\mathrm{Grass}}%

\global\long\def\H{\operatorname{\mathrm{H}}}%

\global\long\def\hattimes{\hat{\times}}%

\global\long\def\hatotimes{\hat{\otimes}}%

\global\long\def\Hilb{\mathrm{Hilb}}%

\global\long\def\Hodge{\mathrm{Hodge}}%

\global\long\def\Hom{\operatorname{Hom}}%

\global\long\def\hyphen{\textrm{-}}%

\global\long\def\I{\operatorname{\mathrm{I}}}%

\global\long\def\id{\mathrm{id}}%

\global\long\def\Image{\operatorname{\mathrm{Im}}}%

\global\long\def\ind{\mathrm{ind}}%

\global\long\def\injlim{\varinjlim}%

\global\long\def\Inn{\mathrm{Inn}}%

\global\long\def\iper{\mathrm{iper}}%

\global\long\def\Iso{\operatorname{Iso}}%

\global\long\def\isoto{\xrightarrow{\sim}}%

\global\long\def\J{\operatorname{\mathrm{J}}}%

\global\long\def\Jac{\mathrm{Jac}}%

\global\long\def\kConj{k\textrm{-}\Conj}%

\global\long\def\KConj{K\textrm{-}\Conj}%

\global\long\def\Ker{\operatorname{Ker}}%

\global\long\def\Kzero{\operatorname{K_{0}}}%

\global\long\def\lcr{\mathrm{lcr}}%

\global\long\def\lcm{\operatorname{\mathrm{lcm}}}%

\global\long\def\length{\operatorname{\mathrm{length}}}%

\global\long\def\M{\operatorname{\mathrm{M}}}%

\global\long\def\MC{\mathrm{MC}}%

\global\long\def\MHS{\mathbf{MHS}}%

\global\long\def\mld{\mathrm{mld}}%

\global\long\def\mod#1{\pmod{#1}}%

\global\long\def\Mov{\overline{\mathrm{Mov}}}%

\global\long\def\mRep{\mathbf{mRep}}%

\global\long\def\mult{\mathrm{mult}}%

\global\long\def\N{\operatorname{\mathrm{N}}}%

\global\long\def\Nef{\mathrm{Nef}}%

\global\long\def\nor{\mathrm{nor}}%

\global\long\def\nr{\mathrm{nr}}%

\global\long\def\NS{\mathrm{NS}}%

\global\long\def\op{\mathrm{op}}%

\global\long\def\orb{\mathrm{orb}}%

\global\long\def\ord{\operatorname{ord}}%

\global\long\def\P{\operatorname{P}}%

\global\long\def\PEff{\overline{\mathrm{Eff}}}%

\global\long\def\PGL{\mathrm{PGL}}%

\global\long\def\pt{\mathbf{pt}}%

\global\long\def\pur{\mathrm{pur}}%

\global\long\def\perf{\mathrm{perf}}%

\global\long\def\perm{\mathrm{perm}}%

\global\long\def\Pic{\mathrm{Pic}}%

\global\long\def\pr{\mathrm{pr}}%

\global\long\def\Proj{\operatorname{Proj}}%

\global\long\def\projlim{\varprojlim}%

\global\long\def\Qbar{\overline{\QQ}}%

\global\long\def\QConj{\mathbb{Q}\textrm{-}\Conj}%

\global\long\def\R{\operatorname{\mathrm{R}}}%

\global\long\def\Ram{\operatorname{\mathrm{Ram}}}%

\global\long\def\rank{\operatorname{\mathrm{rank}}}%

\global\long\def\rat{\mathrm{rat}}%

\global\long\def\Ref{\mathrm{Ref}}%

\global\long\def\rig{\mathrm{rig}}%

\global\long\def\red{\mathrm{red}}%

\global\long\def\reg{\mathrm{reg}}%

\global\long\def\rep{\mathrm{rep}}%

\global\long\def\Rep{\mathbf{Rep}}%

\global\long\def\sbrats{\llbracket s\rrbracket}%

\global\long\def\Sch{\mathbf{Sch}}%

\global\long\def\sep{\mathrm{sep}}%

\global\long\def\Set{\mathbf{Set}}%

\global\long\def\sing{\mathrm{sing}}%

\global\long\def\sm{\mathrm{sm}}%

\global\long\def\SL{\mathrm{SL}}%

\global\long\def\Sp{\operatorname{Sp}}%

\global\long\def\Spec{\operatorname{Spec}}%

\global\long\def\Spf{\operatorname{Spf}}%

\global\long\def\ss{\mathrm{ss}}%

\global\long\def\st{\mathrm{st}}%

\global\long\def\Stab{\operatorname{Stab}}%

\global\long\def\Supp{\operatorname{Supp}}%

\global\long\def\spars{\llparenthesis s\rrparenthesis}%

\global\long\def\Sym{\mathrm{Sym}}%

\global\long\def\T{\operatorname{T}}%

\global\long\def\tame{\mathrm{tame}}%

\global\long\def\tbrats{\llbracket t\rrbracket}%

\global\long\def\top{\mathrm{top}}%

\global\long\def\tors{\mathrm{tors}}%

\global\long\def\tpars{\llparenthesis t\rrparenthesis}%

\global\long\def\Tr{\mathrm{Tr}}%

\global\long\def\ulAut{\operatorname{\underline{Aut}}}%

\global\long\def\ulHom{\operatorname{\underline{Hom}}}%

\global\long\def\ulInn{\operatorname{\underline{Inn}}}%

\global\long\def\ulIso{\operatorname{\underline{{Iso}}}}%

\global\long\def\ulSpec{\operatorname{\underline{{Spec}}}}%

\global\long\def\Utg{\operatorname{Utg}}%

\global\long\def\Unt{\operatorname{Unt}}%

\global\long\def\Var{\mathbf{Var}}%

\global\long\def\Vol{\mathrm{Vol}}%

\global\long\def\Y{\operatorname{\mathrm{Y}}}%

\tableofcontents{}

\section{Introduction}

The Batyrev--Manin conjecture \cite{franke1989rational,batyrev1990surle}
and Malle's conjecture \cite{malle2002onthe,malle2004onthe} are
two important conjectures about asymptotic behaviors of arithmetic
objects; rational points of a variety and $G$-extensions of a number
field for a prescribed transitive subgroup $G$ of the symmetric group
$S_{n}$. There is a clear similarity between them. Under a proper
setting, for each positive real number $B>0$, the number of rational
points with height at most $B$ is finite. Similarly for the number
of $G$-extensions of a number field with discriminant at most $B$.
In both cases, if we denote these numbers by $N(B)$, then the conjectures
claim asymptotic formulas of the form 
\[
N(B)\sim CB^{a}(\log B)^{b-1}\quad(B\to\infty)
\]
for some constants $C>0$, $a>0$ and $b\ge1$, where $a$ and $b$
are invariants admitting simple expressions. There are also works
on searching for formulas of the leading constant $C$, notably \cite{peyre1995hauteurs,batyrev1998tamagawa}
for rational points and \cite{bhargava2007massformulae} for $S_{n}$-extensions. 

In \cite{yasuda2014densities,yasuda2015maninsconjecture}, the second-named
author studied relations between these conjectures from the viewpoint
of the McKay correspondence. In particular, he showed that constants
$a$ and $b$ in the two conjectures are closely related, and observed
by an heuristic argument using zeta functions that there are non-rigorous
``implications'' between the conjectures. Key ingredients in this
study were the \emph{discrepancy} invariant of singularities and the
\emph{age} invariant, as is often the case in studies on the McKay
correspondence and quotient singularities. Recently, the first-named
author \cite{darda2021rational2} and Ellenberg--Satriano--Zureick-Brown
\cite{ellenberg2021heights} started studying distributions of rational
points on algebraic stacks as a step toward unifying the Batyrev--Manin
conjecture and Malle's conjecture. For this purpose, they needed
to introduce new height functions on stacks, based on different ideas.
The first-named author introduced the notion of \emph{quasi-toric
heights} on weighted projective stacks and derived a precise asymptotic
formula for rational points of these stacks. On the other hand, Ellenberg--Satriano--Zureick-Brown
introduced the \emph{height function associated to a vector bundle}
$\cV$, motivated by the work of Wood--Yasuda \cite{wood2015massformulas}.
Moreover, under certain assumptions, they define the Fujita invariant
$a(\cV)$ of a vector bundle and conjecture that for any $\epsilon>0$,
there exist constants $C_{1},C_{2}>0$ satisfying 
\[
C_{1}B^{a(\cV)}\le N(B)\le C_{2}B^{a(\cV)+\epsilon},
\]
where $N(B)$ is the number similarly defined as before but with the
height function associated to $\cV$. They call this a weak form of
the stacky Batyrev--Manin--Malle Conjecture. Note that their definition
of $a(\cV)$ is not given in terms of the pseudo-effective cone like
in the original Batyrev--Manin conjecture. 

The purpose of the present article is to introduce yet another height
function for DM (Deligne--Mumford) stacks and formulate a few conjectures
on the asymptotic behavior of the number of rational points with bounded
height. In particular, we suggest two compatible ways of determining
the exponent of $\log B$ in relevant asymptotic formulas, while we
``normalize'' data necessary to define a height function in such
a way that the exponent of $B$ is expected to be 1. On the other
hand, we do not try to find a formula for the leading constant $C$,
except that we speculate upon conditions for such a formula to exist
(Remark \ref{rem:speculation-leading-C}). Our key ingredient is the
\emph{stack of twisted 0-jets}, denoted by $\cJ_{0}\cX$, of the given
stack $\cX$, which appears in motivic integration over DM stacks
\cite{yasuda2004twisted,yasuda2006motivic} and is also known as the
\emph{cyclotomic inertia stack }in the Gromov--Witten theory for
DM stacks \cite{abramovich2008gromovwitten}. The trivial connected
component of $\cJ_{0}\cX$ is called the \emph{non-twisted sector}
and the other connected components are called \emph{twisted sectors}.
The set of all sectors, denoted by $\pi_{0}(\cJ_{0}\cX)$, is equipped
with the age function, $\age\colon\pi_{0}(\cJ_{0}\cX)\to\QQ_{\ge0}$.
To get a meaningful height function especially when $\cX$ has a
non-trivial generic stabilizer, we need to choose another function
$c\colon\pi_{0}(\cJ_{0}\cX)\to\RR$, which we call a \emph{raising
function,} and define the function $\age_{c}:=\age+c$ on $\pi_{0}(\cJ_{0}\cX)$.
A line bundle $\cL$ on a DM stack $\cX$ given with a raising function
$c$ defines a height function $H_{\cL,c}$ on rational points of
$\cX$, which is unique modulo bounded functions. Specifying an adelic
metric on $\cL$ and a raising datum $c_{*}$ refining the raising
function $c$ completely determines the height function. Our height
function inherits some nice properties of the classical height function
on varieties, that is, the multiplicativity, the functoriality and
the Northcott property. This new height function generalizes ones
of the first-named author \cite{darda2021rational2} and Ellenberg--Satriano--Zureick-Brown
mentioned above (as far as DM stacks over a number field are concerned),
and also the ``$f$-discriminant'' of Ellenberg--Venkatesh \cite[p.\ 163]{ellenberg2005counting}.
Note that the paper \cite{ellenberg2021heights} treats also Artin
stacks over global fields. 

Our first conjecture, Conjecture \ref{conj:Fano-stack}, concerns
the case where our stack $\cX$ is \emph{Fano}, meaning that the anti-canonical
line bundle $\omega_{\cX}^{-1}$ corresponds to an ample $\QQ$-line
bundle on the coarse moduli space. We also assume that the raising
function is \emph{adequate, }which means that $\age_{c}(\cY)\ge1$
for every twisted sector $\cY$ and that if $\cX$ is of dimension
zero, then there exists a sector $\cY$ with $\age_{c}(\cY)=c(\cY)=1$.
The conjecture claims that in this situation, the asymptotic formula
for $H_{\omega_{\cX}^{-1},c}$ be of the form
\[
CB(\log B)^{\rho(\cX)+j_{c}(\cX)-1},
\]
where $\rho(\cX)$ is the Picard number and $j_{c}(\cX)$ is the number
of \emph{$c$-junior sectors}, that is, twisted sectors with $\age_{c}=1$,
provided that a suitable ``accumulation thin subset'' is removed.
As evidence of this conjecture, we show that the conjecture is compatible
with taking products of Fano stacks and with the Manin conjecture
for Fano varieties with canonical singularities.

The second conjecture, Conjecture \ref{conj:general}, treats a more
general situation and is formulated in terms of a variant of the pseudo-effective
cone, which we call the \emph{orbifold pseudo-effective cone}, like
the original Batyrev--Manin conjecture for varieties. We define $a$-
and $b$-invariants for the pair $(\cL,c)$ of a line bundle and a
raising function $c$ and denote them by $a(\cL,c)$ and $b(\cL,c)$.
The conjecture claims that if the pair $(\cL,c)$ is big (meaning
that the pair gives a point lying in the interior of the orbifold
pseudo-effective cone), if we assume the adequacy condition (Definition
\ref{def:adequacy-2}), which in particular implies $a(\cL,c)=1$,
and if we remove a suitable thin subset, then the asymptotic formula
for $H_{\cL,c}$ is of the form
\[
CB(\log B)^{b(\cL,c)-1}.
\]
We show that the second conjecture implies the first one; the proof
is non-trivial unlike in the case of varieties. Of course, our conjectures
are formulated in such a way that they specialize to the Manin and
Batyrev--Manin conjectures for varieties. On the other hand, applying
these conjectures to the classifying stack $\B G$ of a finite group
$G$, we obtain a version of Malle's conjecture with respect to
a generalized discriminant, which was considered in \cite{ellenberg2005counting},
except that the removed ``accumulation subsets'' are different.
We follow Peyre's suggestion \cite[p.\ 345]{peyre2003pointsde} to
remove a thin subset from the set of rational points. Our second conjecture
also incorporates the notion of\emph{ breaking thin morphisms,} which
have been studied in the case of varieties for example in \cite{hassett2015balanced,lehmann2018balanced,lehmann2019geometric2,lehmann2019geometric}.
That enables us to interpret Klüners' counterexample \cite{kluners2005acounter}
to Malle's conjecture in terms of breaking thin maps of stacks.
A variation of this notion, weakly breaking thin morphism, is closely
related to Wood's \emph{fairness} \cite{wood2010onthe}. The slightly
more general case of classifying stacks $\B G$ with $G$ an étale finite
group scheme is studied in more detail in another paper \cite{darda2022torsors}
of the authors. It is proved there that our conjectures hold true
for $\B G$ with $G$ commutative.

In this paper, for the sake of simplicity, we focus on the case where
the base field is a number field. However, almost every argument and
statement in this paper can be easily translated to the case of an
arbitrary global field, provided that we restrict ourselves to \emph{tame}
DM stacks, that is, DM stacks having stabilizers of orders coprime
to the characteristic of the base field. 

The paper is organized as follows. In Section \ref{sec:Twisted-sectors-and},
we see the basics of twisted sectors and ages and define the residue
map, which plays an important role in the definition of our height
function. In Section \ref{sec:Line-bundles}, we recall the correspondence
between $\QQ$-line bundles of a stack and ones of the coarse moduli
space. We also recall the height function associated to a line bundle
with an adelic metric in a generalized context of stacks. In Section
\ref{sec:Raising-data}, we introduce the notion of raising data and
define the height function associated to the pair of a line bundle
and a raising datum. In Section \ref{sec:Fano-stacks}, we formulate
our first conjecture concerning Fano stacks. In Sections \ref{sec:Compatibility-with-products}
and \ref{sec:compare-coarse}, we show that this conjecture is compatible
with taking products of Fano stacks and with the Manin conjecture
for Fano varieties with canonical singularities. In Section \ref{sec:The-orbifold-pseudo-effective},
we introduce the notion of the orbifold Néron--Severi space and the
one of the orbifold pseudo-effective cone. In Section \ref{sec:a-and-b},
we define the $a$- and $b$-invariants of the pair of a line bundle
and a raising function and formulate a few versions of the Batyrev--Manin
conjecture for DM stacks. We also see how one of these conjecture
specializes to the first conjecture about Fano stacks as well as to
Malle's conjecture. We then reinterpret Klüners' counterexample
to Malle's conjecture in our language. Lastly, we introduce the
notion of $c$-comprehensiveness for finite groups, which gives a
sufficient condition that for a finite group $G$, the set of connected
$G$-torsors is disjoint from any weakly breaking thin subset. 

\subsection{Notation\label{subsec:Notation}}

Throughout the paper, we denote by $F$ a number field. A place of
$F$ is usually denoted by $v$. We denote the set of places by $M_{F}$.
The local field of $F$ at a place $v$ is denoted by $F_{v}$. If
$v$ is a finite place, the residue field of $F_{v}$ is denoted by
$\kappa_{v}$. We denote its cardinality by $q_{v}$. If $S$ is a
finite set of places including all infinite places, we denote the
ring of $S$-integers by $\cO_{S}$. We denote the absolute Galois
group of $F$ by $\Gamma_{F}$.

We denote by $\cX$ a DM stack over $F$ satisfying certain condition.
We denote the dimension of $\cX$ by $d$. We denote by $\cX(F)$
the groupoid of $F$-points, and by $\cX\langle F\rangle$ the set
of isomorphism classes of $F$-points.

For a scheme $T$, notation such as $\cX_{T}$ and $G_{T}$ mean that
these objects are \emph{defined over} $T$. They typically arise as
either the base change to $T$ from another scheme or the extension
to $T$ from an open subscheme of $T$. When $T=\Spec R$, we also
use notation like $\cX_{R}$ and $G_{R}$.

\subsection{Acknowledgments}

The authors would like to thank Aaron Landesman, Daniel Loughran, Julian Lyczak, Emmanuel Peyre and an anonymous referee for useful comments.
This work was supported by JSPS KAKENHI Grant Number JP18H01112. The
first named author was supported by JSPS Postdoctoral Fellowship for
Research in Japan. 

\section{Twisted sectors, residue maps and ages\label{sec:Twisted-sectors-and}}

\subsection{Twisted sectors}

Let us fix a number field $F$.
\begin{defn}
A \emph{nice stack over $F$ }means a DM stack $\cX$ over $F$ satisfying
the following conditions:
\begin{enumerate}
\item $\cX$ is separated, geometrically irreducible, and smooth over $F$,
\item a coarse moduli space of $\cX$ is a projective $F$-scheme, and 
\item $\cX$ is not isomorphic to $\Spec F$.
\end{enumerate}
\end{defn}

\begin{rem}
We exclude $\Spec F$ from nice stacks. One reason for this is that
$\Spec F$ has only one $F$-point and counting it is not very meaningful.
Another reason is that the orbifold Néron--Severi space of $\Spec F$
(see Definition \ref{def:orb-NS}) is the trivial/zero space and we
cannot define meaningful $a$- and $b$-invariants. Note that we do
not exclude the possibility that a nice stack may have $\Spec F$
as its coarse moduli space. In fact, such stacks are of great interest
in relation to Malle's conjecture.
\end{rem}

Throughout the rest of the paper, we denote by $\cX$ a nice stack,
unless otherwise specified. We choose a finite set of places, $S\subset M_{F}$,
including all infinite places such that there exists an irreducible,
smooth, proper, and tame model $\cX_{\cO_{S}}$ of $\cX$ over $\cO_{S}$.
Here ``tame'' means that for every point $x\in\cX_{\cO_{S}}(K)$
with $K$ a field, the automorphism group $\Aut(x)$ has order coprime
to the characteristic of $K$. We see that such a finite set $S$
exists.
\begin{defn}[Stacks of twisted $0$-jets]
Let $\mu_{l,F}$ and $\mu_{l,\cO_{S}}$ be the group schemes of $l$-th
roots of unity over $F$ and $\cO_{S}$, respectively. Let $\B\mu_{l,F}:=[\Spec F/\mu_{l,F}]$
and $\B\mu_{l,\cO_{S}}:=[\Spec\cO_{S}/\mu_{l,\cO_{S}}]$ be their
classifying stacks, where the notation $[-/-]$ means a quotient stack.
We define the \emph{stack of twisted 0-jets} of $\cX$ by 
\[
\cJ_{0}\cX:=\coprod_{l>0}\ulHom_{F}^{\rep}(\B\mu_{l,F},\cX).
\]
Here $\ulHom_{F}^{\rep}(-,-)$ means the Hom stack of representable
morphism (see \cite{olsson2006homstacks}). We also define 
\[
\cJ_{0}\cX_{\cO_{S}}:=\coprod_{l>0}\ulHom_{\cO_{S}}^{\rep}(\B\mu_{l,\cO_{S}},\cX_{\cO_{S}}).
\]
\end{defn}

Namely, for each $F$-scheme $T$, the fiber $(\cJ_{0}\cX)(T)$ over
$T$ is the groupoid of representable $T$-morphisms $\B\mu_{l,T}\to\cX_{T}$,
with the subscript $T$ meaning the base change to $T$. Similarly,
for $\cJ_{0}\cX_{\cO_{S}}$. From \cite{olsson2006homstacks}, $\cJ_{0}\cX$
and $\cJ_{0}\cX_{\cO_{S}}$ are DM stacks locally of finite type over
$F$ and $\cO_{S}$, respectively. From \cite[Lem.\ 6.5]{yasuda2020motivic},
they are, in fact, of finite type. 
\begin{rem}
\label{rem:tw-form-IX}If $\I\cX$ denotes the inertia stack of $\cX$,
then we have a (non-canonical) isomorphism $(\cJ_{0}\cX)\otimes_{F}\overline{F}\cong(\I\cX)\otimes_{F}\overline{F}$
(see \cite[Prop.\ 22]{yasuda2006motivic}). 
\end{rem}

\begin{rem}
Stacks of twisted 0-jets were used in \cite{yasuda2004twisted,yasuda2006motivic}
to develop motivic integration over DM stacks. The same notion appears
also in the orbifold Gromov-Witten theory \cite{abramovich2002algebraic,abramovich2008gromovwitten}.
In \cite{abramovich2008gromovwitten}, this is called the \emph{cyclotomic
inertia stack. }Motivic integration (or a variant of it, $p$-adic
integration) and Gromov-Witten theory for stacks may be regarded,
respectively, as local and geometric analogues of counting rational
points of stacks. This explains why looking at the stack of twisted
0-jets is natural in the context of the present paper.
\end{rem}

\begin{defn}[Sectors]
We call the trivial connected component $\ulHom_{F}(\B\mu_{1,F},\cX)$
of $\cJ_{0}\cX$, which is canonically isomorphic to $\cX$, the \emph{non-twisted
sector} and denote it again by $\cX$, abusing the notation. We call
the other connected components \emph{twisted sectors. }We call a connected
component of $\cJ_{0}\cX$ a \emph{sector, }whether twisted or non-twisted.
We denote by $\pi_{0}(\cJ_{0}\cX)$ (resp.~$\pi_{0}^{*}(\cJ_{0}\cX)$)
the set of sectors (resp.~twisted sectors). 
\end{defn}

\begin{lem}
The stack $\cJ_{0}\cX$ is smooth and proper over $F$. Moreover,
if $\cJ_{0}\cX_{\cO_{S}}$ is flat over $\cO_{S}$, then it is also
smooth and proper over $\cO_{S}$. 
\end{lem}

\begin{proof}
From the generic flatness, if we enlarge $S$, we can make $\cJ_{0}\cX_{\cO_{S}}$
flat over $\cO_{S}$. Thus, the first assertion follows from the second.
To show the second assertion, suppose that $\cJ_{0}\cX_{\cO_{S}}$
is flat over $\cO_{S}$. It suffices to show that for each geometric
point $\Spec K\to\Spec\cO_{S}$, the fiber $\cJ_{0}\cX_{K}$ is smooth
and proper over $K$. The $K$-stack $\cJ_{0}\cX_{K}$ is isomorphic
to the inertia stack $\I\cX_{K}$ of $\cX_{K}$, which is well-known
to be smooth and proper. The properness of $\I\cX_{K}$ follows from
the fact that the natural morphism $\I\cX_{K}\to\cX_{K}$ is identified
with the (say first) projection
\[
\cX_{K}\times_{\Delta,\cX_{K}\times_{K}\cX_{K},\Delta}\cX_{K}\to\cX_{K},
\]
which is proper. The smoothness follows from the local description
of the inertia stack \cite[tag 0374]{stacksprojectauthors2022stacksproject}.
\end{proof}
\begin{cor}
\label{cor:bij-Pi-0}If $\cJ_{0}\cX_{\cO_{S}}$ is flat over $\cO_{S}$,
then natural map $\pi_{0}(\cJ_{0}\cX)\to\pi_{0}(\cJ_{0}\cX_{\cO_{S}})$
is bijective. Here $\pi_{0}(-)$ denotes the set of connected components.
\end{cor}

\begin{proof}
Let $K/F$ be a finite Galois extension such that every connected
component of $\cJ_{0}\cX_{K}$ is geometrically connected. Let $G$
be its Galois group. We have a natural $G$-action on $\pi_{0}(\cJ_{0}\cX_{K})$
and can identify $\pi_{0}(\cJ_{0}\cX)$ with the set of $G$-orbits
in $\pi_{0}(\cJ_{0}\cX_{K})$. 

Let $\cO_{T}$ be the integral closure of $\cO_{S}$ in $K$. Since
$\cJ_{0}\cX_{\cO_{T}}$ is smooth over $\cO_{T}$, its irreducible
components are disjoint to each other. This shows that connected components
$\cY$ of $\cJ_{0}\cX_{K}$ has closures $\overline{\cY}$ in $\cJ_{0}\cX_{\cO_{T}}$
which are disjoint with one another. It follows that $\pi_{0}(\cJ_{0}\cX_{K})$
and $\pi_{0}(\cJ_{0}\cX_{\cO_{T}})$ are canonically isomorphic $G$-sets.
Accordingly, their orbit sets are identified, which implies the corollary.
\end{proof}
\begin{rem}
For a more geometric (rather than algebraic) treatise on twisted sectors,
we refer the reader to \cite{adem2007orbifolds}.
\end{rem}

It is sometimes more useful to express $\cJ_{0}\cX$ by using the
pro-finite group scheme $\widehat{\mu}_{F}:=\projlim\mu_{l,F}$ instead
of finite ones $\mu_{l,F}$, $l\in\ZZ_{>0}$. Note that since transition
morphisms of the projective system $\mu_{l,F}$, $l\in\ZZ_{>0}$ are
all affine morphisms, the projective limit $\projlim\mu_{l,F}$ exists
as a scheme (see \cite[Prop.\ 8.2.3]{grothendieck1966elements} or
\cite[tag 01YX]{stacksprojectauthors2022stacksproject}) and has a
natural structure of a group scheme. In other words, the limit $\projlim\mu_{l,F}$
as a functor from $F$-schemes to groups is (represented by) a group
scheme. Similarly for $\widehat{\mu}_{\cO_{S}}:=\projlim\mu_{l,\cO_{S}}$. 

To obtain such an description of $\cJ_{0}\cX$, we need to slightly
generalize the Hom stack as follows. Let $R$ be a ring, let $\Sch_{R}$
be the category of $R$-schemes and let $\cU$ and $\cV$ be two categories
fibered in groupoids over $\Sch_{R}$. We define $\ulHom_{F}(\cU,\cV)$
to be the category fibered in groupoids over $\Sch_{R}$ as follows:
for each $R$-scheme $T$, the fiber $\ulHom_{R}(\cU,\cV)(T)$ is
\[
\Hom_{T}(\cU\times_{\Spec R}T,\cV\times_{\Spec R}T),
\]
the groupoid of $T$-morphisms $\cU\times_{\Spec R}T\to\cV\times_{\Spec R}T$.
\begin{lem}
\label{lem:KXG}Let $G$ be a group scheme over $R$. Let $\cK_{\cX,G}$
be the fibered category over $\Sch_{R}$ such that for each $F$-scheme
$T$, the fiber $\cK_{\cX,G}(T)$ is the groupoid of pairs $(x,\phi)$,
where $x$ is an object of $\cX(T)$ and $\phi$ is a homomorphism
$G_{T}^{\op}\to\ulAut_{T}(x)$ of group schemes over $T$.\footnote{We follow the convention that a group acts on a torsor \emph{from
left}. Then, the automorphism group of a $G$-torsor is identified
with a subgroup of the opposite group $G^{\op}$. } Then 
\[
\ulHom_{F}(\B G,\cX)\cong\cK_{\cX,G}.
\]
\end{lem}

\begin{proof}
Let $T$ be an $R$-scheme. We construct functors between the groupoids
$\Hom_{T}(\B G_{T},\cX_{T})$ and $\cK_{\cX,G}(T)$ which are quasi-inverses
of each other. 

Construction of $\Hom_{T}(\B G_{T},\cX_{T})\to\cK_{\cX}(T)$: Let
$f\colon\B G_{T}\to\cX_{T}$ be a $T$-morphism. Suppose that the
trivial torsor $\tau\colon G\times T\to T$, which is an object of
$(\B G_{T})(T)$, maps to $\alpha\in\cX_{T}(T)=\cX(T)$ by $f$. The
morphism $f$ induces a morphism of group schemes over $T$,
\[
G_{T}^{\op}\xrightarrow{\sim}\ulAut_{T}(\tau)\to\ulAut_{T}(\alpha).
\]

Construction of $\cK_{\cX}(T)\to\Hom_{T}(\B G_{T},\cX_{T})$: Let
$\alpha$ be an object of $\cX(T)$ and $\phi\colon G_{T}^{\op}\to\ulAut_{T}(\alpha)$
a homomorphism. Let $[G_{T}^{\op}\rightrightarrows T]'$ (resp.~$[G_{T}^{\op}\rightrightarrows T]$)be
the prestack (resp.~stack) associated to the trivial groupoid scheme
$G_{T}^{\op}\rightrightarrows T$. We define a morphism $[G_{T}^{\op}\rightrightarrows T]'\to\cX_{T}$
so that a morphism $S\to T$, regarded as an object in $[T\rightrightarrows G_{T}^{\op}]'(S)$,
maps to $\alpha_{S}$ and a morphism $S\to G_{T}^{\op}$, regarded
as a morphism in $[T\rightrightarrows G_{T}^{\op}]'(S)$, maps to
the automorphism of $\alpha_{S}$ corresponding the composition $S\to G_{T}^{\op}\xrightarrow{\phi}\ulAut_{T}(\alpha)$.
This induces a morphism of stacks $\B G_{T}=[T\rightrightarrows G_{T}^{\op}]\to\cX_{T}$. 

It is easy to see that these functors between $\Hom_{T}(\B G_{T},\cX_{T})$
and $\cK_{\cX,G}(T)$ are quasi-inverses of each other, and induce
the claimed isomorphism of the lemma.
\end{proof}
\begin{prop}
\label{prop:J0-mu-hat}We have $\cJ_{0}\cX\cong\ulHom_{F}(\B\widehat{\mu}_{F},\cX)$
and $\cJ_{0}\cX_{\cO_{S}}\cong\ulHom_{\cO_{S}}(\B\widehat{\mu}_{\cO_{S}},\cX_{\cO_{S}})$. 
\end{prop}

\begin{proof}
For an $F$-scheme $T$, a representable morphism $\B\mu_{l,T}\to\cX$
induces the (not representable) morphism 
\[
\B\widehat{\mu}_{T}\to\B\mu_{l,T}\to\cX.
\]
This induces a morphism $\cJ_{0}\cX\to\ulHom_{F}(\B\widehat{\mu}_{F},\cX)$.
Conversely, for a morphism $\B\widehat{\mu}_{T}\to\cX$, we have the
associated morphism
\[
\widehat{\mu}_{T}\to\ulAut_{T}(x)\quad(x\in\cX(F))
\]
from Lemma \ref{lem:KXG}. Each point $t\in T$ has an open neighborhood
$U\subset T$ such that the induced morphism $\widehat{\mu}_{U}\to\ulAut_{U}(x_{U})$
canonically factors into the composition
\[
\widehat{\mu}_{U}\to\mu_{l,U}\rightarrowtail\ulAut_{U}(x_{U})
\]
of the canonical surjection $\widehat{\mu}_{U}\to\mu_{l,U}$ and an
injection $\mu_{l,U}\rightarrowtail\ulAut_{U}(x_{U})$. Thus, we get
an object of $\cJ_{0}\cX$ over $U$. They glue together to give an
object over $T$. Thus, we get a morphism $\ulHom_{F}(\B\widehat{\mu}_{F},\cX)\to\cJ_{0}\cX$.
We see that these morphisms between $\cJ_{0}\cX$ and $\ulHom_{F}(\B\widehat{\mu}_{F},\cX)$
are quasi-inverses to each other, which shows the first isomorphism
of the proposition. The second isomorphism is proved in the same way.
\end{proof}
\begin{rem}
The stack $\B\widehat{\mu}_{F}$ is neither a DM or Artin stack.
\end{rem}

\begin{rem}
If $e$ is a sufficiently factorial positive integer, then we can
describe $\cJ_{0}\cX=\ulHom_{F}(\B\widehat{\mu}_{F},\cX)$ also as
$\ulHom_{F}(\B\mu_{e,F},\cX)$. 
\end{rem}

\begin{cor}
\label{cor:map-J0}For a (not necessarily representable) morphism
$f\colon\cY\to\cX$ of nice stacks, we have a natural morphism $\cJ_{0}\cY\to\cJ_{0}\cX$
and a natural map $\pi_{0}(\cJ_{0}\cY)\to\pi_{0}(\cJ_{0}\cY)$. Moreover,
if $f_{\cO_{S}}\colon\cY_{\cO_{S}}\to\cX_{\cO_{S}}$ is a model of
$f$, then we have also a natural morphism $\cJ_{0}\cY_{\cO_{S}}\to\cJ_{0}\cX_{\cO_{S}}$.
\end{cor}

\begin{proof}
The natural morphism $\cJ_{0}\cY\to\cJ_{0}\cX$ is defined by sending
a morphism $h\colon\B\widehat{\mu}_{T}\to\cY$ to the composition
$f\circ h\colon\B\widehat{\mu}_{T}\to\cY\to\cX$. This morphism induces
a map $\pi_{0}(\cJ_{0}\cY)\to\pi_{0}(\cJ_{0}\cY)$. The morphism $\cJ_{0}\cY_{\cO_{S}}\to\cJ_{0}\cX_{\cO_{S}}$
is similarly defined.
\end{proof}
\begin{example}
\label{exa:BG tw sectors}Let $G$ be a finite étale group scheme
over $F$ and let $\cX=\B G$. Let $\overline{F}$ be an algebraic
closure of $F$. Then, we have identifications
\begin{align*}
|\cJ_{0}\cX_{\overline{F}}| & =\pi_{0}(\cJ_{0}\cX_{\overline{F}})\\
 & =\Hom(\widehat{\mu}^{\op}(\overline{F}),G^{\op}(\overline{F}))/G^{\op}(\overline{F})\\
 & =\Hom(\widehat{\mu}(\overline{F}),G(\overline{F}))/G(\overline{F}),
\end{align*}
where $|-|$ denotes the point set of a stack, the $G(\overline{F})$-action
on $\Hom(\widehat{\mu}(\overline{F}),G(\overline{F}))$ is induced
by the conjugate action of $G(\overline{F})$ on itself, and similarly
for the $G^{\op}(\overline{F})$-action on $\Hom(\widehat{\mu}^{\op}(\overline{F}),G^{\op}(\overline{F}))$.
If $e$ denotes the exponent of $G(\overline{F})$, then we also have
\[
\Hom(\widehat{\mu}(\overline{F}),G(\overline{F}))/G(\overline{F})=\Hom(\mu_{e}(\overline{F}),G(\overline{F}))/G(\overline{F}).
\]
The last set is denoted by $G_{*}$ in \cite{darda2022torsors}. These
identifications are equivariant for natural $\Gamma_{F}$-actions.
We can then identify $\pi_{0}(\cJ_{0}\cX)$ with the set of $\Gamma_{F}$-orbits
in one of these sets. If $G$ is a constant group, then $\pi_{0}(\cJ_{0}\cX)$
is identified with the set of $F$-conjugacy classes . If $G=\mu_{l}$,
then $\pi_{0}(\cJ_{0}\cX)$ is identified with $\ZZ/l\ZZ$ (regardless
of the field $F$). For these facts, see \cite{darda2022torsors}
(cf.~\cite[Th.\ 5.4]{wood2015massformulas}, \cite[Prop.\ 4.5]{yasuda2014densities},
\cite[Prop.\ 8.5]{yasuda2015maninsconjecture}).
\end{example}

\subsection{Residue maps}

Let $v$ be a finite place of $F$ with $v\notin S$
and let $L$ be the maximal unramified extension of $F_{v}$. For
a positive integer $l$ coprime to the residue characteristic of $v$,
let $L_{l}$ denote the unique degree-$l$ extension of $L$, which is given by adjoining the $l$-th roots of a uniformizer, and let $\cO_{L_l}$ denote its integer ring. There exists a natural action of $\mu_{l,\cO_L}$ on $\Spec \cO_{L_l}$, which induces the quotient stack $[\Spec\cO_{L_{l}}/\mu_{l,\cO_L}]$. 
Namely, this is the $l$-th root stack associated to $\Spec\cO_{L_{l}} $ and its closed point regarded as a Cartier divisor.
We can regard $\Spec L$ as an open substack
of $[\Spec\cO_{L_{l}}/\mu_{l,\cO_L}]$ by the canonical open immersion.

\begin{lem}
\label{lem:ext_tw_arc}
With the above notation, 
for an $L$-point $x\colon\Spec L\to\cX$, there exists a unique positive integer $l$ and a unique representable
morphism of the form,
\[
\widetilde{x}\colon[\Spec\cO_{L_{l}}/\mu_{l,\cO_L}]\to\cX_{\cO_{S}}
\]
which extends $x$. 
\end{lem}

\begin{proof}
Let $X_{\cO_{S}}$ be the coarse moduli space of $\cX_{\cO_{S}}$.
The induced $L$-point, $\overline{x}\colon\Spec L\to X$, uniquely
extends to an integral point $\Spec\cO_{L}\to X_{\cO_{L}}$, which
is a closed immersion. Replacing $X_{\cO_{L}}$ with an open neighborhood
of this integral point, we may reduce the problem to the case where
$X_{\cO_{L}}$ is affine, say $X_{\cO_{L}}=\Spec R$. Let $\widehat{X}_{\cO_{L}}$
be the formal completion $\Spec\widehat{R}$ (not as a formal scheme
but as a scheme) of $X_{\cO_{L}}$ along this integral point and let
$\widehat{\cX}_{\cO_{S}}:=\widehat{X}_{\cO_{S}}\times_{X_{\cO_{S}}}\cX$.
We have a natural morphism $\Spec L\to\widehat{\cX}_{\cO_S}$. Let $\cD$
be the relative normalization of $\widehat{\cX}_{\cO_S}$ in $\Spec L$.
Note that the relative normalization for schemes is compatible with
étale base change \cite[tag 0ABP]{stacksprojectauthors2022stacksproject},
and hence generalizes to DM stacks. 
Then, $\cD$ is a regular, irreducible
and one-dimensional DM stack which contains $\Spec L$ as an open
dense substack and has $\Spec\cO_{L}$ as the coarse moduli space.
We claim that $\cD$ is of the form $[\Spec\cO_{L_{l}}/\mu_{l,\cO_L}]$
for some positive integer $l$ coprime to the residue characteristic
of $v$. Indeed, there exists an étale atlas of $\cD$ of the form
\[
\Spec\cO_{L_{l}}\to\cD
\]
for a tame finite extension $L_{l}/L$. Consider the associated groupoid
scheme 
\[
\Spec\cO_{L_{l}}\times_{\cD}\Spec\cO_{L_{l}}\rightrightarrows\Spec\cO_{L_{l}}.
\]
Since $L_{l}$ has an algebraically closed residue field, the scheme
$\Spec\cO_{L_{l}}\times_{\cD}\Spec\cO_{L_{l}}$ is isomorphic to disjoint
union of copies of $\Spec\cO_{L_{l}}$. We conclude that the above
groupoid scheme is isomorphic to the groupoid scheme 
\[
\Spec\cO_{L_{l}}\times G\rightrightarrows\Spec\cO_{L_{l}}
\]
associated to an action of a constant group $G$ on $\Spec\cO_{L_{l}}$.
Thus, $\cD\cong[\Spec\cO_{L_{l}}/G]$. Since $\cD$ contains $\Spec L$
as an open dense substack and has $\Spec\cO_{L}$ as the coarse moduli
space, the $G$-action on $\Spec L_{l}$ is free, and $G$ is the
Galois group of $L_{l}/L$ and isomorphic to $\mu_{l,\cO_L}$ for some
$l$. We have proved the claim.

The induced morphism $\cD \to \widehat{\cX}_{\cO_S}$ is representable, since it becomes a relative normalization of schemes after the base change to an étale atlas of the target.  It follows that $\cD \to {\cX}_{\cO_S}$ is also representable. In summary, we have proved the existence of a positive integer and a representable morphism as claimed in the lemma. To show the uniqueness, consider an arbitrary morphism $\cD'=[\Spec\cO_{L_{l'}}/\mu_{l',\cO_L}]\to \cX_{\cO_S}$  as in the lemma. This induces a representable morphism $\cD' \to \widehat{\cX}_{\cO_S}$, which is again a relative normalization. The uniqueness of relative normalization shows that there exists an isomorphism $\cD\cong \cD'$ compatible with the morphisms to $\widehat{\cX}_{\cO_S}$ as well as the ones to ${\cX}_{\cO_S}$, which, in turn, shows $l=l'$. This completes the proof of the lemma.
\end{proof}
We keep the notation of Lemma \ref{lem:ext_tw_arc} and let $k$ be
the residue field of $L$, which is an algebraic closure of $\kappa_{v}$.
Suppose that $\cJ_{0}\cX_{\cO_{S}}$ is flat at the point of $\Spec\cO_{S}$
corresponding to $v$. The lemma shows that each $L$-point $x$ of
$\cX$ induces the representable morphism 
\[
\B\mu_{l,k}\hookrightarrow[\Spec\cO_{L_l}/\mu_{l,\cO_L}]\xrightarrow{\widetilde{x}}\cX_{\cO_{S}},
\]
which is a $k$-point of $\cJ_{0}\cX_{\cO_{S}}$. Thus, for each $v\in M_{F}\setminus S$,
we get the composite functor
\[
\cX(F_{v})\to\cX(L)\to(\cJ_{0}\cX_{\cO_{S}})(k).
\]
In particular, each $F_{v}$-point of $\cX$ determines a connected
component of $\cJ_{0}\cX_{\cO_{S}}$. From Corollary \ref{cor:bij-Pi-0},
we have the corresponding sector of $\cX$; namely, for each $v\in M_{F}\setminus S$,
we get a map
\[
\psi_{v}\colon\cX\langle F_{v}\rangle\to\pi_{0}(\cJ_{0}\cX),
\]
where $\cX\langle F_{v}\rangle$ denotes the set of isomorphism classes
in the groupoid $\cX(F_{v})$.
\begin{defn}
\label{def:residue-map}We call the map $\psi_{v}$ the \emph{residue
map }of $\cX$ at $v$\emph{ }and $\psi_{v}(x)$ the \emph{residue
}of $x$. 
\end{defn}

From \cite[Section 2.1]{moret-bailly2001problemes} (see also \cite[Section 2.4]{vcesnavivcius2015topology}),
$\cX\langle F_{v}\rangle$ has a natural topology. 
\begin{prop}
The residue map $\psi_{v}$ is locally constant. Equivalently, it
is continuous for the discrete topology on $\pi_{0}(\cJ_{0}\cX)$. 
\end{prop}

\begin{proof}
We keep the notation of the proof of Lemma \ref{lem:ext_tw_arc}.
Let $\cD^{l}:=[\Spec\cO_{L_{l}}/\mu_{l,\cO_L}]$. We consider the stack
\[
\cY_{l}:=\ulHom_{\cO_{L}}^{\rep}(\cD^{l},\cX_{\cO_{L}}).
\]
Let $\cY_{l}=\coprod_{i=1}^{n}\cY_{l,i}$ be the decomposition into
connected components. Then, the natural map
\[
\left(\Hom(\cD^{l},\cX_{\cO_{L}})/\cong\right)=\cY_{l}\langle\cO_{L}\rangle\to\pi_{0}(\cJ_{0}\cX)
\]
is constant on each  open and closed subset $\cY_{l,i}\langle\cO_{L}\rangle\subset\cY\langle\cO_{L}\rangle$.
Since 
\begin{align*}
\cY_{l}\otimes_{\cO_{L}}L & =\ulHom_{L}^{\rep}(\cD^{l}\otimes_{\cO_{L}}L,\cX_{\cO_{L}}\otimes_{\cO_{L}}L)\\
 & =\ulHom_{L}^{\rep}(\Spec L,\cX)\\
 & =\cX_{L},
\end{align*}
we have $\cY_{l}\langle L\rangle=\cX\langle L\rangle$. From \cite[Prop.\ 2.9 (e)]{vcesnavivcius2015topology},
the maps
\[
\cY_{l}\langle\cO_{L}\rangle\to\cY_{l}\langle L\rangle=\cX\langle L\rangle
\]
are open, and hence the images of $\cY_{l,i}\langle\cO_{L}\rangle$
in $\cX\langle L\rangle$ are open. Since $\cX\langle L\rangle$ is
covered by the images of these maps, we see that $\cX\langle L\rangle\to\pi_{0}\langle\cJ_{0}\cX\rangle$
is locally constant and continuous. From \cite[Cor.\ 2.7]{vcesnavivcius2015topology},
$\cX\langle F_{v}\rangle\to\cX\langle L\rangle$ is continuous. We
conclude that the composition $\psi_{v}\colon\cX\langle F_{v}\rangle\to\cX\langle L\rangle\to\pi_{0}(\cJ_{0}\cX)$
is also continuous.
\end{proof}
\begin{rem}
Representable morphisms $[\Spec\cO_{L,l}/\mu_{l,\cO_L}]\to\cX_{\cO_{S}}$
as in Lemma \ref{lem:ext_tw_arc} are a version of twisted arcs \cite{yasuda2004twisted,yasuda2006motivic,yasuda2017towardmotivic,yasuda2020motivic}.
Associating twisted 0-jets $\B\mu_{l,k}\to\cX_{\cO_{S}}$ to them
is a special case of the truncation map. Morphisms $[\Spec\cO_{L,l}/\mu_{l,\cO_L}]\to\cX_{\cO_{S}}$
are also an analogue of twisted stable map \cite{abramovich2002compactifying}
and one of tuning stack \cite{ellenberg2021heights}.
\end{rem}

\begin{prop}
\label{prop:residue-functorial}Let $\cY$ and $\cX$ be nice stacks
having models $\cY_{\cO_{S}}$ and $\cX_{\cO_{S}}$over $\cO_{S}$,
respectively. Let $\cY_{\cO_{S}}\to\cX_{\cO_{S}}$ be a morphism over
$\cO_{S}$. Suppose that $\cJ_{0}\cY_{\cO_{S}}$ and $\cJ_{0}\cX_{\cO_{S}}$
are flat over $\cO_{S}$. Then, for every finite place $v\notin S$,
the following diagram is commutative:
\[
\xymatrix{\cY\langle F_{v}\rangle\ar[r]\ar[d]_{\psi_{v}} & \cX\langle F_{v}\rangle\ar[d]^{\psi_{v}}\\
\pi_{0}(\cJ_{0}\cY)\ar[r] & \pi_{0}(\cJ_{0}\cX)
}
\]
Here the horizontal morphisms are the ones induced by the morphism
$\cY_{\cO_{S}}\to\cX_{\cO_{S}}$.
\end{prop}

\begin{proof}
As before, we denote by $L$ the maximal unramified extension of $F_{v}$
and by $k$ its residue field. An $F_{v}$-point $y\colon\Spec F_{v}\to\cY$
induces an $L$-point $y_{L}\colon\Spec L\to\cY$. From Lemma \ref{lem:ext_tw_arc},
this $L$-point uniquely extends to a representable morphism $\widetilde{y_{L}}\colon[\Spec\cO_{L_{l}}/\mu_{l,\cO_L}]\to\cY_{\cO_{S}}$.
There exists a unique divisor $m$ of $l$ and a unique representable
morphism $\widetilde{x_{L}}\colon[\Spec\cO_{L_{m}}/\mu_{m,\cO_L}]\to\cX_{\cO_{S}}$
making the following diagram commutative: 
\[
\xymatrix{\B\mu_{l,k}\ar[r]\ar[d] & [\Spec\cO_{L_{l}}/\mu_{l,\cO_L}]\ar[r]\sp(0.7){\widetilde{y_{L}}}\ar[d] & \cY_{\cO_{S}}\ar[d]\\
\B\mu_{m,k}\ar[r] & [\Spec\cO_{L_{m}}/\mu_{m,\cO_L}]\ar[r]\sb(0.7){\widetilde{x_{L}}} & \cX_{\cO_{S}}
}
\]
Here arrows without a label are natural morphisms and all horizontal
arrows are representable morphisms. This is basically \cite[Prop.\ 23]{yasuda2006motivic}
with $n=0,\infty$, except that we work over a different base ring,
and the proof there works also in our situation. We also see that
$\widetilde{x_{L}}$ is the same as the morphism induced from the
point $x:=f(y)\in\cX(F_{v})$. Let 
\begin{gather*}
y_{k}\colon\Spec k\to\B\mu_{l,k}\to\cY_{\cO_{S}}\text{ and}\\
x_{k}=f(y_{k})\colon\Spec k\to\B\mu_{m,k}\to\cX_{\cO_{S}}
\end{gather*}
be the induced $k$-points. Injections $\mu_{l,k}\hookrightarrow\ulAut_{k}(y_{k})$
and $\mu_{m,k}\hookrightarrow\ulAut_{k}(x_{k})$ that correspond to
$\B\mu_{l,k}\to\cY_{\cO_{S}}$and $\B\mu_{m,k}\to\cX_{\cO_{S}}$ respectively
fit into the following commutative diagram:
\[
\xymatrix{\widehat{\mu}_{k}\ar@{->>}[r]\ar@{->>}[dr] & \mu_{l,k}\ar@{^{(}->}[r]\ar@{->>}[d] & \ulAut_{k}(y_{k})\ar[d]\\
 & \mu_{m,k}\ar@{^{(}->}[r] & \ulAut_{k}(x_{k})
}
\]
This shows that if $\beta\in(\cJ_{0}\cY_{\cO_{S}})(k)$ and $\alpha\in(\cJ_{0}\cX_{\cO_{S}})(k)$
denote $k$-points corresponding to $\B\mu_{l,k}\to\cY_{\cO_{S}}$
and $\B\mu_{m,k}\to\cX_{\cO_{S}}$ respectively, then $\beta$ maps
to $\alpha$ by the morphism $\cJ_{0}\cY_{\cO_{S}}\to\cJ_{0}\cX_{\cO_{S}}$.

The residue $\psi_{v}(y)$ of the $F_{v}$-point $y$ is the sector
containing $\beta$ and the residue $\psi_{v}(x)$ of the $F_{v}$-point
$x=f(y)$ is the sector containing $\alpha$. Thus, the above fact
that $\beta$ maps to $\alpha$ proves the proposition.
\end{proof}

\subsection{Ages}

A point $\widetilde{x}$ of $(\cJ_{0}\cX)(\overline{F})$ is represented
by the pair $(x,\iota)$ of a point $x\in\cX(\overline{F})$ and a
group monomorphism $\iota\colon\mu_{l}\hookrightarrow\Aut(x)$. Let
$\cV$ be a vector bundle of rank $r$ on $\cX$, we get a representation
of the group $\mu_{l}\subset\overline{F}$:
\[
\rho_{\widetilde{x}}\colon\mu_{l}\hookrightarrow\Aut(x)\to\GL(\cV_{x})\cong\GL_{r}(\overline{F}).
\]
Here $\cV_{x}$ is the fiber of $\cV$ over $x$, in particular, an
$\overline{F}$-vector space of dimension $r$. If $\tau$ denotes
the standard one-dimensional representation 
\[
\tau\colon\mu_{l}\hookrightarrow\overline{F}^{*}=\GL_{1}(\overline{F}),
\]
then we can write 
\[
\rho_{\widetilde{x}}\cong\bigoplus_{i=1}^{r}\tau^{a_{i}}\quad(a_{i}\in\{0,1,\dots,l-1\}).
\]

\begin{defn}
We define the \emph{age }of $\widetilde{x}$ with respect to $\cV$
to be
\[
\age(\widetilde{x};\cV):=\frac{1}{l}\sum_{i=1}^{d}a_{i}\in\QQ_{\ge0}.
\]
When $\cV$ is the tangent bundle $\T\cX$, we simply call it the
\emph{age }of $\widetilde{x}$ and denote it by $\age(\widetilde{x})$.
When we would like to specify the stack $\cX$ in question, we write
$\age_{\cX}(\widetilde{x})$. 
\end{defn}

\begin{lem}
\label{lem:age-sector}If we fix a vector bundle $\cV$, then the
age \emph{$\age(\widetilde{x};\cV)$} depends only on the sector to
which $\widetilde{x}$ belongs.
\end{lem}

\begin{proof}
We omit the proof, as this is well-known for the case $V=\T\cX$ (for
example, see \cite[p.\ 743]{yasuda2006motivic}) and there is no essential
difference in the general case.
\end{proof}
\begin{defn}
For a sector $\cY$, we define $\age(\cY;\cV)$ to be $\age(\widetilde{x};\cV)$
for any point $\widetilde{x}\in\cY(\overline{F})$. When $\cV=\T\cX$,
we denote it simply by $\age(\cY)$. 
\end{defn}

\section{Line bundles and stable heights\label{sec:Line-bundles}}

Let $\cX$ be a nice stack over $F$.
\begin{defn}
The \emph{Picard group} of $\cX$, denoted by $\Pic(\cX)$, is the
group of isomorphism classes of line bundles on $\cX$; the group
structure is given by the tensor product of line bundles. We define
$\Pic(\cX)_{\QQ}:=\Pic(\cX)\otimes_{\ZZ}\QQ$ and call its elements
\emph{$\QQ$-line bundles}. 
\end{defn}

A $\QQ$-line bundle is thus represented by a formal product $\bigotimes_{i=1}^{n}\cL_{i}^{s_{i}}$
of line bundles with $s_{i}\in\QQ$ or even by a formal power $\cL^{s}$,
$s\in\QQ$ of a single line bundle $\cL$. Let $\pi\colon\cX\to\overline{\cX}=X$
be the coarse moduli space morphism. We have the pullback map $\pi^{*}\colon\Pic(X)_{\QQ}\to\Pic(\cX)_{\QQ}.$
\begin{prop}
\label{prop:iso-Pic}The map $\pi^{*}\colon\Pic(X)_{\QQ}\to\Pic(\cX)_{\QQ}$
is an isomorphism.
\end{prop}

\begin{proof}
Let $r$ be a positive integer factorial enough so that the automorphism
group of every point of $\cX$ has order dividing $r$. For any line
bundle $\cL$ on $\cX$, $\pi_{*}(\cL^{r})$ is a line bundle on $X$.
The $\QQ$-linear map $\Pic(\cX)_{\QQ}\to\Pic(X)_{\QQ}$ sending $\cL$
to the $\QQ$-line bundle $(\pi_{*}(\cL^{r}))^{1/r}$ is the inverse
of $\pi^{*}$. 
\end{proof}
Since $\cX$ is smooth, it has the \emph{canonical line bundle} $\omega_{\cX}:=\det(\Omega_{\cX/F})$.
On the other hand, the coarse moduli space $X$ is not generally smooth,
but has quotient singularities. In particular, $X$ is $\QQ$-factorial.
The canonical sheaf $\omega_{X}$ is defined to be the unique reflexive
sheaf such that $\omega_{X}|_{X_{\sm}}=\omega_{X_{\sm}}$ with $X_{\sm}$
denoting the smooth locus. This is not a line bundle in general. However,
for some positive integer $r$, the $r$-th reflexive power $\omega_{X}^{[r]}$,
which can be constructed as the double dual $(\omega_{X}^{r})^{\vee\vee}$
of $\omega_{X}^{r}$, is an invertible sheaf. See \cite{kollar2013singularities,ishii2018introduction}
for more details on these contents. Thus, we get the \emph{canonical
$\QQ$-line bundle }of $X$ as $(\omega_{X}^{[r]})^{1/r}$. The corresponding
element of $\Pic(X)_{\QQ}$ is independent of the choice of $r$.
The canonical line bundle $\omega_{\cX}$ of $\cX$ and the canonical
$\QQ$-line bundle $(\omega_{X}^{[r]})^{1/r}$ of $X$ do not generally
correspond to each other by the isomorphism $\Pic(\cX)_{\QQ}\cong\Pic(X)_{\QQ}$.
However, if $\cX\to X$ is étale in codimension one, then they correspond
to each other. This is the case, for example, when $\cX\to X$ is
a gerbe or when $\cX$ has the trivial generic stabilizer and has
no reflection (Definition \ref{def:reflection}).

For a place $v\in M_{F}$, we denote by $|\cdot|_{v}$ the $v$-adic
norm on $F_{v}$ normalized as follows. For finite places, we normalize
as $|\pi_{v}|_{v}=q_{v}^{-1}$ with $\pi_{v}$ a uniformizer of $F_{v}$.
For a real place, we let $|\cdot|_{v}$ to be the standard absolute
value. For a complex place, we let $|\cdot|_{v}$ to be the square
of the standard absolute value. Let $\cL$ be a line bundle on $\cX$
and let $v\in M_{F}$. We have a continuous map
\[
\eta_{v}\colon\cL\langle F_{v}\rangle\to\cX\langle F_{v}\rangle.
\]
This is not quite a line bundle, but has a structure close to it.
For an isomorphism class $[x]\in\cX\langle F_{v}\rangle$, the fiber
$\eta_{v}^{-1}([x])$ is the quotient of the line $x^{*}\cL\cong F_{v}$
by the action of $\Aut(x)$. Since this action preserves the zero
vector, we get the zero section $\cX\langle F_{v}\rangle\to\cL\langle F_{v}\rangle$
and it makes sense to say whether an element of $\cL\langle F_{v}\rangle$
is zero or not. The following definition restates \cite[Def.\ 4.3.1.1]{darda2021rational2}
in a slightly different way:
\begin{defn}
A \emph{$v$-adic metric} on $\cL$ means a continuous map $\left\Vert \cdot\right\Vert _{v}\colon\cL\langle F_{v}\rangle\to\RR_{\ge0}$
such that 
\begin{enumerate}
\item for $a\in F_{v}$ and $[s]\in\cL\langle F_{v}\rangle$, $\left\Vert [as]\right\Vert _{v}=|a|_{v}\cdot\left\Vert [s]\right\Vert _{v}$,
\item for each $F_{v}$-morphism $f\colon U\to\cX_{F_{v}}$ from an $F_{v}$-scheme
$U$ of finite type and for every section $s\colon U(F_{v})\to\cL_{U}(F_{v})$,
the composition 
\[
U(F_{v})\xrightarrow{s}\cL_{U}(F_{v})\to\cL\langle F_{v}\rangle\xrightarrow{\left\Vert \cdot\right\Vert _{v}}\RR_{\ge0}
\]
is continuous.
\end{enumerate}
\end{defn}

As in the case of varieties, for almost every place $v$, there exists
a canonical choice of $v$-adic metric on the given $\cL$ derived
from a model. Let $\cX$ be a model of $\cX$ over $\cO_{S}$ and
let $\cL_{\cO_{S}}$ be a line bundle on $\cX_{\cO_{S}}$ which is
a model of $\cL$. For $v\in M_{F}\setminus S$ and for $x\in\cX(F_{v})$,
the $\overline{F_{v}}$-point $x_{\overline{F_{v}}}\in\cX(\overline{F_{v}})$
induced from $x$ uniquely extends to an $\cO_{\overline{F_{v}}}$-point
$x_{\cO_{\overline{F_{v}}}}\in\cX(\cO_{\overline{F_{v}}})$. Therefore,
the one-dimensional $\overline{F_{v}}$-vector space $(x_{\overline{F_{v}}})^{*}\cL$
has the canonical lattice $(x_{\cO_{\overline{F_{v}}}})^{*}\cL_{\cO_{S}}\cong\cO_{\overline{F_{v}}}$.
We denote again by $|\cdot|_{v}$ the unique extension of $|\cdot|_{v}$
to $\overline{F_{v}}$. We choose an isomorphism $(x_{\overline{F_{v}}})^{*}\cL\cong\overline{F_{v}}$
by which $(x_{\cO_{\overline{F_{v}}}})^{*}\cL_{\cO_{S}}$ and $\cO_{\overline{F_{v}}}$
correspond to each other, which induces a norm $\left\Vert \cdot\right\Vert _{v}$
on $(x_{\overline{F_{v}}})^{*}\cL$ corresponding to $|\cdot|_{v}$
on $\overline{F_{v}}$. 
\[
\xymatrix{\cO_{\overline{F_{v}}}\ar@{^{(}->}[r]\ar@{=}[d]^{\wr} & \overline{F_{v}}\ar[r]^{|\cdot|_{v}}\ar@{=}[d]^{\wr} & \RR\\
(x_{\cO_{\overline{F_{v}}}})^{*}\cL_{\cO_{S}}\ar@{^{(}->}[r] & (x_{\overline{F_{v}}})^{*}\cL\ar[ur]_{\left\Vert \cdot\right\Vert _{v}}
}
\]
The norm $\left\Vert \cdot\right\Vert _{v}$ given in this way is
independent of the choice of the isomorphism $(x_{\overline{F_{v}}})^{*}\cL\cong\overline{F_{v}}$
as above and invariant under the actions of $\Aut(x_{\overline{F_{v}}})$
and $\Aut(x)$. Thus, we get a continuous map $\left\Vert \cdot\right\Vert _{v}\colon\cL\langle\overline{F_{v}}\rangle\to\RR_{\ge0}$.
\begin{defn}[Adelic metric]
For $v\in M_{F}\setminus S$, the \emph{$v$-adic metric on $\cL$
induced by the model $\cL_{\cO_{S}}$} is defined to be the composite
map
\[
\cL\langle F_{v}\rangle\to\cL\langle\overline{F_{v}}\rangle\xrightarrow{\left\Vert \cdot\right\Vert _{v}}\RR_{\ge0};
\]
we denote it again by $\left\Vert \cdot\right\Vert _{v}$. An \emph{adelic
metric} on $\cL$ is a collection $(\left\Vert \cdot\right\Vert _{v})_{v\in M_{F}}$
of $v$-adic metrics such that for almost every $v$, $\left\Vert \cdot\right\Vert _{v}$
is induced from a model. 
\end{defn}

\begin{rem}
When $\cX$ is a variety, the $v$-adic metric $\left\Vert \cdot\right\Vert _{v}$
induced by a model takes values in $\{0\}\cup q_{v}^{\ZZ}\subset\RR_{\ge0}$.
This is no longer true for stacks. For a point $x\in\cX(F_{v})$,
the map 
\[
F_{v}\cong x^{*}\cL\to\cL\langle F_{v}\rangle\xrightarrow{\left\Vert \cdot\right\Vert _{v}}\RR_{\ge0}
\]
takes values in $\{0\}\cup q_{v}^{r+\ZZ}$ for some $r\in\QQ$. 
\end{rem}

\begin{defn}[Stable height]
\label{def:stable-height}Let $\cL$ be a line bundle on $\cX$ given
with an adelic metric $(\left\Vert \cdot\right\Vert _{v})_{v\in M_{F}}$
and let $x\in\cX(F)$. We choose a nonzero element $0\ne s\in x^{*}\cL\cong F$.
For each place $v\in M_{F}$, we get the induced $F_{v}$-point $x_{v}\in\cX(F_{v})$
and $0\ne s_{v}\in(x_{v})^{*}\cL\cong F_{v}$. We define the \emph{height
function }$H_{\cL}\colon\cX\langle F\rangle\to\RR$ by
\[
H_{\cL}(x):=\prod_{v\in M_{F}}\left\Vert s_{v}\right\Vert _{v}^{-1}.
\]
\end{defn}

This function is well-defined; it is independent of the choice of
an $F$-point $x$ from an isomorphism class as well as of the choice
of $s$. If $\cL$ and $\cL'$ are line bundles given with adelic
metrics, then the tensor product $\cL\otimes\cL'$ has the naturally
induced adelic metric. The height function has the following multiplicativity:
\[
H_{\cL\otimes\cL'}(x)=H_{\cL}(x)\cdot H_{\cL'}(x).
\]
The height function has also the following functoriality. For a morphism
$f\colon\cY\to\cX$ of nice stacks and a line bundle $\cL$ on $\cX$
given with an adelic metric, the pullback $f^{*}\cL$ has a naturally
induced adelic metric. For every $y\in\cY(F)$, 
\[
H_{f^{*}\cL}(y)=H_{\cL}(f(y)).
\]
The above constructions have natural generalizations to $\QQ$-line
bundles:
\begin{defn}
An \emph{adelic metric} on a $\QQ$-line bundle $\bigotimes_{i=1}^{n}\cL_{i}^{s_{i}}$
is a collection consisting of an adelic metric for each $\cL_{i}$.
For a $\QQ$-line bundle $\cL=\bigotimes_{i=1}^{n}\cL_{i}^{s_{i}}$,
given with an adelic metric, the \emph{height function} of $\cL$
is defined by
\[
H_{\cL}=\prod_{i}(H_{\cL_{i}})^{s_{i}}.
\]
\end{defn}

\begin{rem}[Stability properties]
\label{rem:stability}The height function $H_{\cL}$ has the following
stability properties, which were observed already in \cite{darda2021rational2,ellenberg2021heights}. 
\begin{enumerate}
\item Let $K/F$ be a finite extension. For a line bundle $\cL$ on $\cX$,
let $\cL_{K}$ be the induced line bundle on $\cX_{K}$. Suppose that
$\cL_{K}$ is given an adelic metric, which induces one on $\cL$
in the obvious way. In this situation, the associated height functions
$H_{\cL}$ and $H_{\cL_{K}}$ are related as follows; for an $F$-point
$x\in\cX\langle F\rangle$ and for the induced $K$-point $x_{K}\in\cX_{K}\langle K\rangle$,
\[
H_{\cL_{K}}(x_{K})=H_{\cL}(x)^{[K:F]}.
\]
\item Let $\cL$ and $L$ be $\QQ$-line bundles on $\cX$ and $X$ corresponding
to each other. An adelic metric on $\cL$ induces one on $L$, and
vice versa. When they are given adelic metrics corresponding to each
other in this way, the height functions $H_{\cL}$ and $H_{L}$ are
related by
\[
H_{\cL}=H_{L}\circ\pi,
\]
where $\pi$ is the map $\cX\langle F\rangle\to X(F)$.
\end{enumerate}
\end{rem}

\begin{rem}
\label{rem:Northcott-fails}Even if a line bundle $\cL$ on $\cX$
corresponds to an \emph{ample }$\QQ$-line bundle on $X$, the height
function $H_{\cL}\colon\cX\langle F\rangle\to\RR$ does not generally
satisfy the Northcott property.\footnote{We say that a function $H\colon U\to\RR$ has the Northcott property
if for every $B\in\RR$, the set $\{x\in U\mid H(x)\le B\}$ is finite.} This follows from either of the two stability properties in Remark
\ref{rem:stability}. For, by a map $\cX\langle K\rangle\to\cX\langle F\rangle$
or $\cX\langle F\rangle\to X(F)$, it may happen that infinitely many
distinct points map to a single point. This is the main reason why
we need to introduce ``unstable'' heights in the next section.
\end{rem}

\section{Raising data and unstable heights\label{sec:Raising-data}}
\begin{defn}
A \emph{raising function }of $\cX$ is a function
\[
c\colon\pi_{0}(\cJ_{0}\cX)\to\RR
\]
satisfying $c(\cX)=0$. A \emph{raising datum} of $\cX$ is a collection
$c_{*}=(c,(c_{v})_{v\in M_{F}})$ of a raising function $c$ and continuous
functions,
\[
c_{v}\colon\cX\langle F_{v}\rangle\to\RR,
\]
such that for almost every $v$, $c_{v}$ factors as $c_{v}=c\circ\psi_{v}$
with $\psi_{v}$ the residue map (see Definition \ref{def:residue-map}):
\[
c_{v}:\cX\langle F_{v}\rangle\xrightarrow{\psi_{v}}\pi_{0}(\cJ_{0}\cX)\xrightarrow{c}\RR.
\]
We call $c$ the \emph{generic raising function of} $c_{*}$. We say
that a raising function $c$ is \emph{positive} if $c(\cY)>0$ for
every twisted sector $\cY$. We say that a raising datum is \emph{positive
}if its generic raising function is positive.
\end{defn}

Note that from \cite[Prop.\ 2.9]{vcesnavivcius2015topology}, $\cX\langle F_{v}\rangle$
are compact and the continuous maps $c_{v}$ are bounded. Note also
that if $\psi_{v}$ are surjective for infinitely many $v$, then
the generic raising function $c$ of $c_{*}=(c,(c_{v})_{v\in M_{F}})$
is determined by the subcollection $(c_{v})_{v\in M_{F}}$. If $\cX=\B G$
for an étale group scheme $G$ over $F$, then $\psi_{v}$ are surjective
for infinitely many $v$. In this case, the raising function is essentially
the same as what is called the rational class function in \cite{ellenberg2005counting}
and the counting function in \cite{wood2010onthe,darda2022torsors}.

\begin{defn}
A \emph{raised line bundle }(resp.~\emph{strictly raised line bundle})
on $\cX$ means the pair $(\cL,c)$ (resp.~$(\cL,c_{*})$) of a line
bundle $\cL$ on $\cX$ and a raising function $c$ (resp.~a raising
datum $c_{*}$). 
\end{defn}

\begin{defn}[Unstable height]
\label{def:height}For $x\in\cX(F)$ and $v\in F_{v}$, let $x_{v}$
denote the induced $F_{v}$-point $x_{v}\in\cX(F_{v})$. Let $(\cL,c_{*})$
be a strictly raised line bundle. We suppose that $\cL$ is given
an adelic metric so that the height function $H_{\cL}$ is defined
as in Definition \ref{def:stable-height}. Then, we define the \emph{associated
height function} of $(\cL,c_{*})$ to be 
\[
H_{\cL,c_{*}}(x):=H_{\cL}(x)\times\prod_{v\in M_{F}}q_{v}^{c_{v}(x_{v})}.
\]
Here $q_{v}$ is the cardinality of the residue field $\kappa_{v}$
for a finite place $v$ and we put $q_{v}:=e$ for an infinite place
$v$.
\end{defn}

Note that if we fix $x\in\cX(F)$, then for almost every $v$, the
residue $\psi_{v}(x_{v})$ is the non-twisted sector, and hence $c_{v}(x_{v})=0$.
Thus, the above product $\prod_{v\in M_{F}}q_{v}^{c_{v}(x_{v})}$
is a finite product. For strictly raised line bundles $(\cL,c_{*})$
and $(\cL',c_{*}')$, the pair $(\cL\otimes\cL',c_{*}+c_{*}')$ is
also a strictly raised line bundle. The height function satisfies
the multiplicativity:
\[
H_{\cL\otimes\cL',c_{*}+c_{*}'}(x)=H_{\cL,c_{*}}(x)\cdot H_{\cL',c_{*}'}(x).
\]
The height function has the functoriality:
\begin{prop}
Let $f\colon\cY\to\cX$ be a morphism of nice stacks, let $(\cL,c_{*})$
be a strictly raised line bundle on $\cX$, which induces a strictly
raised line bundle on $\cY$. For every $F$-point $y$ of $\cY$,
we have
\[
H_{f^{*}\cL,f^{*}c_{*}}(y)=H_{\cL,c_{*}}(f(y)).
\]
\end{prop}

\begin{proof}
Note that the raising datum $c_{*}=(c,(c_{v})_{v\in M_{F}})$ for
$\cX$ induces the collection of maps, $f^{*}c_{*}=(f^{*}c,(f^{*}c_{v})_{v})$,
by natural maps $\pi_{0}(\cJ_{0}\cY)\to\pi_{0}(\cJ_{0}\cX)$ and $\cY\langle F_{v}\rangle\to\cX\langle F_{v}\rangle$.
Using Proposition \ref{prop:residue-functorial}, we can show that
this collection is again a raising datum. Suppose that we have a morphism
$f\colon\cY_{\cO_{S}}\to\cX_{\cO_{S}}$ of models over $\cO_{S}$
of nice stacks and that $\cJ_{0}\cY_{\cO_{S}}$ and $\cJ_{0}\cX_{\cO_{S}}$
are flat over $\cO_{S}$. The desired equality follows from the functoriality
of stable height functions mentioned in Section \ref{sec:Line-bundles}
and the definition of unstable height functions.
\end{proof}
If two raising data $c_{*}$ and $c_{*}'$ have the same generic raising
function $c$, then the function $H_{\cL,c_{*}}/H_{\cL,c_{*}'}$ is
bounded. Thus, a height function $H_{\cL,c}$ of a (non-strictly)
raised line bundle $(\cL,c)$ is determined modulo bounded functions.
In what follows, we mainly consider (non-strictly) raised line bundles
and their height functions.
\begin{rem}[Unstableness, cf.~Remark \ref{rem:stability}]
\label{rem:unstable}Let $K/F$ be a finite extension. Suppose that
a line bundle $\cL$ on $\cX$ and the induced line bundle $\cL_{K}$
on $\cX_{K}$ have raising data $c_{*}$ and $d_{*}$ which are compatible
with each other in the obvious way. Then, for an $F$-point $x\in\cX(F)$
and for the induced $K$-point $x_{K}\in\cX(K)=\cX_{K}(K)$, the formula
\[
H_{\cL_{K},d_{*}}(x_{K})=H_{\cL,c_{*}}(x)^{[K:F]}
\]
does \emph{not} hold unlike the case of stable height as in Remark
\ref{rem:stability}. For example, if $\cX=\B G$ for a finite group
$G$ and if $x\in\cX(F)$ is a $G$-torsor over $F$, then $H_{\cL,c_{*}}(x)$
is a generalized discriminant as will be explained in Example \ref{exa:discriminant}
and can become arbitrarily large as $x$ varies. On the other hand,
$x$ is trivialized by some finite extension $K/F$ and the corresponding
$K$-point $x_{K}\in\cX(K)$ has height bounded by a constant independent
of $x$. 
\end{rem}

\begin{lem}
\label{lem:Dedeking-Iso}Let $T$ be a connected Dedekind scheme and
let $\cX$ be a separated DM stack over $T$. Let $a,b,c\in\cX(T)$.
Let $\ulIso_{T}(a,b)^{\circ}$ denote the open and closed subscheme
of the Iso scheme $\ulIso_{T}(a,b)$ consisting of one-dimensional
connected components. Similarly for $\ulIso_{T}(a,c)^{\circ}$, $\ulAut_{T}(a)^{\circ}$,
etc.
\begin{enumerate}
\item The scheme $\ulIso_{T}(a,b)^{\circ}$ is finite and étale over $T$.
In particular, $\ulAut_{T}(a)^{\circ}$ is a finite étale group scheme
over $T$.
\item If the scheme $\ulIso_{T}(a,b)^{\circ}$ is not empty, then it has
a natural structure of $\ulAut_{T}(b)^{\circ}$-torsor as well as
one of $\ulAut_{T}(a)^{\circ}$-torsor. In particular, $\ulIso_{T}(a,b)^{\circ}$
is étale and finite over $T$.
\item Suppose that there exists an isomorphism $\ulIso_{T}(a,b)^{\circ}\to\ulIso_{T}(a,c)^{\circ}$
which is equivariant for $\ulAut_{T}(a)^{\circ}$-actions. Then, $b$
and $c$ are isomorphic in $\cX(T)$. 
\end{enumerate}
\end{lem}

\begin{proof}
Note that $\ulIso_{T}(a,b)$ is finite and unramified over $T$. From
the local structure of unramified morphisms \cite[tag 04HJ]{stacksprojectauthors2022stacksproject},
$\ulIso_{T}(a,b)^{\circ}$ is étale over $T$. It is easy to see that
$\ulAut_{T}(a)^{\circ}$ is a subgroup scheme of $\ulAut_{T}(a)$.
The first assertion follows.

As for the second assertion, we easily see that there is a natural
action of $\ulAut_{T}(b)^{\circ}$-action on $\ulIso_{T}(a,b)^{\circ}$
and that this action makes $\ulIso_{T}(a,b)^{\circ}$ an $\ulAut_{T}(b)^{\circ}$-torsor
over an open dense subscheme of $T$, provided that $\ulIso_{T}(a,b)^{\circ}\ne\emptyset$.
For a geometric point $t\in T(K)$, the action of $\ulAut_{T}(b)^{\circ}(t)$
on $\ulIso_{T}(a,b)^{\circ}(t)$ is free. Since $\ulAut_{T}(b)^{\circ}(t)$
and $\ulIso_{T}(a,b)^{\circ}(t)$ have the same cardinality, the action
is also transitive. It follows that $\ulIso_{T}(a,b)^{\circ}$ is
an $\ulAut_{T}(b)^{\circ}$-torsor over the whole $T$. Similarly
for the $\ulAut_{T}(a)^{\circ}$-action.

We show the third assertion. Note that we have the obvious $T$-isomorphism
$\ulIso_{T}(a,b)^{\circ}\cong\ulIso_{T}(b,a)^{\circ}$. Consider the
morphism:
\[
\ulIso_{T}(b,a)^{\circ}\times\ulIso_{T}(a,c)^{\circ}\to\ulIso_{T}(b,c)^{\circ},\,(f,g)\mapsto g\circ f.
\]
This morphism is invariant under the $\ulAut_{T}(a)^{\circ}$-action
given by $\alpha(f,g):=(\alpha\circ f,g\circ\alpha^{-1})$ and induces
a morphism
\begin{align*}
\left(\ulIso_{T}(b,a)^{\circ}\times\ulIso_{T}(a,c)^{\circ}\right) & /\ulAut_{T}(a)^{\circ}\to\ulIso_{T}(b,c)^{\circ}.
\end{align*}
The last morphism is a morphism of $\ulAut_{T}(c)^{\circ}$-torsors,
and hence an isomorphism. From the assumption, we have the induced
morphism
\[
\ulIso_{T}(b,a)^{\circ}\to\ulIso_{T}(b,a)^{\circ}\times\ulIso_{T}(a,c)^{\circ},
\]
which is equivariant for $\ulAut_{T}(a)^{\circ}$-actions. We get
the composite morphism
\[
T=\ulIso_{T}(a,b)^{\circ}/\ulAut_{T}(a)^{\circ}\to\left(\ulIso_{T}(b,a)^{\circ}\times\ulIso_{T}(a,c)^{\circ}\right)/\ulAut_{T}(a)^{\circ}\to\ulIso_{T}(b,c)^{\circ}.
\]
This shows that $\ulIso_{T}(b,c)^{\circ}(T)$ is nonempty, and that
$b$ and $c$ are isomorphic in $\cX(T)$. 
\end{proof}
\begin{lem}
\label{lem:num-torsor-str}Let $T$ be a connected Dedekind scheme
and let $G$ be a finite étale group scheme of order $r$ over $T$.
Let $P$ be a finite étale $T$-scheme of degree $r$. Then, the number
of $G$-torsor structures that can be given to $P$ is at most $(r!)^{r}$. 
\end{lem}

\begin{proof}
A $G$-torsor structure on $P$ is determined by a homomorphism $G\to\ulAut_{T}(P)$
of group schemes over $T$. The Aut scheme $\ulAut_{T}(P)$ is a finite étale $T$-scheme of degree
$r!$. We can see this, for example, by taking a base change to reduce to the case where $P$ is the trivial $T$-scheme  $T \sqcup \cdots \sqcup T \to T$ of degree $r$.
 If $t\colon \Spec K \to T$ is the geometric generic point, then there are $(r!)^r$ $K$-morphisms $G_K\to \ulAut_{T}(P)_K$. It follows that 
the number of $T$-morphisms $G\to\ulAut_{T}(P)$  as well as the number of $T$-\textit{homo}morphisms $G\to\ulAut_{T}(P)$ is at most
$(r!)^r$.
\end{proof}
\begin{prop}[Northcott property]
\label{prop:Northcott}Let $(\cL,c)$ be a raised line bundle on
$\cX$ and let $L$ be the $\QQ$-line bundle corresponding to $\cL$
on the coarse moduli space $\overline{\cX}$. Suppose that $L$ is
ample and $c$ is positive. Let $H_{\cL,c}$ be the height function
of $(\cL,c)$. Then, for each real number $B$, there are only finitely
many points $x\in\cX\langle F\rangle$ with $H_{\cL,c}(x)\le B$.
Moreover, there exist positive constants $C$ and $m$ such that for
every $B>0$,
\[
\#\{x\in\cX\langle F\rangle\mid H_{\cL,c}(x)\le B\}\le CB^{m}.
\]
\end{prop}

\begin{proof}
We choose a raising datum $c_{*}$ refining $c$ and prove the proposition
for $H_{\cL,c}=H_{\cL,c_{*}}$. Note that the second factor $\prod_{v\in M_{F}}q_{v}^{c_{v}(x_{v})}$
in the definition of $H_{\cL,c_{*}}$ is at least 1, and hence $H_{\cL}(x)\le H_{\cL,c_{*}}(x)$
for every $x$. 
On the other hand, the first factor $H_{\cL}(x)$ is equal to $H_{L}(\pi(x))$ with $\pi$ denoting the natural map $\cX\langle F\rangle \to X(F)$.
Using \cite[Corollary, p. 447]{schanuel1979heights} and
basic properties of the Weil height machine, we can see that for some
positive constants $C_{1}$ and $m_{1}$, 
\[
    \#\pi(\{x\in\cX\langle F\rangle\mid H_{\cL,c}(x)\le B\})\le
\#\{y\in\overline{\cX}(F)\mid H_{L}(y)\le B\}\le C_{1}B^{m_{1}}.
\]
If $H_{\cL,c_{*}}(x)\le B$, then $\prod_{v}q_{v}^{c_{v}(x_{v})}\le a_{1}B$
for some constant $a_{1}$. It follows that there are some positive
constants $a_{2}$ and $a_{3}$ such that if $H_{\cL,c}(x)\le B$,
then the $F$-point $x$ extends to an $\cO_{T}$-point of $\cX$
for a finite set $T$ of places satisfying 
\[
\prod_{v\in T\text{ finite places}}q_{v}\le a_{2}B^{a_{3}}.
\]

Let $x_{i}\in\cX(F)$, $i\in\{0,1\}$ be $F$-points of $\cX$ with
$H_{\cL,c_{*}}(x_{i})\le B$ which have the same image $y\in\overline{\cX}(F)$.
Suppose that $x_{i}$ extends to an $\cO_{T_{i}}$-point with $\prod_{v\in T_{i}}q_{v}\le a_{2}B^{a_{3}}$.
Then, the $\ulAut_{F}(x_{0})$-torsor $\ulIso_{F}(x_{0},x_{1})$ extends
to $\ulAut_{\cO_{T_{0}\cup T_{1}}}(x_{0}')^{\circ}$-torsor $\ulIso_{\cO_{T_{0}\cup T_{1}}}(x_{0}',x_{1}')^{\circ}$,
where $x_{i}'$ are the $\cO_{T_{0}\cup T_{1}}$-points induced from
$x_{i}$ respectively. Let $N$ be an integer such that the automorphism
group of every point of $\cX$ has order $\le N$ and let $M$ be
an integer such that for every finite place $v$, every finite étale
cover $\Spec L\to\Spec F_{v}$ of degree $\le N$ has discriminant
exponent at most $M$. Then, the torsor $\ulIso_{F}(x_{0},x_{1})^{\circ}$
has discriminant at most $(a_{2}B^{a_{3}})^{2M}$. From \cite{schmidt1995numberfields}
and Lemma \ref{lem:num-torsor-str}, there exist positive constants
$C_{2}$ and $m_{2}$ such that the number of $\ulAut_{F}(x_{0})$-torsors
with discriminant at most $(a_{2}B^{a_{3}})^{2M}$ is at most $C_{2}B^{m_{2}}$.
Note that the result in \cite{schmidt1995numberfields} is about counting
number fields, but it is easy to generalize it to counting torsors.
From Lemma \ref{lem:Dedeking-Iso}, the number of lifts of $y$ in
$\cX\langle F\rangle$ is at most $C_{2}B^{m_{2}}$. In summary, we
have
\[
\#\{x\in\cX\langle F\rangle\mid H_{\cL,c_{*}}(x)\le B\}\le(C_{1}B^{m_{1}})(C_{2}B^{m_{2}})=(C_{1}C_{2})B^{m_{1}+m_{2}}.
\]
\end{proof}
\begin{rem}
A slight change of the above proof derives a stronger version of the
above proposition: for every positive integer $n$, there exist positive
constants $C$ and $m$ such that for every $B>0$,
\[
\#\left\{ x\in\bigsqcup_{L/F\colon\text{finite field ext.}}\cX\langle L\rangle\bigmid[L:F]\le n\text{ and }H_{\cL,c}(x)\le B\right\} \le CB^{m}.
\]
For this purpose, we need to use a result in \cite{schmidt1993northcotts},
a polynomial bound for algebraic points with bounded degree and bounded
height, in place of one in \cite{schanuel1979heights}. Note that
for a finite extension $L'/L$, the map $\cX\langle L\rangle\to\cX\langle L'\rangle$
is not generally injective, and the height $H_{\cL,c_{*}}$ is not
preserved by this map (Remark \ref{rem:unstable}). This is why we
take the \emph{disjoint} union $\bigsqcup_{L/F}\cX\langle L\rangle$.
\end{rem}

\begin{example}[The height function associated to a vector bundle]
Ellenberg--Satriano--Zureick-Brown \cite{ellenberg2021heights}
introduced a new height function on stacks which is associated to
a vector bundle. They consider also Artin stacks, but if we restrict
ourselves to DM stacks over number fields, then our height function
$H_{\cL,c_{*}}$ generalizes theirs. Let $\cV$ be a vector bundle
on a nice stack $\cX$. Translating their definition into the setting
of multiplicative heights, we can write the height function $H_{\cV}$
associated to $\cV$ as 
\[
H_{\cV}(x)=H_{\det(\cV)}(x)\times\prod_{v}q_{v}^{\epsilon_{v}(x_{v})}.
\]
The exponent $\epsilon_{v}(x_{v})$ is generalization of the $w$-
and $v$-functions in \cite{yasuda2017towardmotivic,wood2015massformulas}
to semilinear representations. We claim that $(\epsilon_{v})_{v}$
together with some raising function is a raising datum. Indeed, at
tame places $v$, $\epsilon_{v}(x_{v})$ is nothing but the age, $\age(-;\cV)$
(cf.~\cite[Example 6.7]{yasuda2017towardmotivic}, \cite[Lemma 4.3]{wood2015massformulas}).
To see this, let $v\in M_{F}$ be a general finite place and let $x\in\cX(L)$
with $L=F_{v}^{\nr}$, following the notation of Lemma \ref{lem:ext_tw_arc}.
We take the induced representable morphism $\widetilde{x}\colon[\Spec\cO_{L_{l}}/\mu_{l,\cO_L}]\to\cX_{\cO_{S}}$.
The vector bundle $\widetilde{x}^{*}\cV$ on $[\Spec\cO_{L_{l}}/\mu_{l,\cO_L}]$
corresponds to a free $\cO_{L_{l}}$-module of finite rank with a
semilinear action of $\mu_{l}=\ZZ/l\ZZ$. This module is equivariantly
isomorphic to
\[
M=\bigoplus_{i=1}^{r}\fm_{L_{l}}^{a_{i}}\quad(0\le a_{i}<l),
\]
where $\fm_{L_{l}}$ denotes the maximal ideal of $\cO_{L_{l}}$.
If we put 
\[
b_{i}:=\begin{cases}
0 & (a_{i}=0)\\
l & (a_{i}\ne0)
\end{cases},
\]
then
\begin{align*}
\epsilon_{v}(x_{v}) & =\frac{1}{l}\length\frac{M}{M^{\mu_{l}}\cdot\cO_{L_{l}}}\\
 & =\frac{1}{l}\length\bigoplus_{i=1}^{r}\fm_{L_{l}}^{a_{i}}/\fm_{L_{l}}^{b_{i}}\\
 & =\frac{1}{l}\sum_{i=1}^{r}(b_{i}-a_{i}).
\end{align*}
Note that the special fiber of $\widetilde{x}^{*}\cV$ is identified
with the representation $\left(\bigoplus_{i}\fm_{L_{l}}^{a_{i}}/\fm_{L_{l}}^{a_{i}+1}\right)^{\vee}$.
The age computed from the last representation is equal to $\epsilon_{v}(x_{v})$.
From Lemma \ref{lem:age-sector}, the age depends only on the sector
associated to $x_{v}$. Thus, for almost every place $v$, $\epsilon_{v}$
is the composite map
\[
\cX(F_{v})\xrightarrow{\psi_{v}}\pi_{0}(\cJ_{0}\cX)\xrightarrow{\age(-;\cV)}\RR_{\ge0}.
\]
Namely, the collection $(\age(-;\cV),(\epsilon_{v})_{v})$ is a raising
datum. The height function $H_{\cV}$ is equal to $H_{\det(\cV),\epsilon_{*}}$. 
\end{example}

\begin{example}[Raising functions not coming from any vector bundle]
The raising function $\age(-;\cV)$ in the last example associated
to a vector bundle $\cV$ is $\QQ$-valued. Thus, any raising function
with non-rational values is not associated to any vector bundle. If
a raising function takes a rational value at some twisted sector and
a non-rational value at another twisted sector, then this raising
function is not obtained by multiplying $\age(-;\cV)$ with a positive
real constant. There is also a more interesting example of raising
function that takes only rational values, but not coming from a vector
bundle. Consider $\cX=\B\mu_{p,F}$ for a prime number $p$. Then,
$\pi_{0}(\cJ_{0}\cX)=\ZZ/p\ZZ=\{0,1,\dots,p-1\}$. Consider a vector
bundle of $\cX$ associated to the representation
\[
V=\bigoplus_{i=1}^{d}L^{\otimes a_{i}}\quad(0\le i<p).
\]
Here $L$ is the standard one-dimensional representation of $\mu_{p}$.
Since the summands with $a_{i}=0$ do not contribute to the raising
function, we may assume that $a_{i}>0$ for every $i$, without loss
of generality. Then, the raising function associated to $\cV$, $c_{\cV}\colon\pi_{0}(\cJ_{0}\cX)=\{0,1,\dots,p-1\}\to\QQ$,
is given by $c_{\cV}(j)=\sum_{i=1}^{d}\left\{ \frac{ja_{i}}{p}\right\} $,
where $\{-\}$ denotes the fractional part. In particular,
\[
\frac{d}{p}\le c_{\cV}(j)\le\frac{(p-1)d}{p}\quad(j\in\{1,\dots,p-1\}).
\]
This implies that for $j,j'\in\{1,\dots,p-1\}$,
\[
c_{\cV}(j)\le(p-1)c_{\cV}(j').
\]
If $c\colon\pi_{0}(\cJ_{0}\cX)=\{0,1,\dots,p-1\}\to\QQ$ is any function
violating the last inequality, it never comes from a vector bundle,
nor can it be a scalar multiple of a raising function associated to a vector
bundle. An example of such raising functions is the following one
for $p\ge3$:
\[
c(j)=\begin{cases}
0 & (j=0)\\
1 & (j=1,\dots,p-2)\\
p & (j=p-1).
\end{cases}
\]
\end{example}

\begin{example}[Generalized discriminants]
\label{exa:discriminant}Let $G\subset S_{n}$ be a transitive subgroup,
let $G_{1}\subset G$ be the stabilizer subgroup of $1\in\{1,\dots,n\}$,
and let $\cX=\B G_{F}$. An object of $\cX(F)$ is a $G$-torsor $K/F$.
Then, $K^{G_{1}}/F$ is an étale $F$-algebra of degree $n$. Similarly
for an object of $\cX(F_{v})$. For a finite place $v$ and for $G$-torsor
$K_{v}/F_{v}$, let $d_{v}(K_{v})\in\ZZ$ denote the discriminant
exponent of the degree $n$ algebra $K_{v}^{G_{1}}/F_{v}$. For an
infinite place, we put $d_{v}\equiv0$. Then, $d_{*}=(d_{v})$ is
a raising datum and its generic raising function is given by 
\[
c\colon\pi_{0}(\cJ_{0}\cX)=\FConj(G)\to\ZZ_{\ge0},\,[g]\mapsto\mathrm{ind}(g).
\]
Here $\ind(g)$ denotes the index appearing in the context of 
Malle's conjecture \cite{malle2002onthe,malle2004onthe}. The height
function associated to a general raising function $c$ (with putting
$c_{v}\equiv0$ at wild places $v$) is the modified discriminant
considered by Ellenberg--Venkatesh \cite[p.\ 163]{ellenberg2005counting}.
We refer the reader to \cite{darda2022torsors} for more details. 
\end{example}

\begin{rem}
Landesman \cite{landesman2021stackyheights} shows that the Faltings
height on $\overline{\cM_{1,1}}$ in characteristic three is not induced
from a vector bundle in the sense of \cite{ellenberg2021heights}.
The Faltings height in characteristic three would not fit into our
definition of height, either. For, the stack $\overline{\cM_{1,1}}$
in characteristic three is wild. In that case, the stack $\cJ_{0}\cX=\ulHom(\B\widehat{\mu},\cX)$
would not be the right object to look at. In \cite{yasuda2020motivic},
the stack of twisted 0-jets is considered also in the wild case, but
it is of infinite dimension and has a structure considerably more
complex than $\ulHom(\B\widehat{\mu},\cX)$. One might be able to
generalize our height by using this version of the stack of twisted
0-jets or something similar to incorporate ``wild aspects,'' in
particular, the Faltings height in characteristic three. 
\end{rem}

\section{Fano stacks; the stacky Manin conjecture\label{sec:Fano-stacks}}
\begin{defn}
We mean by a \emph{Fano stack }a nice stack $\cX$ such that the anti-canonical
line bundle $\omega_{\cX}^{-1}$ corresponds to an ample $\QQ$-line
bundle on the coarse moduli space $\overline{\cX}$ by the isomorphism
$\Pic(\cX)_{\QQ}\cong\Pic(\overline{\cX})_{\QQ}$ in Proposition \ref{prop:iso-Pic}. 
\end{defn}

Typical examples of Fano stacks are Fano varieties (except $\Spec F$)
and nice stacks of dimension zero.

\begin{defn}
\label{def:adequate-Fano}Let $c$ be a raising function of a Fano
stack $\cX$. We define a function $\age_{c}$ on $\pi_{0}(\cJ_{0}\cX)$
by $\age_{c}:=\age+c$. We say that $c$ is \emph{adequate }if 
\begin{enumerate}
\item $\age_{c}(\cY)\ge1$ for every twisted sector $\cY$, and
\item if $\dim\cX=0$, then $\min\{c(\cY)\mid\cY\text{ twisted sector}\}=1$.
\end{enumerate}
\end{defn}

Note that if $\dim\cX=0$, then the height function $H_{\cL,c}$ depends
only on $c=\age_{c}$ (up to bounded functions) and the choice of
$\cL$ plays no role. Therefore, we usually choose the structure sheaf
$\cO=\cO_{\cX}$ (say with the standard adelic metric). For a constant
$r\in\RR_{\ge0}$, we have
\[
H_{\cO,rc}=(H_{\cO,c})^{r}.
\]
Thus, the second condition in the above definition is just a normalization
condition and does not lead to a loss of generality. As for the first
condition, we speculate that without this condition, it would be more
difficult to control the asymptotic behavior of the number of rational
points with bounded height (see Remark \ref{rem:log-terminal}).
\begin{defn}
Let $\cX$ be a Fano stack given with an adequate raising function
$c$. We call a twisted sector $\cY$ of $\cX$ to be $c$-\emph{junior
}if $\age_{c}(\cY)=1$. We denote by $j_{c}(\cX)$ the number of $c$-junior
twisted sectors of $\cX$:
\[
j_{c}(\cX):=\#\{\cY\in\pi_{0}^{*}(\cJ_{0}\cX)\mid\age_{c}(\cY)=1\}.
\]
\end{defn}

\begin{defn}
\label{def:thin-mor}Let $f\colon\cY\to\cX$ be a morphism of nice
stacks. We say that $f$ is \emph{birational} if there exist open
dense substacks $\cV\subset\cY$ and $\cU\subset\cX$ such that $f$
restricts to an isomorphism $\cV\xrightarrow{\sim}\cU$. We say that
$f$ is \emph{thin} if it is non-birational, representable, and generically
finite onto the image. A \emph{thin subset }of $\cX\langle F\rangle$
means a subset contained in $\bigcup_{i=1}^{n}f_{i}(\cY_{i}\langle F\rangle)$
for finitely many thin morphisms $f_{i}\colon\cY_{i}\to\cX$. 
\end{defn}

\begin{defn}
For a nice stack $\cX$, we denote its \emph{Néron--Severi space} by
$\N^{1}(\cX)_{\RR}$, that is, the $\RR$-vector space generated by
numerical classes of line bundles. We denote its dimension by $\rho(\cX)$
and call it the \emph{Picard number }of $\cX$.
\end{defn}

From Proposition \ref{prop:iso-Pic}, the Néron--Severi space of $\cX$
is identified with the corresponding space $\N^{1}(\overline{\cX})_{\RR}$
of the coarse moduli space $\overline{\cX}$. In particular, $\cX$
and $\overline{\cX}$ have the same Picard number. Below is our first
conjecture, which concerns the number of rational points with bounded
height on a Fano stack.
\begin{conjecture}[The Manin conjecture for Fano stacks]
\label{conj:Fano-stack}Let $\cX$ be a Fano stack and let $c$ be
an adequate raising function. Suppose that $\cX\langle F\rangle$
is Zariski dense. Then, there exists a thin subset $T\subset\cX\langle F\rangle$
and a positive constant $C$ such that
\begin{align*}
\#\{x\in\cX\langle F\rangle\setminus T\mid H_{\omega_{\cX}^{-1},c}(x)\le B\} & \sim CB(\log B)^{\rho(\cX)+j_{c}(\cX)-1}\quad(B\to\infty).
\end{align*}
\end{conjecture}

\begin{rem}
The above conjecture, in particular, says that the left hand side is a finite number, even if \(c\) is not necessarily positive.
On the other hand, we proved the Northcott property in Proposition \ref{prop:Northcott} only when the raising function is positive and the line bundle corresponds to an ample \(\QQ\)-line bundle on the coarse moduli space. Note that a version of the Northcott property (the Northcott property with a certain Zariski closed subset removed) holds for a big (but not necessarily ample) divisor on a variety. We speculate that the same would be true for a big raised line bundle defined later in Definition \ref{def:big}, where the raising function is allowed to have negative values. In \cite{darda2023toric}, we proved this for the anticanonical raised line bundle for a toric stack. 
\end{rem}

\begin{example}[Malle's conjecture]
Let $G$ be a finite group and let $G_{F}$ be the corresponding
constant group scheme over $F$. Consider the DM stack $\cX:=\B G_{F}$
over $F$. From Example \ref{exa:BG tw sectors}, we may identify
$\pi_{0}(\cJ_{0}\cX)$ with the set $\FConj(G)$ of $F$-conjugacy
classes of $G$. From Example \ref{exa:discriminant}, the height
function $H_{\cO,c}$ of the raised line bundle $(\cO,c)$ is a generalization
of discriminant. For a proper subgroup $H\subsetneq G$, the induced
morphism $\B H_{F}\to\B G_{F}$ is a thin morphism. Thus, 
\[
\bigcup_{H\subsetneq G}\Image((\B H_{F})\langle F\rangle\to\cX\langle F\rangle)
\]
is a thin subset of $\cX\langle F\rangle$, whose complement is the
set of connected $G$-torsors over $F$, that is, $G$-Galois extensions
of $F$. If we can take this thin subset as $T$, then Conjecture
\ref{conj:Fano-stack} is the same as Malle's conjecture for a generalized
discriminant associated to $c$. See also Section \ref{subsec:Malle-revisited}
for discussion about the choice of a thin subset and one about the
leading constant. 
\end{example}

\begin{example}[Weighted projective stacks]
\label{exa:wt-proj-stack}For $\ba=(a_{0},\dots,a_{n})\in(\ZZ_{>0})^{n}$,
we consider the weighted projective stack 
\[
   \cX=\cP(\ba)=\cP(a_{0},\dots,a_{n})=[(\AA_F^{n+1}\setminus\{0\})/ \GG_{m,F}].
\]
Here we follow the convention that \(\GG_{m,F}\) acts on \(\AA_F^{n+1}\) by \(t\cdot (x_i)_i = (t^{-a_i}x_i)_i\). Note that using the alternative action \(t\cdot (x_i)_i = (t^{a_i}x_i)_i\) leads to an isomorphic stack via an automorphism of \(\GG_{m,F}\) given by \(t\mapsto t^{-1}\). 
Sectors of $\cP(\ba)$ are in one-to-one correspondence with elements
of
\[
I:=\left(\bigcup_{i}\frac{1}{a_{i}}\ZZ\right)\cap[0,1).
\]
When the base field is $\CC$, this is proved in \cite[3.b]{mann2008orbifold},
\cite[p.\ 148]{coates2009thequantum}. But, this is true over any
field, because the correspondence is valid over $\overline{F}$ and
the natural action of the absolute Galois group $\Gamma_{F}$ on $\pi_{0}(\cJ_{0}\cX_{\overline{F}})$
is trivial. From \cite[p.\ 149]{coates2009thequantum} or \cite[Prop.\ 3.9]{mann2008orbifold},
for the sector $\cY_{r}$ of $\cX$ corresponding to $r\in I$, we
have
\[
\age(\cY_{r})=\sum_{j=0}^{n}\{-a_{j}r\}.
\]
Here $\{-\}$ denotes the fractional part, that is, $\{t\}:=t-\left\lfloor t\right\rfloor $.
Let $|\ba|:=\sum_{j=0}^{r}a_{j}$, which is an important number since
$\omega_{\cX}^{-1}=\cO_{\cX}(|\ba|)$. Consider the raising function
of $\cX$ given by
\[
c(\cY_{r})=r\cdot|\ba|.
\]
Then, 
\[
\age_{c}(\cY_{r})=r\sum_{j=0}^{n}a_{j}+\sum_{j=0}^{n}\{-a_{j}r\}=\sum_{j}\left\lceil a_{j}r\right\rceil .
\]
Here $\left\lceil -\right\rceil $ denotes the ceiling function. In
particular, if $\dim\cX=n>0$, then $\age_{c}(\cY_{r})\ge2$ for $r\ne0$.
If $\dim\cX=0$, then 
\[
\age_{c}(\cY_{r})\begin{cases}
=0 & (r=0)\\
=1 & (r=1/a_{0})\\
>1 & (r>1/a_{0})
\end{cases}.
\]
We get that $(\omega_{\cX}^{-1},c)$ is adequate and that 
\[
\rho(\cX)+j_{c}(\cX)=1.
\]
The height function associated to a quasi-toric degree $|\ba|$ family
of smooth functions considered in \cite{darda2021rational2} is, in
our language, the height function $H_{\omega_{\cX}^{-1},c}$ associated
to the raised line bundle $(\omega_{\cX}^{-1},c)$ with the raising
function $c$ described above. The equality $\rho(\cX)+j_{c}(\cX)=1$
explains the result \cite[Th.\ 8.2.2.12]{darda2021rational2} that
the asymptotic formula for this height is linear (that is, there is
no log factor).
\end{example}

\section{Compatibility with products of Fano stacks\label{sec:Compatibility-with-products}}

For $i=1,2$, let $\cX_{i}$ be nice stacks, and let $(\cL_{i},c_{i})$
be raised line bundles on $\cX_{i}$, respectively. These line bundles
induce the line bundle $\cL_{1}\boxtimes\cL_{2}:=\pr_{1}^{*}\cL_{1}\otimes\pr_{2}^{*}\cL_{2}$
on the product $\cX_{1}\times\cX_{2}$. The following lemma enables
us to derive a raising function of $\cX_{1}\times\cX_{2}$ from $c_{1}$
and $c_{2}$:
\begin{lem}
\label{lem:J0 prod}There exists a natural isomorphism $\cJ_{0}(\cX_{1}\times\cX_{2})\cong(\cJ_{0}\cX_{1})\times(\cJ_{0}\cX_{2})$.
\end{lem}

\begin{proof}
From Proposition \ref{prop:J0-mu-hat},
\begin{align*}
\cJ_{0}(\cX_{1}\times\cX_{2}) & \cong\ulHom(\B\widehat{\mu},\cX_{1}\times\cX_{2})\\
 & \cong\ulHom(\B\widehat{\mu},\cX_{1})\times\ulHom(\B\widehat{\mu},\cX_{2})\\
 & \cong(\cJ_{0}\cX_{1})\times(\cJ_{0}\cX_{2}).
\end{align*}
\end{proof}
From the above lemma, there exist projections
\[
\pr_{i}\colon\cJ_{0}(\cX_{1}\times\cX_{2})\to\cJ_{0}\cX_{i}.
\]
The raising functions $c_{i}$ of $\cX_{i}$ induce raising functions
$\pr_{i}^{*}(c_{i})$ of $\cX_{1}\times\cX_{2}$. 
\begin{defn}
We define a raising function $c_{1}\boxplus c_{2}$ of $\cX_{1}\times\cX_{2}$
to be $\pr_{1}^{*}(c_{1})+\pr_{2}^{*}(c_{2})$. 
\end{defn}

It is easy to see that we have the following equality of functions
on $(\cX_{1}\times\cX_{2})\langle F\rangle$:
\begin{equation}
H_{\cL_{1}\boxtimes\cL_{2},c_{1}\boxplus c_{2}}=\prod_{i=1}^{2}H_{\cL_{i},c_{i}}\circ\pr_{i}.\label{eq:ht prod}
\end{equation}

\begin{lem}
\label{lem:age prod}We have the following equality of functions on
$\pi_{0}(\cJ_{0}(\cX_{1}\times\cX_{2}))$:
\begin{align*}
\age_{\cX_{1}\times\cX_{2}} & =\age_{\cX_{1}}\boxplus\age_{\cX_{2}}.
\end{align*}
\end{lem}

\begin{proof}
For an algebraically closed field $K$, points $p_{i}\in(\cJ_{0}\cX_{i})(K)$
determine $\widehat{\mu}(K)$-representations $\T_{x_{i}}\cX_{i}$,
where $x_{i}\in\cX_{i}(K)$ are points induced by $p_{i}$. The point
$(p_{1},p_{2})\in(\cJ_{0}(\cX_{1}\times\cX_{2}))(K)$ then determines
a $\widehat{\mu}(K)$-representation $\T_{(x_{1},x_{2})}(\cX_{1}\times\cX_{2})$,
which is isomorphic to $\T_{x_{1}}\cX_{1}\oplus\T_{x_{2}}\cX_{2}$.
This shows the lemma. 
\end{proof}

As a consequence of this section, we obtain:
\begin{prop}
For $i\in\{1,2\}$, let $\cX_{i}$ be a Fano stack whose $F$-points
are Zariski dense and let $c_{i}$ be an adequate raising function
of $\cX_{i}$. Suppose that Conjecture \ref{conj:Fano-stack} holds
for pairs $(\cX_{1},c_{1})$ and $(\cX_{2},c_{2})$. Then, $\cX_{1}\times\cX_{2}$
is again a Fano stack whose $F$-points are Zariski dense, $c_{1}\boxplus c_{2}$
is an adequate raising function of $\cX_{1}\times\cX_{2}$, and Conjecture
\ref{conj:Fano-stack} holds for $(\cX_{1}\times\cX_{2},c_{1}\boxplus c_{2})$. 
\end{prop}

\begin{proof}
It is easy to see that $\cX_{1}\times\cX_{2}$ is again a Fano stack
whose $F$-points are Zariski dense. From Lemma \ref{lem:age prod},
we have that $\age_{c_{1}\boxplus c_{2}}=\age_{c_{1}}\boxplus\age_{c_{2}}$.
If $\cY_{1}$ and $\cY_{2}$ are sectors of $\cX_{1}$ and $\cX_{2}$
respectively, then the sector $\cY_{1}\times\cY_{2}$ of $\cX_{1}\times\cX_{2}$
is twisted if and only if either $\cY_{1}$ or $\cY_{2}$ is twisted.
If this is the case, then 
\[
\age_{c_{1}\boxplus c_{2}}(\cY_{1}\times\cY_{2})=\age_{c_{1}}(\cY_{1})+\age_{c_{2}}(\cY_{2})\ge1.
\]
If $\cX_{1}\times\cX_{2}$ is of dimension zero, then $\cX_{1}$ is
also of dimension zero and has a twisted sector $\cY_{1}$ with $c_{1}(\cY_{1})=1$.
Therefore, the twisted sector $\cY_{1}\times\cX_{2}$ of $\cX_{1}\times\cX_{2}$
satisfies 
\[
(c_{1}\boxplus c_{2})(\cY_{1}\times\cX_{2})=c_{1}(\cY_{1})+c_{2}(\cX_{2})=1.
\]
Thus, $c_{1}\boxplus c_{2}$ is adequate. 

To show the last assertion of the proposition, we first claim that
\[
j_{c_{1}\boxplus c_{2}}(\cX_{1}\times\cX_{2})=j_{c_{1}}(\cX_{1})+j_{c_{2}}(\cX_{2}).
\]
Indeed, a $c_{1}\boxplus c_{2}$-junior sector of $\cX_{1}\times\cX_{2}$
is either $\cY_{1}\times\cX_{2}$ with $\cY_{1}$ a $c_{1}$-junior
sector or $\cX_{1}\times\cY_{2}$ with $\cY_{2}$ a $c_{2}$-junior
sector. Moreover, sectors of these forms are only $c_{1}\boxplus c_{2}$-junior
sectors, which shows the claim. As is well-known, we have $\rho(\cX_{1}\times\cX_{2})=\rho(\cX_{1})+\rho(\cX_{2})$. 

Let $T_{i}\subset\cX_{i}\langle F\rangle$, $i\in\{1,2\}$ be thin
subsets as in Conjecture \ref{conj:Fano-stack}. Then, $T_{1}\times\cX_{2}\langle F\rangle$
and $\cX_{1}\langle F\rangle\times T_{2}$ are both thin subsets of
$(\cX_{1}\times\cX_{2})\langle F\rangle$ and we have that 
\[
(\cX_{1}\times\cX_{2})\langle F\rangle\setminus\left((T_{1}\times\cX_{2}\langle F\rangle)\cup(\cX_{1}\langle F\rangle\times T_{2})\right)=(\cX_{1}\langle F\rangle\setminus T_{1})\times(\cX_{2}\langle F\rangle\setminus T_{2}).
\]
From the assumption that Conjecture \ref{conj:Fano-stack} holds for
pairs $(\cX_{1},c_{1})$ and $(\cX_{2},c_{2})$ and from \cite[\S 1,  2.\ Prop.]{franke1989rational},
we have that for some constant $C>0$, 
\begin{align*}
 & \#\{x\in(\cX_{1}\times\cX_{2})\langle F\rangle\setminus\left((T_{1}\times\cX_{2}\langle F\rangle)\cup(\cX_{1}\langle F\rangle\times T_{2})\right)\mid H_{\omega_{\cX_{1}\times\cX_{2}}^{-1},c_{1}\boxplus c_{2}}(x)\le B\}\\
 & =\#\{x\in(\cX_{1}\langle F\rangle\setminus T_{1})\times(\cX_{2}\langle F\rangle\setminus T_{2})\mid H_{\omega_{\cX_{1}}^{-1},c_{1}}(x)\cdot H_{\omega_{\cX_{2}}^{-1},c_{2}}(x)\le B\}\\
 & \sim CB(\log B)^{\rho(\cX_{1})+\rho(\cX_{2})+j_{c_{1}}(\cX_{1})+j_{c_{2}}(\cX_{2})-1}\\
 & =CB(\log B)^{\rho(\cX_{1}\times\cX_{2})+j_{c_{1}\boxplus c_{2}}(\cX_{1})-1}.
\end{align*}
This completes the proof of the proposition. 
\end{proof}
\begin{example}
\label{exa:product}Taking products enables us to construct new examples
from known cases where Conjecture \ref{conj:Fano-stack} holds. For
example, let $\cX_{1}=\cP(a_{0},\dots,a_{n})$ be a weighted projective
stack given with the raising function described in Example \ref{exa:wt-proj-stack}
and let $\cX_{2}=\B\mu_{m,F}$ be the classifying stack of $\mu_{m,F}$
given with an arbitrary adequate raising function $c_{2}$. Then,
the product $\cX_{1}\times\cX_{2}$ is the weighted projective stack
$\cP(ma_{0},\dots,ma_{n})$. Conjecture \ref{conj:Fano-stack} holds
for the Fano stack $\cX_{1}\times\cX_{2}$ given with the raising
function $c_{1}\boxplus c_{2}$. The stack $\cX_{1}\times\cX_{2}$
has at least one $c_{1}\boxplus c_{2}$-junior sector, provided that
$m\ge2$; the $c_{1}\boxplus c_{2}$-junior sectors are the sectors
of the form $\cX_{1}\times\cY$ with $\cY$ a $c_{2}$-junior sector
of $\cX_{2}$. Thus, the function $c_{1}\boxplus c_{2}$ is different
from the raising function of $\cP(ma_{0},\dots,ma_{n})$ described
in Example \ref{exa:wt-proj-stack}. For example, applying this construction
to the moduli stack of elliptic curves, $\overline{\cM_{1,1}}=\cP(4,6)=\cP(2,3)\times\B\mu_{2,F}$,
defines a height function for which the asymptotic formula has the
form $\#\{\cdots\}\sim CB(\log B)$. Note that the notation $\overline{\cM_{1,1}}$
does not mean the coarse moduli space. Under the identifications 
\begin{gather*}
\pi_{0}(\cJ_{0}(\cP(2,3)))=\{0,1/3,1/2,2/3\},\,\pi_{0}(\cJ_{0}(\B\mu_{2,F}))=\{0,1/2\},\\
\pi_{0}(\cJ_{0}(\cP(2,3)\times\B\mu_{2,F}))=\{0,1/3,1/2,2/3\}\times\{0,1/2\}
\end{gather*}
(see Example \ref{exa:wt-proj-stack}), invariants of sectors of $\cP(2,3)$,
$\B\mu_{2,F}$ and $\cP(2,3)\times\B\mu_{2,F}$ are summarized in
Table \ref{table}.

\begin{table}
\begin{centering}
{\footnotesize{}}%
\begin{tabular}{|c|c|c|c|c|}
\hline 
{\footnotesize{}sectors of $\cX_{1}=\cP(2,3)$} & {\footnotesize{}$0$} & {\footnotesize{}$\nicefrac{1}{3}$} & {\footnotesize{}$\nicefrac{1}{2}$} & {\footnotesize{}$\nicefrac{2}{3}$}\tabularnewline
\hline 
\hline 
{\footnotesize{}$c_{1}$} & {\footnotesize{}0} & {\footnotesize{}$\nicefrac{5}{3}$} & {\footnotesize{}$\nicefrac{5}{2}$} & {\footnotesize{}$\nicefrac{10}{3}$}\tabularnewline
\hline 
{\footnotesize{}$\age$} & {\footnotesize{}0} & {\footnotesize{}$\nicefrac{1}{3}$} & {\footnotesize{}$\nicefrac{1}{2}$} & {\footnotesize{}$\nicefrac{2}{3}$}\tabularnewline
\hline 
{\footnotesize{}$\age_{c_{1}}$} & {\footnotesize{}0} & {\footnotesize{}$2$} & {\footnotesize{}$3$} & {\footnotesize{}$4$}\tabularnewline
\hline 
\end{tabular}{\footnotesize{}\qquad{}}%
\begin{tabular}{|c|c|c|}
\hline 
{\footnotesize{}sectors of $\cX_{2}=\B\mu_{2,F}$} & {\footnotesize{}0} & {\footnotesize{}$\nicefrac{1}{2}$}\tabularnewline
\hline 
\hline 
{\footnotesize{}$c_{2}$} & {\footnotesize{}0} & {\footnotesize{}1}\tabularnewline
\hline 
{\footnotesize{}$\age$} & {\footnotesize{}0} & {\footnotesize{}0}\tabularnewline
\hline 
{\footnotesize{}$\age_{c_{2}}$} & {\footnotesize{}0} & {\footnotesize{}1}\tabularnewline
\hline 
\end{tabular}{\footnotesize\par}
\par\end{centering}
{\footnotesize{}\bigskip{}
}{\footnotesize\par}
\begin{centering}
{\footnotesize{}}%
\begin{tabular}{|c|c|c|c|c|c|c|c|c|}
\hline 
{\footnotesize{}sectors of $\cX_{1}\times\cX_{2}$} & {\footnotesize{}$(0,0)$} & {\footnotesize{}$(0,\nicefrac{1}{2})$} & {\footnotesize{}$(\nicefrac{1}{3},0)$} & {\footnotesize{}$(\nicefrac{1}{3},\nicefrac{1}{2})$} & {\footnotesize{}$(\nicefrac{1}{2},0)$} & {\footnotesize{}$(\nicefrac{1}{2},\nicefrac{1}{2})$} & {\footnotesize{}$(\nicefrac{2}{3},0)$} & {\footnotesize{}$(\nicefrac{2}{3},\nicefrac{1}{2})$}\tabularnewline
\hline 
\hline 
{\footnotesize{}$c_{1}\boxplus c_{2}$} & {\footnotesize{}$0$} & {\footnotesize{}$1$} & {\footnotesize{}$\nicefrac{5}{3}$} & {\footnotesize{}$\nicefrac{8}{3}$} & {\footnotesize{}$\nicefrac{5}{2}$} & {\footnotesize{}$\nicefrac{7}{2}$} & {\footnotesize{}$\nicefrac{10}{3}$} & {\footnotesize{}$\nicefrac{13}{3}$}\tabularnewline
\hline 
{\footnotesize{}$\age$} & {\footnotesize{}$0$} & {\footnotesize{}$0$} & {\footnotesize{}$\nicefrac{1}{3}$} & {\footnotesize{}$\nicefrac{1}{3}$} & {\footnotesize{}$\nicefrac{1}{2}$} & {\footnotesize{}$\nicefrac{1}{2}$} & {\footnotesize{}$\nicefrac{2}{3}$} & {\footnotesize{}$\nicefrac{2}{3}$}\tabularnewline
\hline 
{\footnotesize{}$\age_{c_{1}\boxplus c_{2}}$} & {\footnotesize{}$0$} & {\footnotesize{}$1$} & {\footnotesize{}$2$} & {\footnotesize{}$3$} & {\footnotesize{}$3$} & {\footnotesize{}$4$} & {\footnotesize{}$4$} & {\footnotesize{}$5$}\tabularnewline
\hline 
\end{tabular}{\footnotesize{}\bigskip{}
}{\footnotesize\par}
\par\end{centering}
\caption{Sectors and their invariants of $\protect\cP(2,3)$, $\protect\B\mu_{2,F}$
and $\protect\cP(2,3)\times\protect\B\mu_{2,F}$.}
\label{table}
\end{table}
\end{example}

\section{Compatibility with coarse moduli spaces\label{sec:compare-coarse}}

We consider a Fano stack $\cX$ with the trivial raising function
$c\equiv0$. Let $X$ be the coarse moduli space of $\cX$. The function
$H_{\cL,c}$ is identical to the stable height function $H_{\cL}$
considered in Section \ref{sec:Line-bundles}. In particular, it factors
as 
\[
\cX\langle F\rangle\to X(F)\xrightarrow{H_{L}}\RR,
\]
where $L$ is the $\QQ$-line bundle on $X$ corresponding to $\cL$.
We suppose that $c$ is adequate. This means that $\age(\cY)\ge1$
for every twisted sector $\cY$. 
\begin{defn}
\label{def:reflection}Let $\cW$ be a DM stack. Let $w\in\cW(\overline{F})$
and let $\alpha\in\Aut(w)\setminus\{1\}$. We say that $\alpha$ is
a \emph{reflection }if the automorphism of $\T_{w}\cW$ induced by
$\alpha$ is a reflection, that is, the fixed point locus has codimension
one. We say that $\cW$ has no reflection if there is no such pair
$(w,\alpha)$. 
\end{defn}

\begin{lem}
The stack $\cX$ has no reflection.
\end{lem}

\begin{proof}
If $x\in\cX(\overline{F})$ and if $\alpha\in\Aut(x)\setminus\{1\}$
is a reflection, then its action on $\T_{x}\cX$ is represented by
a diagonal matrix $\diag(\zeta,1,\dots,1)$ with $\zeta$ a root of
unity for a suitable basis. This shows that the age of $\alpha$ is
less than 1, which contradicts our assumption of adequacy.
\end{proof}
This lemma shows that the morphism $\cX\to X$ is an isomorphism in
codimension one. In particular, the canonical line bundle $\omega_{\cX}$
of $\cX$ corresponds to the canonical $\QQ$-line bundle of $X$
by the isomorphism $\Pic(\cX)_{\QQ}\cong\Pic(X)_{\QQ}$ (see Section
\ref{sec:Line-bundles}); we denote the canonical $\QQ$-line bundle
of $X$ by $\omega_{X}'$. Let $\cU\subset\cX$ and $U\subset X$
be the largest open substacks such that $\cX\to X$ restricts to an
isomorphism $\cU\to U$. The height function $H_{\omega_{\cX}^{-1}}|_{\cU\langle F\rangle}$
and the height function $H_{(\omega_{X}')^{-1}}|_{U(F)}$ are then
identical through the identification $\cU\langle F\rangle=U(F)$.
We check that the asymptotic formula for $H_{\omega_{\cX}^{-1}}|_{\cU\langle F\rangle}$
expected by Conjecture \ref{conj:Fano-stack} coincides with the one
for $H_{(\omega_{X}')^{-1}}|_{U(F)}$ expected by a version of the
Manin conjecture for singular Fano varieties. 

We now briefly review the terminology on singularities appearing in
the minimal model program, which we will use. For details, we refer
the reader to \cite{kollar2013singularities}. Let $Y$ be a variety
over $F$ with the $\QQ$-Cartier canonical divisor $K_{Y}$. A \emph{divisor
over }$Y$ means a prime divisor on some resolution of $Y$. We identify
two divisors over $Y$ when they determine the same valuation on the
function field $K(Y)$. For a divisor over $Y$ say lying on a resolution
$Z$ of $Y$, the \emph{discrepancy} $a(E;Y)$ is defined to be the
multiplicity of $E$ in the relative canonical divisor $K_{Z/Y}$,
which is a rational number. We say that $Y$ \emph{has only canonical
singularities }if $a(E;Y)\ge0$ for every divisor $E$ over $Y$.
A divisor $E$ over $Y$ is said to be \emph{crepant }if it has zero
discrepancy. When $Y$ has canonical singularities, there are at most
finitely many crepant divisors over $Y$ up to identification explained
above. We denote by $\gamma(Y)$ the number of crepant divisors over
$Y$. 
\begin{rem}
\noindent A divisor over $Y$ may not be geometrically irreducible.
By the base change to $\overline{F}$, a crepant divisor over $Y$
may split into several distinct crepant divisors over $Y_{\overline{F}}$.
Thus, we may have the strict inequality $\gamma(Y_{\overline{F}})>\gamma(Y)$. 
\end{rem}

\begin{lem}
The coarse moduli space $X$ has only canonical singularities.
\end{lem}

\begin{proof}
This follows from the Reid--Shepherd-Barron--Tai criterion \cite[ Th.\ 3.21]{kollar2013singularities}
and the adequacy condition that $\age(\cY)\ge1$.
\end{proof}
\begin{prop}[{cf.~\cite[Prop.\ 4.5]{yasuda2014densities}, \cite[Prop.\ 8.5]{yasuda2015maninsconjecture}}]
\label{prop:crep-jun}We have $\gamma(X)=j_{c}(\cX)$.
\end{prop}

\begin{proof}
We only sketch the proof. Let $W\subset X$ be the singular locus
and let $\cW\subset\cX$ be its preimage, which is nothing but the
locus of stacky points. From \cite[Th.\ 3]{yasuda2006motivic}, we
have $\Sigma(X)_{W}=\Sigma(\cX)_{\cW}$ (these invariants are also
denoted by $\M_{\st}(X)_{W}$ and $\M_{\st}(\cX)_{\cW}$ in the literature).
The left hand side has an expression in terms of a log resolution
and exceptional divisors, while the right hand side has an expression
in terms of sectors. Taking the E-polynomial or Poincaré polynomial
realization of this equality and comparing the coefficients of the
leading terms, we see that the number of geometric crepant divisors
(that is, crepant divisors over $X_{\overline{F}}$) is equal to the
number of geometric junior sectors (that is, junior sectors of $\cX_{\overline{F}}$).
In fact, not only the numbers are equal, but also there is a natural
one-to-one correspondence:
\begin{equation}
\{\text{geometric crepant divisors}\}\leftrightarrow\{\text{geometric junior sectors}\}.\label{eq:corr}
\end{equation}
This follows from the correspondence between arcs on a resolution
of $X$ and twisted arcs of $\cX$ (see \cite{yasuda2006motivic}).
Correspondence \ref{eq:corr} is $\Gamma_{F}$-equivariant. In particular,
both sides have the same number of $\Gamma_{F}$-orbits. This implies
the proposition.
\end{proof}
Let us recall the following version of the Manin conjecture for Fano
varieties with canonical singularities:
\begin{conjecture}[{\cite[Conj.\ 2.3]{yasuda2014densities}, \cite[Conj.\ 5.6]{yasuda2015maninsconjecture}}]
\label{conj:singular-Fano}Let $W$ be a Fano variety over $F$ with
only canonical singularities, that is, a normal projective variety
with only canonical singularities such that the anti-canonical $\QQ$-line
bundle $(\omega_{W}')^{-1}$ is ample. Suppose that $W$ has Zariski
dense $F$-points. Then, there exists a thin subset $T\subset W\langle F\rangle$
such that 
\[
\#\{x\in W\langle F\rangle\setminus T\mid H_{(\omega_{W}')^{-1}}(x)\le B\}\sim CB(\log B)^{\rho(W)+\gamma(W)-1}.
\]
\end{conjecture}

We have the following equivalence of conjectures:
\begin{prop}
Let $\cZ\subset\cX$ be the locus of stacky points given with the
structure of a reduced closed substack and let $Z\subset X$ be the
singular locus. Then, the following conditions are equivalent:
\begin{enumerate}
\item Conjecture \ref{conj:Fano-stack} holds for the pair $(\cX,c)$ with
a thin subset $T$ containing $\cZ\langle F\rangle$.
\item Conjecture \ref{conj:singular-Fano} holds for $X$ with a thin subset
$T$ containing $Z(F)$. 
\end{enumerate}
\end{prop}

\begin{proof}
We first note that $\cZ$ and $Z$ are of codimension $\ge2$ in $\cX$
and $X$, respectively and that $Z$ is the image of $\cZ$. If we
identify $(\cX\setminus\cZ)\langle F\rangle$ and $(X\setminus Z)(F)$,
then the height functions $H_{\omega_{\cX}^{-1},c}|_{(\cX\setminus\cZ)\langle F\rangle}$
and $H_{(\omega_{X}')^{-1}}|_{(X\setminus Z)(F)}$ are the same. If
$T=\cZ\langle F\rangle\sqcup T'$ is a thin subset of $\cX\langle F\rangle$
containing $\cZ\langle F\rangle$, then $T'$ is also a thin subset
of $X(F)$ through the identification $(\cX\setminus\cZ)\langle F\rangle=(X\setminus Z)(F)$.
We also have that $\rho(\cX)=\rho(X)$ and $j_{c}(\cX)=\gamma(X)$.
It follows that if Conjecture \ref{conj:Fano-stack} holds for the
pair $(\cX,c)$ with a thin $T=\cZ\langle F\rangle\sqcup T'$, then
Conjecture \ref{conj:singular-Fano} holds for $X$ with the thin
set $Z(F)\sqcup T'$. Similarly for the opposite implication. 
\end{proof}
\begin{rem}
\label{rem:log-terminal}If a nice stack $\cX$ has trivial generic
stabilizer and has no reflection, and if some twisted sector of it
has age less than one, then singularities of the coarse moduli space
$X$ is not canonical but only log terminal. Observation in \cite[Section 13.3]{yasuda2015maninsconjecture}
suggests that it is more difficult to control the asymptotic of heights
of rational points on such a variety as well as to formulate a Batyrev--Manin-type
conjecture for stacks which is compatible with one for varieties.
This is a reason why we impose the condition of adequacy. 
\end{rem}

\begin{example}
Let $\cX:=\cP(1,1,2)$. This stack has only one twisted sector, which
has age 1. Its coarse moduli space $X=\PP(1,1,2)$, the weighted projective
space, is a toric surface with a unique singular point. Applying a
result in \cite{batyrev1996heightzeta} (see also \cite[Example 5.1.1]{batyrev1998tamagawa})
to the minimal resolution of $X$, we get an asymptotic formula of
the form $\#\{\cdots\}\sim CB(\log B)$ for the height function of
the anti-canonical sheaf $\omega_{X}^{-1}$, which is a genuine line
bundle (rather than a $\QQ$-line bundle), since $X$ is Gorenstein.
It follows that Conjecture \ref{conj:Fano-stack} holds for the stack
$\cX$ with the trivial raising function $c\equiv0$; the exponent
of the log factor in this case is
\[
\rho(\cX)+j_{c}(\cX)-1=1+1-1=1.
\]
Note that the $c$-junior sector of $\cX$ is of dimension zero, while
ones in Example \ref{exa:product} have the same dimension as the product
stack $\cX_{1}\times\cX_{2}$ in question. Combining examples of this
kind with the product construction in Example \ref{exa:product} provides
more cases where Conjecture \ref{conj:Fano-stack} holds.
\end{example}

\section{The orbifold pseudo-effective cone\label{sec:The-orbifold-pseudo-effective}}

Our next goal is to formulate the Batyrev--Manin conjecture for DM
stacks which treats non-Fano stacks and line bundles different from
the anti-canonical line bundle. Recall that the original Batyrev--Manin
conjecture was formulated in terms of the pseudo-effective cone. In
this section, we introduce an ``orbifold version'' of the pseudo-effective
cone. Let $\cX$ be a nice stack over a number field $F$ as before. 
\begin{defn}
\label{def:orb-NS}We define the \emph{orbifold Néron--Severi space
to be}
\[
\N_{\orb}^{1}(\cX)_{\RR}:=\N^{1}(\cX)_{\RR}\oplus\bigoplus_{\cY\in\pi_{0}^{*}(\cJ_{0}\cX)}\RR[\cY].
\]
\end{defn}

This is a finite-dimensional $\RR$-vector space. Since $\cX\ne\Spec F$
from the definition of nice stack, this space has positive dimension;
if $\dim\cX>0$, then $\N^{1}(\cX)_{\RR}\ne0$, and if $\dim\cX=0$,
then $\pi_{0}^{*}(\cJ_{0}\cX)\ne\emptyset$. Each element $\theta\in\N_{\orb}^{1}(\cX)_{\RR}$
is uniquely written as
\[
\theta=\theta_{0}+\sum_{\cY}\theta_{\cY}[\cY]\quad(\theta_{0}\in\N^{1}(\cX)_{\RR},\,\theta_{\cY}\in\RR).
\]

\begin{rem}
It is worth noting that the orbifold Néron--Severi space $\N_{\orb}^{1}(\cX)_{\RR}$
can be regarded as a subspace of the Chen--Ruan orbifold cohomology
$\bigoplus_{i\in\QQ}\H_{\mathrm{CR}}^{i}(\cX(\CC),\RR)$ (for this
notion, see \cite{adem2007orbifolds}), like the usual Néron--Severi
space $\N_{\orb}^{1}(X)_{\RR}$ of a smooth variety $X$ can be regarded
as a subspace of $\H^{2}(X(\CC),\RR)$. 
\end{rem}

\begin{defn}
For a raised line bundle $(\cL,c)$, we define its \emph{numerical
class }to be
\[
[\cL,c]:=[\cL]+\sum_{\cY}c(\cY)[\cY]\in\N_{\orb}^{1}(\cX)_{\RR}.
\]
\end{defn}

For a $\QQ$-factorial projective variety $X$, the pseudo-effective
cone $\PEff(X)\subset\N^{1}(X)_{\RR}$ is by definition the closure
of the cone generated by classes of effective divisors. By \cite{boucksom2013thepseudoeffective},
this cone is dual to the cone of moving curves.\footnote{The paper \cite{boucksom2013thepseudoeffective} proves this for smooth
varieties, and the singular case is easily reduced to the smooth case
by desingularization.}The pseudo-effective cone plays an essential role in formulation of
the Batyrev--Manin conjecture. This would be partly because moving
curves are geometric analogue of ``general'' rational points. We
consider the following generalization of moving curves to stacks:
\begin{defn}
A \emph{stacky curve} on $\cX_{\overline{F}}$ is a representable
$\overline{F}$-morphism $\cC\to\cX_{\overline{F}}$, where $\cC$
is a one-dimensional, irreducible, proper, and smooth DM stack over
$\overline{F}$ which has trivial generic stabilizer. A \emph{covering
family of stacky curves} of $\cX_{\overline{F}}$ is a pair 
\[
(\pi\colon\widetilde{\cC}\to T,\widetilde{f}\colon\widetilde{\cC}\to\cX_{\overline{F}})
\]
of $\overline{F}$-morphisms of DM stacks such that 
\begin{enumerate}
\item $\pi$ is smooth and surjective,
\item $T$ is an integral scheme of finite type over $\overline{F}$,
\item for each point $t\in T(\overline{F})$, the morphism
\[
\widetilde{f}|_{\pi^{-1}(t)}\colon\pi^{-1}(t)\to\cX_{\overline{F}}
\]
is a stacky curve on $\cX_{\overline{F}}$,
\item $\widetilde{f}$ is dominant. 
\end{enumerate}
We often omit $\pi$ and write a covering family of stacky curves
simply as $\widetilde{f}\colon\widetilde{\cC}\to\cX_{\overline{F}}$. 
\end{defn}

Let $f\colon\cC\to\cX_{\overline{F}}$ be a stacky curve and let $p\in\cC\langle\overline{F}\rangle$
be a stacky point. The formal neighborhood of $p$ is written as $[\Spec\overline{F}\tbrats/\mu_{l,\overline{F}}]$
for some positive integer $l$. The morphism $f$ induces a twisted
arc of $\cX_{\overline{F}}$, that is, a representable morphism 
\[
[\Spec\overline{F}\tbrats/\mu_{l,\overline{F}}]\to\cX_{\overline{F}}.
\]
In turn, this induces a twisted 0-jet
\[
\B\mu_{l,\overline{F}}\hookrightarrow[\Spec\overline{F}\tbrats/\mu_{l,\overline{F}}]\to\cX_{\overline{F}},
\]
which is an $\overline{F}$-point of $\cJ_{0}\cX$. We denote the
sector of $\cX$ (not of $\cX_{\overline{F}}$) containing this point
by $\cY_{p}$. Note that this construction is an analogue of construction
of the residue map $\psi_{v}$ (see Definition \ref{def:residue-map}).
As an analogue of our unstable height function (Definition \ref{def:height}),
we define:

\begin{defn}
Let $\theta=\theta_{0}+\sum_{\cY}\theta_{\cY}[\cY]$ be an element
of $\N_{\orb}^{1}(\cX)_{\RR}$ and let $f\colon\cC\to\cX_{\overline{F}}$
be a stacky curve. The \emph{intersection number} $(f,\theta)=(\cC,\theta)$
is defined by
\[
(\cC,\theta):=\deg f^{*}\theta_{0}+\sum_{p\in\cC\langle\overline{F}\rangle:\text{stacky point}}\theta_{\cY_{p}}.
\]
\end{defn}

\begin{lem}
Let $(\pi\colon\widetilde{\cC}\to T,\widetilde{f}\colon\widetilde{\cC}\to\cX_{\overline{F}})$
be a covering family of stacky curves and let $\theta\in\N_{\orb}^{1}(\cX)$.
Then, there exists an open dense subscheme $T_{0}\subset T$ such
that for all $t\in T_{0}(\overline{F})$, the intersection numbers
$(\widetilde{f}|_{\pi^{-1}(t)},\theta)$ are the same.
\end{lem}

\begin{proof}
We let $f$ vary among the morphisms $\widetilde{f}|_{\pi^{-1}(t)}\colon\pi^{-1}(t)\to\cX_{\overline{F}}$,
$t\in T(\overline{F})$. The first term $\deg f^{*}\theta_{0}$ of
the intersection number is constant on the entire $T(\overline{F})$.
We show that the second term $\sum_{p}\theta_{\cY_{p}}$ is also constant
on some open dense subset $T_{0}\subset T$. There exists a dominant
morphism $T'\to T$ of $\overline{F}$-varieties satisfying 
\begin{enumerate}
\item the base change $\pi_{T'}\colon\widetilde{\cC}_{T'}\to T'$ of $\pi\colon\widetilde{\cC}\to T$
has sections $s_{1},\dots,s_{l}\colon T'\to\widetilde{\cC}_{T'}$
such that the union of their images, $\bigcup_{i=1}^{l}\Image(s_{i})$,
is precisely the stacky locus of $\widetilde{\cC}_{T'}$, and
\item each reduced closed substack $\Image(s_{i})\subset\widetilde{\cC}_{T'}$
is isomorphic to $\B\mu_{l_{i},T'}$ over $T'$ for some $l_{i}>0$. 
\end{enumerate}
The morphism 
\[
\B\mu_{l_{i},T'}\hookrightarrow\widetilde{\cC}_{T'}\to\cX_{\overline{F}}
\]
defines a $T'$-point of $\cJ_{0}\cX_{\overline{F}}$ and one of $\cJ_{0}\cX$,
which in turn determines a sector of $\cX$. This shows that when
an $\overline{F}$-point $p\in\widetilde{\cC}_{T'}\langle\overline{F}\rangle$
varies along $\Image(s_{i})$, then the associated sector $\cY_{p}$
is unchanged. This shows that the second term $\sum_{p}\theta_{\cY_{p}}$
of the intersection number is constant, when $t$ varies in the image
of $T'\langle\overline{F}\rangle\to T\langle\overline{F}\rangle$. 
\end{proof}
\begin{defn}
For a covering family $(\pi\colon\widetilde{\cC}\to T,\widetilde{f}\colon\widetilde{\cC}\to\cX_{\overline{F}})$
of stacky curves and $\theta\in\N_{\orb}^{1}(\cX)$, we define the
\emph{intersection number }$(\widetilde{f},\theta)$ to be $(\widetilde{f}|_{\pi^{-1}(t)},\theta)$
for a general point $t\in T(\overline{F})$. 
\end{defn}

\begin{defn}[Orbifold pseudo-effective cone]
We define the \emph{orbifold pseudo-effective cone} $\PEff_{\orb}(\cX)$
to be the cone in $\N_{\orb}^{1}(\cX)$ consisting of elements $\theta$
such that for every covering family $\widetilde{f}\colon\widetilde{\cC}\to\cX_{\overline{F}}$
of stacky curves, $(\widetilde{f},\theta)\ge0$. Elements in $\PEff_{\orb}(\cX)$
are called \emph{pseudo-effective classes. }We also define the (non-orbifold)
\emph{pseudo-effective cone $\PEff(\cX)$ }to be $\PEff(\cX):=\PEff_{\orb}(\cX)\cap\N^{1}(\cX)_{\RR}$.
\end{defn}

We see that by the identification $\N^{1}(\cX)_{\RR}=\N^{1}(\overline{\cX})_{\RR}$,
we have $\PEff(\cX)=\PEff(\overline{\cX})$. 
\begin{rem}[Intersection pairing]
\label{rem:intersetion-pairing}If $\N_{1,\orb}(\cX)_{\RR}$ denotes
the dual space of $\N_{\orb}^{1}(\cX)_{\RR}$, we may define the intersection
number in terms of the pairing 
\[
\N_{\orb}^{1}(\cX)_{\RR}\times\N_{1,\orb}(\cX)_{\RR}\to\RR.
\]
The space $\N_{1,\orb}(\cX)_{\RR}$ is expressed as 
\[
\N_{1,\orb}(\cX)_{\RR}=\N_{1}(\cX)_{\RR}\oplus\bigoplus_{\cY\in\pi_{0}^{*}(\cJ_{0}\cX)}\RR[\cY]^{*},
\]
where $\N_{1}(\cX)_{\RR}$ is the space of real 1-cycles modulo the
numerical equivalence and $\{[\cY]^{*}\}_{\cY}$ is the dual basis
of $\{[\cY]\}_{\cY}$, which is supposed to be orthogonal to $\N^{1}(\cX)_{\RR}$.
For a covering family $\widetilde{f}\colon\widetilde{\cC}\to\cX_{\overline{F}}$
of stacky curves, we can associate an element
\[
[\widetilde{f}]_{\orb}:=[\widetilde{f}|_{\pi^{-1}(t)}]+\sum_{p\in(\pi^{-1}(t))\langle\overline{F}\rangle:\text{stacky point}}[\cY_{p}]^{*}\quad([\widetilde{f}|_{\pi^{-1}(t)}]\in\N_{1}(\cX)_{\RR})
\]
so that for every $\theta\in\N_{\orb}^{1}(\cX)_{\RR}$, the intersection
number $(\widetilde{f},\theta)$ is expressed as the pairing $([\widetilde{f}]_{\orb},\theta)$.
The dual cone of $\PEff_{\orb}(\cX)$ is then the closure of the cone
generated by the classes $[\widetilde{f}]_{\orb}$ of covering families
$\widetilde{f}\colon\widetilde{\cC}\to\cX_{\overline{F}}$ of stacky
curves. 
\end{rem}

\begin{lem}
\label{lem:peff-peff}Let $\theta=\theta_{0}+\sum_{\cY}\theta_{\cY}[\cY]\in\N_{\orb}^{1}(\cX)_{\RR}$.
If $\theta_{0}$ is not pseudo-effective as an element of $\N^{1}(\overline{\cX})$,
then $\theta$ is not pseudo-effective, either.
\end{lem}

\begin{proof}
There exits a finite morphism $W\to\cX$ from a scheme $W$ (see \cite[tag 04V1]{stacksprojectauthors2022stacksproject}).
We may suppose that $W\to\overline{\cX}$ is a Galois cover, say with
Galois group $G$. Then, $W$ is $\QQ$-factorial and we have an isomorphism
\[
(\N^{1}(W)_{\RR})^{G}\cong\N^{1}(\overline{\cX})_{\RR}.
\]
Moreover, this isomorphism preserves the pseudo-effectivity. Thus,
if $\theta_{0}\in\N^{1}(\cX)_{\RR}=\N^{1}(\overline{\cX})_{\RR}$
is not pseudo-effective, then the corresponding $\eta\in(\N^{1}(W)_{\RR})^{G}$
is not pseudo-effective, either. Therefore, there exists a moving
curve $C$ of $W$ with $(C,\eta)<0$. Then, $C\to\cX$ induces a
covering family of stacky curves whose source is a scheme. Since $C$
has no stacky point, $\theta_{\cY}$'s do not contribute to the intersection
number $(C,\theta)$ and we have
\[
(C,\theta)=(C,\eta)<0.
\]
We see that $\theta$ is not pseudo-effective.
\end{proof}
\begin{cor}
\label{cor:PEff inclusions}We have 
\[
\PEff(\cX)+\sum_{\cY}\RR_{\ge0}[\cY]\subset\PEff_{\orb}(\cX)\subset\PEff(\cX)+\sum_{\cY}\RR[\cY].
\]
\end{cor}

\begin{proof}
The left inclusion is obvious. The right inclusion is nothing but
the last lemma. 
\end{proof}
\begin{lem}
\label{lem:cover}Let $K$ be an infinite field and let $n$ be a
positive integer coprime to the characteristic of $K$. Let $C$ be
a geometrically irreducible, smooth, and proper curve over $K$. Let
$p_{0},p_{1},\dots,p_{m}\in C(K)$ be $K$-points of $C$. Then, there
exist a geometrically irreducible, smooth, and proper curve $D$ over
$K$ and a finite morphism $f\colon D\to C$ of degree $n$ such that
for $i\in\{1,\dots,m\}$, $f^{-1}(p_{i})_{\red}\cong\Spec K$ and
$f$ is étale over $p_{0}$. 
\end{lem}

\begin{proof}
We take an affine open subscheme $C'=\Spec R\subset C$ containing
all of $p_{0},\dots,p_{m}$ and suppose that it is embedded in the
affine space $\AA_{K}^{l}$. Since $K$ is an infinite field, for
each $i>0$, there exists a hyperplane $H_{i}=\{f_{i}=0\}\subset\AA_{K}^{l}$
which intersects with $C'$ transversally at $p_{i}$, but does not
meet $p_{j}$ for any $j\ne i$. Let $f\in R$ be the restriction
of $\prod_{i>0}f_{i}$ to $C'$. This function on $C'$ has zeroes
of order one at each of $p_{1},\dots,p_{m}$ and does not vanish at
$p_{0}$. It follows that the finite cover 
\[
D':=\Spec R[x]/(x^{n}-f)\to C'
\]
satisfies the desired property except the properness. We only need
to take projective compactification. 
\end{proof}
\begin{prop}
\label{prop:choosing-sector}Let $\cX$ be a nice stack and let $\cY_{0}$
be a twisted sector of $\cX$. Then, for any positive integer $n$,
there exists a covering family $\widetilde{f}\colon\widetilde{\cC}\to\cX_{\overline{F}}$
of stacky curves such that if we write 
\[
[\widetilde{f}]_{\orb}=[\widetilde{f}]+\sum_{\cY\in\pi_{0}^{*}(\cJ_{0}\cX)}\theta_{\widetilde{f},\cY}[\cY]^{*}\quad(\theta_{\cY}\in\ZZ_{\ge0})
\]
(see Remark \ref{rem:intersetion-pairing}), then
\[
\theta_{\widetilde{f},\cY_{0}}>n\sum_{\cY\ne\cY_{0}}\theta_{\widetilde{f},\cY}.
\]
\end{prop}

\begin{proof}
In this proof, we identify geometric points of a DM stack and ones
of the coarse moduli space. Using a local description of $\cX_{\overline{F}}$
as a quotient stack, we can construct a covering family $(\pi\colon\widetilde{\cC}\to T,\widetilde{f}\colon\widetilde{\cC}\to\cX_{\overline{F}})$
such that $\theta_{\widetilde{f},\cY_{0}}>0$. Let $\widetilde{C}$
denote the coarse moduli space of $\widetilde{\cC}$. Let $L=\overline{K(T)}$
be an algebraic closure of the function field $K(T)$ of $T$. Let
\[
p_{0},\dots,p_{m}\in\widetilde{\cC}_{L}\langle L\rangle=\widetilde{C}_{L}(L)
\]
be stacky points of $\widetilde{\cC}_{L}$. For an integer $n'>nm$,
we take a finite cover $\widetilde{D}_{L}\to\widetilde{C}_{L}$ of
degree $n'$ as in Lemma \ref{lem:cover}. The induced rational map
$\widetilde{D}_{L}\to\widetilde{C}_{L}\dashrightarrow\cX_{\overline{F}}$
uniquely extends to a representable morphism 
\[
\widetilde{g}_{L}\colon\widetilde{\cD}_{L}\to\cX_{\overline{F}},
\]
where $\widetilde{\cD}_{L}$ is a smooth proper DM stack over $L$
with the coarse moduli space morphism $\widetilde{\cD}_{L}\to\widetilde{D}_{L}$.
The only stacky $L$-points of $\widetilde{\cD}_{L}$ are the $n'$
points $r_{0,1},\dots,r_{0,n'}$ lying over $p_{0}$ and possibly
some of the points $r_{1},\dots,r_{m}$ lying over $p_{1},\dots p_{m}$
respectively. If we define sectors of $\cX$ associated to stacky
$L$-points of $\widetilde{\cD}_{L}$ in the same way as we did for
stacky $\overline{F}$-points before, then for every $j\in\{1,\dots,n'\}$,
we have $\cY_{r_{0.j}}=\cY_{0}$. If we define $[\widetilde{g}_{L}]_{\orb}\in\N_{1,\orb}(\cX)$
similarly as before, then we get
\begin{align*}
[\widetilde{g}_{L}]_{\orb} & =[\widetilde{g}_{L}]+\sum_{j=1}^{n'}[\cY_{r_{0,j}}]^{*}+\sum_{i=1}^{m}[\cY_{r_{i}}]^{*}\\
 & =[\widetilde{g}]+n'[\cY_{0}]^{*}+\sum_{i=1}^{m}[\cY_{r_{i}}]^{*}.
\end{align*}
It follows that
\[
\theta_{\widetilde{g}_{L},\cY_{0}}\ge n'>nm\ge n\sum_{\cY\ne\cY_{0}}\theta_{\widetilde{g}_{L},\cY}.
\]
This is basically the desired inequality, except that $\widetilde{\cD}_{L}$
is defined not over an $\overline{F}$-variety, but over $L=\overline{K(T)}$.
By a standard argument, we can find a finitely generated $\overline{F}$-subalgebra
$R\subset L$ such that $\widetilde{\cD}_{L}$ is the base change
of an $R$-stack $\widetilde{\cD}_{R}$ and the morphism $\widetilde{g}_{L}$
is induced from a morphism $\widetilde{\cD}_{R}\to\cX_{\overline{F}}$.
The pair $(\widetilde{\cD}_{R}\to\Spec R,\widetilde{\cD}_{R}\to\cX_{\overline{F}})$
is a covering family of stacky curves that has the desired property. 
\end{proof}
\begin{prop}
\label{prop:negative-coeff}Consider an element $\eta\in\N_{\orb}^{1}(\cX)$
of the form $\sum_{\cY}\eta_{\cY}[\cY]$, $\eta_{\cY}\in\RR$ and
suppose that $\eta_{\cY_{0}}<0$ for some twisted sector $\cY_{0}$.
Then, $\eta\notin\PEff_{\orb}(\cY)$. Namely, 
\[
\PEff_{\orb}(\cX)\cap\bigoplus_{\cY\in\pi_{0}^{*}(\cJ_{0}\cX)}\RR[\cY]=\sum_{\cY\in\pi_{0}^{*}(\cJ_{0}\cX)}\RR_{\ge0}[\cY].
\]
\end{prop}

\begin{proof}
For each $n>0$, we take a covering family $\widetilde{f}_{n}\colon\widetilde{\cC}_{n}\to\cX_{\overline{F}}$
of stacky curves as in Proposition \ref{prop:choosing-sector}. Then,
\begin{align*}
(\widetilde{f}_{n},\eta) & =\sum_{\cY}\theta_{\widetilde{f}_{n},\cY}\cdot\eta_{\cY}\\
 & =\theta_{\widetilde{f}_{n},\cY_{0}}\cdot\eta_{\cY_{0}}+\sum_{\cY\ne\cY_{0}}\theta_{\widetilde{f}_{n},\cY}\cdot\eta_{\cY}\\
 & \le\theta_{\widetilde{f}_{n},\cY_{0}}\cdot\eta_{\cY_{0}}+\left(\sum_{\cY\ne\cY_{0}}\theta_{\widetilde{f}_{n},\cY}\right)\max\{\eta_{\cY}\mid\cY\ne\cY_{0}\}.
\end{align*}
For sufficiently large $n$, the last expression is negative, which
implies that $\eta\notin\PEff_{\orb}(\cY)$. The equality of the proposition
follows from Corollary \ref{cor:PEff inclusions}.
\end{proof}
\begin{defn}\label{def:big}
We say that a raised line bundle $(\cL,c)$ on $\cX$ is \emph{big
}if its numerical class $[\cL,c]$ lies in the interior of $\PEff_{\orb}(\cX)$. 
\end{defn}

\begin{lem}
Let $(\cL,c)$ be a raised line bundle on $\cX$. Suppose that $\cX$
has positive dimension. Then, the following conditions are equivalent:
\begin{enumerate}
\item $\cL$ is big, that is, $[\cL]$ is in the interior of the cone $\PEff(\cX)$.
\item $[\cL,c]$ is big. 
\end{enumerate}
\end{lem}

\begin{proof}
From Corollary \ref{cor:PEff inclusions}, the second condition implies
the first one. We now assume the first condition and prove the second
condition. For $\alpha\in\PEff_{\orb}(\cX)^{\vee}$, we have
\[
(\alpha,[\cL,c])\ge(\alpha,[\cL])>0.
\]
Corollary \ref{cor:PEff inclusions} implies
\[
\PEff_{\orb}(\cX)^{\vee}\subset\PEff(\cX)^{\vee}+\sum_{\cY}\RR_{\ge0}[\cY]^{*}.
\]
It follows that there exists a hyperplane $H\subset\N_{1,\orb}(\cX)_{\RR}$
such that $H$ does not contain the origin and the intersection $\PEff_{\orb}(\cX)^{\vee}\cap H$
is a non-empty compact set. From the extreme value theorem, the function
given by the intersection pairing with $[\cL,c]$,
\[
\PEff_{\orb}(\cX)^{\vee}\cap H\to\RR_{>0},\,\alpha\mapsto(\alpha,[\cL,c]),
\]
has a positive minimum value. Perturbing $[\cL,c]$ a little in $\N_{\orb}^{1}(\cX)_{\RR}$
in an arbitrary direction does not break the positivity of the minimum
value of the associated function $\PEff_{\orb}(\cX)^{\vee}\cap H\to\RR$.
This means that $[\cL,c]$ is not pushed out from $\PEff_{\orb}(\cX)$
by a small perturbation in an arbitrary direction. Namely, $[\cL,c]$
is in the interior of $\PEff_{\orb}(\cX)$.
\end{proof}

\section{$a$- and $b$-invariants; the stacky Batyrev--Manin conjecture\label{sec:a-and-b}}

\subsection{$a$- and $b$-invariants and breaking thin subsets}

We keep denoting by $\cX$ a nice stack over $F$.

\begin{defn}
We define the\emph{ orbifold canonical class }of $\cX$ to be\emph{
}
\[
[K_{\cX,\orb}]:=[\omega_{\cX}]+\sum_{\cY}(\age(\cY)-1)[\cY]\in\N_{\orb}^{1}(\cX)_{\RR}.
\]
Here $\omega_{\cX}$ denotes the canonical line bundle of $\cX$.
\end{defn}

\begin{rem}
Generalizing Proposition \ref{prop:crep-jun}, we can show that, if
$\cX$ has trivial generic stabilizer, some twisted sectors $\cY$
of $\cX$ correspond to divisors $E$ over the coarse moduli space
$\overline{\cX}$. Then, for $\cY$ and $E$ corresponding to each
other, $\age(\cY)-1$ is equal to the discrepancy of $E$. This explains
the reason why the coefficient $\age(\cY)-1$ in the above definition
is natural.
\end{rem}

Now we are ready to define the $a$-invariant in the context of stacks:
\begin{defn}
We define the $a$\emph{-invariant }of a big raised line bundle $(\cL,c)$,
denoted by $a(\cL,c)$, to be the unique real number $a$ such that
$a[\cL,c]+[K_{\cX,\orb}]$ lies on the boundary of $\PEff_{\orb}(\cX)$. 
\end{defn}

As for a (non-raised) big line bundle $\cL$ on $\cX$, we define
the $a$-invariant $a(\cL)$ in a similar way by using the (non-orbifold)
pseudo-effective cone $\PEff(\cX)=\PEff(\overline{\cX})$ in $\N^{1}(\cX)_{\RR}=\N^{1}(\overline{\cX})_{\RR}$
and the canonical class $[\omega_{\cX}]$. Note that $a(\cL)$ may
be different from $a(L)$ of the corresponding $\QQ$-line bundle
$L$ on $\overline{\cX}$. This is because $[\omega_{\cX}]$ may be
different from $[\omega_{\overline{\cX}}']$, unless the morphism
$\cX\to\overline{\cX}$ is étale in codimension one. 

We generalize the notion of adequacy for a raising function of a Fano
stack as follows:
\begin{defn}
\label{def:adequacy-2}Let $(\cL,c)$ be a big raised line bundle.
We say that $(\cL,c)$ is \emph{adequate }if 
\begin{enumerate}
\item $\age_{c}(\cY)\ge1$ for every twisted sector $\cY$, 
\item if $\dim\cX>0$, then $a(\cL)=1$, and
\item if $\dim\cX=0$, then $\min\{c(\cY)\mid\cY\text{ twisted sector}\}=1$.
\end{enumerate}
When $\cX$ is a Fano stack, the raised line bundle $(\omega_{\cX}^{-1},c)$
is adequate if and only if $c$ is adequate. 
\end{defn}

\begin{rem}
\label{rem:adequate-normalization}The condition $a(\cL)=1$ may be
viewed as a normalization condition. Indeed, for $r\in\QQ$, $a(\cL^{r})=a(\cL)/r$. 
\end{rem}

\begin{prop}
If $(\cL,c)$ is adequate, then $a(\cL,c)=1$.
\end{prop}

\begin{proof}
The zero-dimensional case follows from Proposition \ref{prop:Malle-1}
proved later. We prove the case $\dim\cX>0$ here. For $a'<1$, 
\[
a'[\cL,c]+K_{\cX,\orb}=(a'[\cL]+[\omega_{\cX}])+\sum_{\cY}(a'\cdot c(\cY)+\age(\cY)-1)[\cY].
\]
Since $a(\cL)=1>a'$, we have that $a'[\cL]+[\omega_{\cX}]$ is not
pseudo-effective. From Lemma \ref{lem:peff-peff}, $a'[\cL,c]+[K_{\cX,\orb}]\notin\PEff_{\orb}(\cX)$.
This shows that $a(\cL,c)\ge1$. On the other hand, 
\[
[\cL,c]+[K_{\cX,\orb}]=([\cL]+[\omega_{\cX}])+\sum_{\cY}(\age_{c}(\cY)-1)[\cY].
\]
Since $[\cL]+[\omega_{\cX}]=a(\cL)[\cL]+[\omega_{\cX}]$ is pseudo-effective
and $\age_{c}(\cY)-1\ge0$, from Corollary \ref{cor:PEff inclusions},
$[\cL,c]+[K_{\cX,\orb}]$ is pseudo-effective. This shows that $a(\cL,c)\le1$.
\end{proof}
\begin{defn}
We define the $b$\emph{-invariant} $b(\cL,c)$ of a big raised line
bundle $(\cL,c)$ to be the codimension of the minimal face of $\PEff_{\orb}(\cX)$
containing $a(\cL,c)[\cL,c]+[K_{\cX,\orb}]$, that is, the dimension
of the following face of the dual cone $\PEff_{\orb}(\cX)^{\vee}\subset\N_{1,\orb}(\cX)_{\RR}$:
\[
\PEff_{\orb}(\cX)^{\vee}\cap\left(a(\cL,c)[\cL,c]+[K_{\cX,\orb}]\right)^{\perp}.
\]
 
\end{defn}

Recall that for a morphism $f\colon\cY\to\cX$ of nice stacks, we
have a natural morphism $\cJ_{0}\cY\to\cJ_{0}\cX$, and hence a raising
function $c$ of $\cX$ induces a raising function $f^{*}c$ of $\cY$.
\begin{defn}
Let us fix a big raised line bundle $(\cL,c)$ of $\cX$. A thin morphism
$f\colon\cY\to\cX$ of nice stacks is called a \emph{breaking thin
morphism} (resp.~a\emph{ weakly breaking thin morphism}) if the raised
line bundle $(f^{*}\cL,f^{*}c)$ is big and if 
\begin{gather*}
(a(f^{*}\cL,f^{*}c),b(f^{*}\cL,f^{*}c))>(a(\cL,c),b(\cL,c))\\
(\text{resp.\,}(a(f^{*}\cL,f^{*}c),b(f^{*}\cL,f^{*}c))\ge(a(\cL,c),b(\cL,c)))
\end{gather*}
in the lexicographic order. A \emph{(resp.~weakly) breaking thin
subset} of $\cX\langle F\rangle$ means a nonempty subset of $\cX\langle F\rangle$
which is the image of the map $\cY\langle F\rangle\to\cX\langle F\rangle$
associated to a (resp.~weakly) breaking thin morphism $\cY\to\cX$. 
\end{defn}

\begin{rem}
The breaking thin morphism is basically the same as the breaking thin
map in \cite{lehmann2019geometric} except that we work with stacks
and put the representability condition. If we did not put the representability
condition, then a morphism of the form $\cX\times\B G\to\cX$ for
a finite group scheme $G$ would become a breaking thin morphism.
We would not like to include such a map, since it induces a surjection
of $F$-point sets.
\end{rem}

\subsection{The stacky Batyrev--Manin conjecture}

We now formulate the first version of the Batyrev--Manin conjecture
for DM stacks as follows: 
\begin{conjecture}[The stacky Batyrev--Manin conjecture I]
\label{conj:general}  Let $\cX$ be a nice stack over $F$ whose coarse moduli space is geometrically rationally connected. Let $(\cL,c)$
be a raised line bundle which is big and adequate. Suppose that $\cX\langle F\rangle$
is Zariski dense. Then the union $T$ of breaking thin subsets of
$\cX\langle F\rangle$ is a thin subset. Moreover, there exists a
constant $C>0$ such that
\begin{align*}
\#\{x\in\cX\langle F\rangle\setminus T\mid H_{\cL,c}(x)\le B\} & \sim CB(\log B)^{b(\cL,c)-1}\quad(B\to\infty).
\end{align*}
\end{conjecture}

\begin{rem}
 In the conjecture, we impose the geometric rational connectedness on the coarse moduli space rather than the more restrictive (and more common) Fano condition, following the version of the Batyrev--Manin conjecture in \cite{Lehmann-Tanimoto-2019-RiMS}. 
 Since we would like to allow singular coarse moduli spaces, it would be more suitable to use a condition which is invariant under birational equivalences like the geometric rational connectedness.
\end{rem}

\begin{rem}
If $\cX$ is a smooth variety, then this conjecture is a version of
the Batyrev--Manin conjecture for a smooth variety and a big line
bundle $L$ with $a(L)=1$ such that the removed accumulation subset
is a thin subset. When the $a$-invariant is not equal to 1 but positive,
then we can reduce it to the case with $a=1$ by considering the $\RR$-line
bundle $L^{1/a(L)}$ and generalizing the above conjecture to $\RR$-line
bundles in the obvious way. 
\end{rem}

\begin{rem}
\label{rem:le Rudulier}One may be tempted to remove also\emph{ weakly}
breaking thin subsets. An evidence for this idea in the case of varieties
was provided by a work of Le Rudulier \cite{lerudulier2014pointsalgebriques}.
She showed that for some algebraic surface, the leading constant becomes
the same as the one conjectured by Peyre \cite{peyre1995hauteurs}
only after removing a weakly breaking thin subset. However, if we
remove all weakly breaking thin subsets, then it may happen that all
$F$-points are removed, and no point is left to count. See Remark
\ref{rem:remove-weak}. 
\end{rem}

\begin{rem}
Checking the first assertion of Conjecture \ref{conj:general}, that
the union of breaking thin subsets is a thin subset, is an interesting
problem on its own. When the target $\cX$ is a geometrically uniruled
variety, this problem was affirmatively solved by Lehmann--Sengupta--Tanimoto
\cite{lehmann2019geometric}. The problem has an affirmative answer
also when $\cX$ is $\B G$ for a commutative group scheme $G$ (see
Corollary \ref{cor:no-breaking}) as well as for some constant group
scheme $G$ (see Proposition \ref{prop:comprehensive}). 
\end{rem}

We also formulate a variant of Conjecture \ref{conj:general}, incorporating
the following notions.
\begin{defn}
A subset $U\subset\cX\langle F\rangle$ is said to be \emph{cothin
}if its complement $\cX\langle F\rangle\setminus U$ is thin. When
a raised line bundle $(\cL,c)$ on $\cX$ is fixed, we say that an
element of $\cX\langle F\rangle$ is \emph{secure }(resp.~\emph{strongly
secure}) if it is not contained in any breaking thin subset (resp.~any
weakly breaking thin subset) of $\cX\langle F\rangle$. We say that
a subset of $\cX\langle F\rangle$ is \emph{secure }(resp.~\emph{strongly
secure}) if it contains only secure (resp.~strongly secure) elements. 
\end{defn}

The following is the second version of the Batyrev--Manin conjecture
for DM stacks, which allows some freedom in the choice of the set
of counted $F$-points.
\begin{conjecture}[The stacky Batyrev--Manin conjecture II]
\label{conj:general-secure}Let $\cX$ be a nice stack over $F$ whose coarse moduli space is geometrically rationally connected.
Let $(\cL,c)$ be a raised line bundle which is big and adequate.
Suppose that $\cX\langle F\rangle$ is Zariski dense. Let $U\subset\cX\langle F\rangle$
be a secure cothin subset. Then, there exists a constant $C>0$ such
that
\begin{align*}
\#\{x\in U\mid H_{\cL,c}(x)\le B\} & \sim CB(\log B)^{b(\cL,c)-1}\quad(B\to\infty).
\end{align*}
\end{conjecture}

\begin{rem}[Speculation on a formula for the leading constant $C$]
\label{rem:speculation-leading-C}In the situation of Conjecture
\ref{conj:general-secure}, we speculate that if $U$ is also \emph{strongly}
secure and if we use the height function $H_{\cL,c_{*}}$ for a strictly
raised line bundle $(\cL,c_{*})$ given with an adelic metric on $\cL$,
then the constant $C$ may admit an explicit expression similar to
ones of Peyre \cite{peyre1995hauteurs} and Bhargava \cite{bhargava2007massformulae}.
Note that there are nice stacks $\cX$ without any strongly secure
$F$-point. See Remark \ref{rem:Wood-fairness} and Section \ref{prop:comprehensive}
for related discussion in the case of zero-dimensional stacks. 
\end{rem}

\subsection{Fano stacks revisited}
\begin{prop}
Let $\cX$ be a Fano stack and let $c$ be a raising function of $\cX$
such that $(\omega_{\cX}^{-1},c)$ is adequate. Then, we have
\[
b(\omega_{\cX}^{-1},c)=\rho(\cX)+j_{c}(\cX).
\]
In particular, Conjecture \ref{conj:general} implies Conjecture \ref{conj:Fano-stack}. 
\end{prop}

\begin{proof}
We have
\begin{align*}
\eta & :=a(\omega_{\cX}^{-1},c)[\omega_{\cX}^{-1},c]+[K_{\cX,\orb}]\\
 & =[\omega_{\cX}^{-1},c]+[K_{\cX,\orb}]\\
 & =\sum_{\cY}(\age_{c}(\cY)-1)[\cY].
\end{align*}
For a vector 
\[
w\in\bigoplus_{\cY:\text{ not \ensuremath{c}-junior}}\RR[\cY]=:W
\]
and for $0<\epsilon\ll1$, $\eta\pm\epsilon w\in\PEff_{\orb}(\cX)$.
On the other hand, if $w$ is chosen outside $W$, then from Proposition
\ref{prop:negative-coeff} and Lemma \ref{lem:peff-peff}, either
of $\eta\pm\epsilon w$ is not contained in $\PEff_{\orb}(\cX)$.
This shows that the minimal face containing $\eta$ has dimension
equal to $\dim W$. We have proved the first assertion of the proposition.

To show the second assertion, we only need to check that the coarse moduli space $X$ of $\cX$ is geometrically rationally connected. Let $D$ be the branch divisor of the coarse moduli space morphism $\pi \colon \cX\to X$ given with standard $\QQ$-coefficients. Then, the $\QQ$-divisor $K_X+D$ corresponds to the canonical  line bundle $\omega_\cX$ of $\cX$ and the pair $(X,D)$ has only Kawamata log terminal singularities. From \cite[Th.~1]{Zhang-rationally-connected}, $X$ is geometrically rationally connected. 
\end{proof}
\begin{rem}
\label{rem:Gal-mod-jun-1}To give a conjectural expression of the
leading constant $C$ in the asymptotic formula for rational points
of a Fano variety $X$, Peyre \cite{peyre1995hauteurs} uses the three
ingredients; the volume of $\overline{X(F)}\subset X(\AA)$, the volume
of the intersection of $\PEff(X)$ and the hyperplane $(-K_{X},*)=1$,
and the $\Gamma_{F}$-module $\N^{1}(X_{\overline{F}})$. We may speculate
upon how to generalize it to stacks. For a nice stack $\cX$, as soon
as we find the correct topology and measure on $\cX\langle\AA\rangle$,
we can define the volume of $\overline{\cX\langle F\rangle}\subset X\langle\AA\rangle$.
The pseudo-effective cone should be replaced with the orbifold pseudo-effective
cone $\PEff_{\orb}(\cX)$. The volume of the hyperplane section $\PEff_{\orb}(\cX)\cap\{([K_{\cX,\orb}],-0)=1\}$
for an appropriate measure would be the counterpart of the second
ingredient. As for the $\Gamma_{F}$-module, we may consider 
\[
\N_{\orb,\cjun}^{1}(\cX_{\overline{F}})_{\RR}:=\N^{1}(\cX_{\overline{F}})_{\RR}\oplus\bigoplus_{\cY:\text{\ensuremath{c}-junior sector}}\RR[\cY].
\]
From Proposition \ref{prop:crep-jun}, if $\cX$ has trivial generic
stabilizer and if the raising function $c\equiv0$ is adequate, then
this module is isomorphic to the $\Gamma_{F}$-module $\N^{1}(X_{\overline{F}})_{\RR}\oplus\bigoplus_{E}\RR[E]$,
where $E$ runs over crepant divisors over the coarse moduli space
$X_{\overline{F}}$. Thanks to Proposition \ref{prop:crep-jun}, this
Galois module is an analogue of $\cL$-Picard group considered by
Batyrev and Tschinkel \cite{batyrev1998tamagawa}. We do not know
whether the literal translation of Peyre's constant by using the above
orbifold versions of ingredients gives the correct value of the leading
constant, however expect that the above speculation is headed in the
right direction.
\end{rem}

\subsection{Malle's conjecture revisited\label{subsec:Malle-revisited}}

In this subsection, we suppose that $\cX$ over $F$ has dimension
zero and has at least one $F$-point. From \cite[tag 06QK]{stacksprojectauthors2022stacksproject},
this is equivalent to saying that $\cX$ is a neutral gerbe over $F$.
From \cite[(3.21)]{laumon2000champsalgebriques}, for an $F$-point
$x\in\cX(F)$, there exists a canonical isomorphism $\cX\cong\B\ulAut_{F}(x)$.
In particular, $\cX$ is isomorphic to the classifying stack $\B G$
of a finite group scheme $G$ over $F$. Note that the isomorphism
class of such a group scheme $G$ is not generally uniquely determined
by $\cX$, as the next result shows.
\begin{lem}[{\cite[Prop.\ 2.2.3.6]{calm`es2015groupes}, \cite[p.\ 127]{emsalem2017twisting}}]
\label{lem:Aut-twist}Let $G$ be a finite group scheme over $F$,
let $x\in(\B G)(F)$ and let $x'\in(\B\ulAut(G^{\op}))(F)$ be the
$\ulAut(G^{\op})$-torsor derived from $x$ and the conjugation morphism
$G\to\ulAut(G)=\ulAut(G^{\op})$. Then, $\ulAut_{F}(x)$ is isomorphic
to the twisted form of $G^{\op}$ associated to $x'$.
\end{lem}

\begin{cor}
\label{cor:comm-aut}If $G$ is commutative, then for every $x\in(\B G)(F)$,
we have $\ulAut_{F}(x)\cong G$.
\end{cor}

\begin{proof}
This follows from the last lemma and the fact that in the commutative
case, the conjugation map $G\to\ulAut(G)$ is the trivial map onto
the identity point. 
\end{proof}
Our principal interest is in the stack $\cX=\B G_{F}$, where $G$
is a finite group and $G_{F}$ is the corresponding constant group
scheme over $F$. In this case, from Example \ref{exa:BG tw sectors},
we may identify $\pi_{0}(\cJ_{0}\cX)$ with the set $\FConj(G)$ of
$F$-conjugacy classes of $G$. Thus, we may write
\begin{align*}
\N_{\orb}^{1}(\cX) & =\bigoplus_{[1]\ne[g]\in\FConj(G)}\RR[g].
\end{align*}

\begin{prop}
\label{prop:BG-Peff}Suppose that $\cX=\B G_{F}$ for a finite group
$G$. Suppose that $F$ contains $\#G$-th roots of unity. Then, 
\[
\PEff_{\orb}(\cX)=\sum_{[1]\ne[g]\in\Conj(G)}\RR_{\ge0}[g].
\]
\end{prop}

\begin{proof}
We first note that from the assumption, we have 
\begin{align*}
\N_{\orb}^{1}(\cX) & =\bigoplus_{[1]\ne[g]\in\Conj(G)}\RR[g].
\end{align*}
We claim that for any twisted sector $\cY$ of $\cX_{\overline{F}}$,
there exists a stacky curve $f\colon\cC\to\cX_{\overline{F}}$ such
that for every stacky point $p\in\cC(\overline{F})$, the associated
sector $\cY_{p}$ is $\cY$. To show this, suppose that $\cY$ corresponds
to an injection $\iota\colon\mu_{l}\hookrightarrow G$. Let $f:=\prod_{i=0}^{l-1}(x-i)$
and let $L:=\overline{F}(x)(f^{1/l})$ be the $l$-cyclic extension
of $\overline{F}(x)$ associated to $f$. Let $C\to\PP_{\overline{F}}^{1}$
be the associated $\mu_{l}$-cover. This is ramified exactly at $x=0,1,\dots,l-1$.
Note that the cover is unramified over $\infty\in\PP_{\overline{F}}^{1}$,
since $f$ has degree $l$. The morphism $C\to\Spec\overline{F}$
is $\iota$-equivariant and induces 
\[
f\colon\cC:=[C/\mu_{l}]\to\cX_{\overline{F}}=[\Spec\overline{F}/G].
\]
From construction, stacky points of $\cC$ are the images of $0,1,\dots,l-1\in\PP_{\overline{F}}^{1}$.
Moreover, the twisted sector corresponding to each stacky point is
$\cY$. This proves the claim. 

For a twisted sector $\cY_{0}$, we choose $f_{0}\colon\cC\to\cX_{\overline{F}}$
as above. For $\theta=\sum\theta_{\cY}[\cY]\in\N_{\orb}^{1}(\cX)$,
\[
(f_{0},\theta)=b\theta_{\cY_{0}}\quad(b>0).
\]
Note that $f_{0}$ is a covering family of stacky curves of $\cX_{\overline{F}}$.
Thus, if $\theta\in\PEff_{\orb}(\cX)$, then $\theta_{\cY}\ge0$ for
every twisted sector $\cY$. The converse is easy to show. 
\end{proof}
\begin{cor}
\label{cor:PEff-zerodim}We have
\[
\PEff_{\orb}(\cX)=\sum_{\cY\in\pi_{0}^{*}(\cJ_{0}\cX)}\RR_{\ge0}[\cY].
\]
In particular, if $\cX=\B G_{F}$ for a finite group $G$, then
\[
\PEff_{\orb}(\cX)=\sum_{[1]\ne[g]\in\FConj(G)}\RR_{\ge0}[g].
\]
\end{cor}

\begin{proof}
There exists a finite Galois extension $K/F$ such that $\cX_{K}\cong\B G_{K}$
for a finite group $G$ and $K$ contains $\#G$-th roots of unity.
From the last proposition, 
\[
\PEff_{\orb}(\cX_{K})=\sum_{\cY'\in\pi_{0}^{*}(\cJ_{0}\cX_{K})}\RR_{\ge0}[\cY'],
\]
We have an isomorphism $\N_{\orb}^{1}(\cX)\cong(\N_{\orb}^{1}(\cX_{K}))^{\Gal(K/F)}$
such that for a sector $\cY$ of $\cX$, the class $[\cY]$ corresponds
to the sum $\sum_{i=1}^{m}[\cY_{i}']$, where $\{\cY_{1}',\dots,\cY_{m}'\}\subset\pi_{0}(\cJ_{0}\cX_{K})$
is the Galois orbit corresponding to $\cY$. By the isomorphism, $\PEff_{\orb}(\cX)$
corresponds to 
\[
\PEff_{\orb}(\cX_{K})\cap(\N_{\orb}^{1}(\cX_{K}))^{\Gal(K/F)}.
\]
The last cone is generated by such sums $\sum_{i=1}^{m}[\cY_{i}']$
as above. This shows the corollary.
\end{proof}
As in Section \ref{sec:Fano-stacks}, we only consider the structure
sheaf $\cO$ as a line bundle on $\cX$. Corollary \ref{cor:PEff-zerodim}
implies:
\begin{cor}
A raised line bundle $(\cO,c)$ on $\cX$ is big if and only if $c$
is positive. 
\end{cor}

\begin{prop}
\label{prop:Malle-1}For a positive raising function $c$ of $\cX$,
we have that
\begin{align*}
a(\cO,c) & =\max\left\{ c(\cY)^{-1}\mid\cY\in\pi_{0}^{*}(\cJ_{0}\cX)\right\} \\
 & =\left(\min\left\{ c(\cY)\mid\cY\in\pi_{0}^{*}(\cJ_{0}\cX)\right\} \right)^{-1}
\end{align*}
and 
\[
b(\cO,c)=\sharp\{\cY\in\pi_{0}^{*}(\cJ_{0}\cX)\mid c(\cY)=a(\cO,c)^{-1}\}.
\]
\end{prop}

\begin{proof}
For a real number $a\in\RR$, we have
\[
v_{a}:=a[\cO,c]+[K_{\cX,\orb}]=\sum_{\cY\in\pi_{0}^{*}(\cJ_{0}\cX)}(a\cdot c(\cY)-1)[\cY].
\]
Therefore, $v_{a}\in\PEff_{\orb}(\cX)$ if and only if $a\cdot c(\cY)-1\ge0$
for every twisted sector $\cY$. The minimum real number $a$ satisfying
the last condition is exactly the value of $a(\cO,c)$ stated in the
proposition. 

The minimal face of $\PEff_{\orb}(\cX)$ containing $a(\cO,c)[\cO,c]+[K_{\cX,\orb}]$
is the cone generated by the classes $[\cY]$ of those twisted sectors
that \emph{do not} have the minimal $c$-value. Thus, the codimension
of this face is equal to the number of those twisted sectors that
\emph{do} have the minimal $c$-value. The desired formula for the
$b$-invariant follows.
\end{proof}
\begin{rem}[Malle's conjecture revisited]
\label{rem:Malle-revisited}If $\cX=\B G_{F}$ for a finite group
$G$, then $\pi_{0}(\cJ_{0}\cX)=\FConj(G)$ as was already mentioned.
In this situation, Conjecture \ref{conj:general} with $a$- and $b$-invariants
computed in Proposition \ref{prop:Malle-1} is the same as the version
of Malle's conjecture for the generalized discriminant (see Example
\ref{exa:discriminant}) associated to the given raising function
$c$, except that we count $G$-torsors over $F$ giving points of
$\cX\langle F\rangle\setminus T$, while the original conjecture of Malle 
concerns $G$-fields over $F$ (corresponding to connected $G$-torsors
over $F$). Counting only $G$-fields amounts to removing the thin
subset $T'$ given by
\[
T'=\bigcup_{H\subsetneq G}((\B H_{F})\langle F\rangle\to(\B G_{F})\langle F\rangle).
\]
Points coming from breaking thin morphisms $\cY\to\B G_{F}$ (see
Section \ref{sub:Kl=0000FCners'-counterexample} for an example) is
expected to give an asymptotic larger than the expectation of 
Malle's conjecture and should be removed. The same remark shows that
$T$ may include some $G$-fields. 
\end{rem}

\begin{rem}
\label{rem:thin-BH}Let $f\colon\cY\to\cX$ be a thin morphism with
$\cY\langle F\rangle\ne\emptyset$ and let $y\in\cY(F)$. Then, we
have the injection $\ulAut_{F}(y)\to\ulAut_{F}(f(y))$ of group schemes
over $F$. Through the canonical isomorphisms $\cY\cong\B\ulAut_{F}(y)$
and $\cX\cong\B\ulAut_{F}(f(y))$, the morphism $f\colon\cY\to\cX$
is identified with the natural morphism $\B\ulAut_{F}(y)\to\B\ulAut_{F}(f(y))$.
\end{rem}

The following proposition enables us to check the secureness of each
point $x\in\cX\langle F\rangle$ by looking only at (necessarily finitely
many) subgroup schemes of $\ulAut_{F}(x)$. 
\begin{prop}
\label{prop:elementwise-criterion}A point $x\in\cX\langle F\rangle$
is secure (resp.~strongly secure) if and only if for any proper subgroup
scheme $H$ of $\ulAut_{F}(x)$, the induced morphism $\B H\to\B\ulAut_{F}(x)\cong\cX$
is not breaking thin morphism (resp.~weakly breaking thin morphism). 
\end{prop}

\begin{proof}
This result also appears in \cite{darda2022torsors}, but stated and
proved in a slightly different language. For the sake of completeness,
we give a proof here by using the language of the present article.
If there exists a proper subgroup scheme $H\subsetneq\ulAut_{F}(x)$
such that $\B H\to\B\ulAut_{F}(x)\cong\cX$ is a breaking thin morphism,
then by definition, $x$ is not secure. Conversely, if $x$ is not
secure, then there exist a breaking thin morphism $f\colon\cY\to\cX$
and an $F$-point $y\in\cY\langle F\rangle$ with $f(y)=x$. From
Remark \ref{rem:thin-BH}, $f$ is identified with $\B H\to\B\ulAut_{F}(x)$
for some proper subgroup scheme $H\subsetneq\ulAut_{F}(x)$. We have
proved the assertion about the condition for the secureness. The one
for the strong secureness is similarly proved.
\end{proof}
\begin{cor}
\label{cor:no-breaking}If $G$ is commutative, then $(\B G_{F})\langle F\rangle$
has no breaking thin subset. Namely, $(\B G_{F})\langle F\rangle$
is a secure cothin subset of itself.
\end{cor}

\begin{proof}
From Corollary \ref{cor:comm-aut} and Remark \ref{rem:thin-BH},
every thin morphism $\cY\to\cX$ with $\cY\langle F\rangle\ne\emptyset$
is of the form $\B H\to\B G$ for a proper subgroup scheme $H$ of
$G$. From the commutativity and from Example \ref{exa:BG tw sectors},
the map $\pi_{0}(\cJ_{0}(\B H))\to\pi_{0}(\cJ_{0}(\B G))$ is identified
with 
\[
\Hom(\widehat{\mu}(\overline{F}),H(\overline{F}))/\Gamma_{F}\to\Hom(\widehat{\mu}(\overline{F}),G(\overline{F}))/\Gamma_{F}.
\]
This is injective, which implies that $\cY\to\cX$ is not a breaking
thin morphism.
\end{proof}
\begin{rem}
Suppose that $\cX=\B G$ with $G$ a commutative finite group scheme
over $F$. From Corollary \ref{cor:no-breaking}, the subset $T$
in Conjecture \ref{conj:general} is empty. In this case, Conjecture
\ref{conj:general} holds (see \cite{darda2022torsors} and references
therein).
\end{rem}

\begin{example}
\label{exa:weakly-breaking}Let $p$ be a prime number and let $\cX=\B\mu_{p^{2},F}$.
We can identify $\pi_{0}(\cJ_{0}\cX)$ with $\ZZ/p^{2}\ZZ$. For $0<c_{1}<c_{2}$,
we define a raising function $c$ by
\[
c(i)=\begin{cases}
0 & (i=0)\\
c_{1} & (i\in p\ZZ/p^{2}\ZZ)\\
c_{2} & (\text{otherwise})
\end{cases}.
\]
Then, the natural morphism $\B\mu_{p,F}\to\cX$ is weakly breaking
thin with respect to $c$.
\end{example}

\begin{rem}
\label{rem:Wood-fairness}For an abelian group $G$, there may exist
a \emph{weakly} breaking thin morphism $\cY\to\B G_{F}$. Wood \cite{wood2010onthe}
introduces the notion of fairness to have a nice leading constant
for counting abelian extensions of a number field. If we adapt this
notion into our language and if a raising function $c$ of $\B G$
for an abelian finite group $G$ is fair, then there is no weakly
breaking thin morphism $\cY\to\B G$ with $\cY\langle F\rangle\ne\emptyset$. 
\end{rem}

\begin{rem}
\label{rem:remove-weak}Let $\cX=\B G$ with $G$ a commutative finite
group scheme over $F$. Then, for any two points $x,x'\in\cX\langle F\rangle$,
there exists an automorphism $\cX\xrightarrow{\sim}\cX$ mapping $x$
to $x'$. It follows that if $x$ is in the image of the map $\cY\langle F\rangle\to\cX\langle F\rangle$
associated to a morphism $f\colon\cY\to\cX$ of stacks, then $x'$
is in the image of the map $\cY\langle F\rangle\to\cX\langle F\rangle$
associated to the composite map $\cY\xrightarrow{f}\cX\xrightarrow{g}\cX$
of $f$ and some automorphism $g$. Thus, if $c$ is a raising function
of $\cX$ preserved by automorphisms of $\cX$ and if $\cX$ contains
a weakly breaking thin subset (see Example \ref{exa:weakly-breaking}),
then $\cX\langle F\rangle$ is covered by weakly breaking thin subsets.
In particular, $\cX\langle F\rangle$ has no strongly secure element.
Note that this pathology does not occur for breaking thin morphisms,
since there is no breaking thin morphism to $\cX$ at all.
\end{rem}

\subsection{Klüners' counterexample revisited\label{sub:Kl=0000FCners'-counterexample}}

In this subsection, we explain Klüners' counterexample to Malle's
conjecture \cite{kluners2005acounter} in our language. Consider the
wreath product 
\[
G:=C_{3}\wr C_{2}=(C_{3}\times C_{3})\rtimes C_{2}.
\]
Here $C_{n}$ denotes the cyclic group of order $n$. This is a group
of order 18 and realized as the transitive subgroup of $S_{6}$ generated
by permutations $(1,2,3)$, $(4,5,6)$, and $(1,4)(2,5)(3,6)$. Recall
that the index function 
\[
\ind\colon G\to\ZZ_{\ge0}
\]
is defined by 
\[
\ind(g):=6-\#\{g\text{-orbits in }\{1,2,\dots,6\}\}.
\]
This function restricted to $G\setminus\{1\}$ takes the minimal value
$2$ exactly on the following four elements 
\begin{equation}
(1,2,3),\,(1,3,2),\,(4,5,6),\,(4,6,5),\label{eq:ind-2}
\end{equation}
all of which have order 3. Note that there are also elements of order
3 having index 4, for example, $(1,2,3)(4,5,6)$. 

Let us now consider the stack $\cX:=\B G_{\QQ}$. Its set of sectors,
$\pi_{0}(\cJ_{0}\cX)$, is identified with the set of $\QQ$-conjugacy
classes, $\QConj(G)$. The four elements of index 2 are divided into
two conjugacy classes $\{(1,2,3),(4,5,6)\}$ and $\{(1,3,2),(4,6,5)\}$.
In turn, these two conjugacy classes form one $\QQ$-conjugacy class.
In summary, there is the unique $\QQ$-conjugacy class of index 2. 

We now discuss possible forms of breaking thin morphisms $\cY\to\cX$
with $\cY\langle F\rangle\ne\emptyset$. The stack $\cY$ needs to
be isomorphic to $\B H'$ where $H'$ is a twisted form of a subgroup
$H\nsubseteq G$ which contains an element of index 2. We can easily
see that any subgroup of order 2 or 6 does not contain an element
of index 2. Thus, the order of $H$ is either 3 or 9. It follows that
$H$ is contained in the unique $3$-Sylow subgroup 
\[
N:=\langle(1,2,3),(4,5,6)\rangle\lhd G.
\]

Let $x\colon\Spec L\to\Spec\QQ$ be a $G$-torsor and let $\Spec K:=\Spec L^{N}\to\Spec\QQ$
be the induced $C_{2}$-torsor. Then, we claim that the normal subgroup
scheme $N'\lhd\ulAut_{\QQ}(x)$ corresponding to $N$ is isomorphic
to the twisted form $N^{(K)}$ of $N_{\QQ}$ induced by 
\begin{equation}
\Gamma_{\QQ}\xrightarrow{\tau_{K}}C_{2}\xrightarrow{\phi}\Aut(N_{\QQ}),\label{eq:Gal}
\end{equation}
where $\tau_{K}$ is the map corresponding to the $C_{2}$-torsor
$\Spec K\to\Spec\QQ$ and $\phi$ is the one used in the definition
of the wreath product $G=C_{3}\wr C_{2}$. Since it induces the $N$-torsor
$\Spec L\to\Spec K$, the group scheme $N'$ is trivialized by the
scalar extension $K/\QQ$; 
\[
N'_{K}\cong N_{K}\cong\overset{9\text{ copies}}{\overbrace{\Spec K\amalg\cdots\amalg\Spec K}}.
\]
The Galois action on this trivial group is again induced by 
\[
\Gamma_{\QQ}\xrightarrow{\tau_{K}}C_{2}\xrightarrow{\phi}\Aut(N_{\QQ})=\Aut(N_{K}).
\]
This shows the claim. Let 
\[
f_{x}\colon\B N^{(K)}\to\B\ulAut_{F}(x)\cong\cX
\]
be the induced morphism. Then, the image of the induced map $f_{x}\langle\QQ\rangle\colon(\B N^{(K)})\langle\QQ\rangle\to\cX\langle\QQ\rangle$
is exactly the set of the $G$-torsors $\Spec M\to\Spec\QQ$ such
that the associated $C_{2}$-torsor $\Spec M^{N}\to\Spec\QQ$ is isomorphic
to $\Spec K\to\Spec\QQ$; namely, the fiber of $\cX\langle\QQ\rangle\to(\B C_{2})\langle\QQ\rangle$
over the isomorphism class of $\Spec K\to\Spec\QQ$. In particular,
the image of $f_{x}\langle\QQ\rangle$ depends only on the isomorphism
class of $K$.

We next show that there is no breaking thin morphism coming from a
group scheme of order three. If $K=\QQ^{2}$, that is, if $\Spec K\to\Spec\QQ$
is the trivial $C_{2}$-torsor, then $N'$ is the constant group $N_{\QQ}$.
It contains (necessarily constant) subgroup schemes of order three
which contains elements of index 2. But, since $\B C_{3,\QQ}$ has
only one twisted sector, any morphism $\B C_{3,\QQ}\to\cX$ is not
breaking thin. If $K$ is a field, then $N$ has six connected components
and is of the form
\[
\Spec\QQ\amalg\Spec\QQ\amalg\Spec\QQ\amalg\Spec K\amalg\Spec K\amalg\Spec K.
\]
The first three components correspond to the diagonal elements of
$N=C_{3}\times C_{3}$, those elements fixed by the $C_{2}$-action.
The other three components correspond to the following $C_{2}$-orbits,
respectively: 
\[
\{(1,2,3),(4,5,6)\},\,\{(1,3,2),(4,6,5)\},\,\{(1,2,3)(4,6,5),(1,3,2)(4,5,6)\}.
\]
The only subgroup schemes of $N$ of order three are the diagonal
one $\Spec\QQ\amalg\Spec\QQ\amalg\Spec\QQ$ and the union $\Spec\QQ\amalg\Spec K$
of the identity component $\Spec\QQ$ and the component $\Spec K$
corresponding to 
\[
\{(1,2,3)(4,6,5),(1,3,2)(4,5,6)\}.
\]
The diagonal elements, $(1,2,3)(4,6,5)$ and $(1,3,2)(4,5,6)$, do
not have index 2. Thus, the induced morphisms $\B C_{3,\QQ}\to\cX$
and $\B(\Spec\QQ\amalg\Spec K)\to\cX$ are not breaking thin. In conclusion,
there is no breaking thin morphism of the form $\B H'\to\cX$, where
$H'$ is of order three. 

It remains possible that $H'$ is of order nine and isomorphic
to $N^{(K)}$, the twisted form of $N$ associated to a $C_{2}$-torsor
$\Spec K\to\Spec\QQ$. Since $N$ contains elements of index 2, the
morphism $\B N^{(K)}\to\cX$ is always a weakly breaking thin morphism.
To see when it is also breaking thin, we now compute $\pi_{0}(\cJ_{0}(\B N^{(K)}))$.
Recall that this set is identified with 
\[
\Hom(\mu_{9}(\overline{\QQ}),N^{(K)}(\overline{\QQ}))/\Gamma_{\QQ}.
\]
Fixing a generator of $\mu_{9}(\overline{\QQ})$, we may identify
\[
\Hom(\mu_{9}(\overline{\QQ}),N^{(K)}(\overline{\QQ}))=N^{(K)}(\overline{\QQ})=N=C_{3}\times C_{3}.
\]
The $\Gamma_{\QQ}$-action on this set is obtained by combining two
actions; the one induced from the $\Gamma_{\QQ}$-action on $\mu_{9}(\overline{\QQ})$
and the one induced from the action on $N^{(K)}(\overline{\QQ})$.
Through the above identification, the first action transitively interchanges
a nonidentity element $g\in N$ with its square $g^{2}$. The second
action is given by (\ref{eq:Gal}). We see that the four elements
of index 2, listed in (\ref{eq:ind-2}), are transitively permuted
by the combined $\Gamma_{\QQ}$-action except the two cases; the case
$K=\QQ^{2}$ and $K=\QQ(\zeta_{3})$ with $\zeta_{3}$ a primitive
cubic root of $1$. In the former case, the second action is trivial,
and hence the four elements of index 2 are divided into the two $\Gamma_{\QQ}$-orbits
\[
\{(1,2,3),\,(1,3,2)\},\,\{(4,5,6),(4,6,5)\}.
\]
In the latter case, since $\mu_{3}=\Spec\QQ\amalg\Spec\QQ(\zeta_{3})$,
the two $\Gamma_{\QQ}$-actions are \emph{synchronized}. Namely, whenever
$\gamma\in\Gamma_{F}$ maps $(1,2,3)\leftrightarrow(1,3,2)$ and $(4,5,6)\leftrightarrow(4,6,5)$
by the first action, then the same element $\gamma$ maps $(1,2,3)\leftrightarrow(4,5,6)$
and $(1,3,2)\leftrightarrow(4,5,6)$ by the second action. Similarly,
whenever $\gamma\in\Gamma_{F}$ fixes the four elements by the first
action, then the same element fixes them also by the second action.
Consequently, the combined action divides the four elements into two
orbits
\[
\{(1,2,3),\,(4,6,5)\},\,\{(1,3,2),(4,5,6)\}.
\]
Thus, the morphisms $\B N^{(\QQ^{2})}\to\cX$ and $\B N^{(\QQ(\zeta_{3}))}\to\cX$
are the only breaking thin morphisms. In particular, the union $T$
of breaking thin subsets of $\cX\langle\QQ\rangle$, which appears
in Conjecture \ref{conj:general}, is a thin subset. Note that the
image of $(\B N^{(\QQ^{2})})\langle\QQ\rangle\to\cX\langle\QQ\rangle$
contains only non-connected torsors, which are removed also in the
usual version of Malle's conjecture. We can summarize these results as follows:
\begin{prop}
For a $C_{2}$-torsor $\Spec K\to\Spec\QQ$, let $T_{K}$ be the image
of the map $(\B N^{(K)})\langle\QQ\rangle\to\cX\langle\QQ\rangle$.
Then, $\cX\langle\QQ\rangle=\bigsqcup_{K}T_{K}$. Moreover, we have:
\begin{enumerate}
\item Every $T_{K}$ is a weakly breaking thin subset. 
\item The subset $T_{K}$ is a breaking thin subset exactly when $K$ is
isomorphic to $\QQ^{2}$ or $\QQ(\zeta_{3})$. 
\item The subset $\bigsqcup_{K\ne\QQ^{2},\QQ(\zeta_{3})}T_{K}\subset\cX\langle\QQ\rangle$
as well as any cothin subset contained in it is a secure cothin subset.
\item There is no strongly secure subset of $\cX\langle\QQ\rangle$. 
\end{enumerate}
\end{prop}

\begin{rem}
Türkelli \cite{turkelli2015connected} takes a different approach
to modify Malle's conjecture incorporating Klüners' counterexample;
he proposes changing the exponent of the log factor instead of keeping
the exponent unchanged and removing a thin subset.
\end{rem}

\subsection{Comprehensiveness\label{subsec:Comprehensiveness}}
\begin{defn}
Let $G$ be a finite group and let $c\colon\Conj(G)\to\RR_{\ge0}$
be a function such that for $C\in\Conj(G)$, $c(C)=0$ if and only
if $C=[1]$. We say that $G$ is $c$\emph{-comprehensive }if for
any non-identity conjugacy class $C\in\Conj(G)$ with the minimal
$c$-value, the elements in $C$ generate $G$.
\end{defn}

\begin{example}
The subgroup generated by the elements of a conjugacy class is a normal
subgroup. Therefore, a simple subgroup is $c$-comprehensive for any
$c$.
\end{example}

\begin{example}
For the symmetric group $S_{n}$, let us consider the index function,
$\ind\colon\Conj(S_{n})\to\ZZ_{\ge0}$. This takes the minimal value
only at the conjugacy class of all transpositions. Since $S_{n}$
is generated by transpositions, $S_{n}$ is $\ind$-comprehensive.
\end{example}

\begin{prop}
\label{prop:comprehensive}Let $G$ be a finite group, let $\cX:=\B G_{F}$
and let $x\colon\Spec L\to\Spec F$ be a connected $G$-torsor, that
is, an $F$-point of $\cX$. Let $c\colon\pi_{0}(\cJ_{0}\cX)=\FConj(G)\to\RR_{\ge0}$
be a raising function. Denoting the composition $\Conj(G)\to\FConj(G)\xrightarrow{c}\RR_{\ge0}$
again by $c$, we suppose that $G$ is $c$-comprehensive. Then, $x$
is a secure element of $\cX\langle F\rangle.$ Namely, the subset
of $\cX\langle F\rangle$ consisting of all connected $G$-torsors
is a strongly secure cothin subset. In particular, the union of weakly
breaking thin subsets of $\cX\langle F\rangle$ as well as the one
of breaking thin subsets is a thin subset. 
\end{prop}

\begin{proof}
As for the first assertion, it suffices to show that the twisted form
$\ulAut_{F}(x)$ of $G_{F}^{\op}$ does not contain a proper subgroup
scheme containing a $c$-minimal element. Since $x$ is connected,
the map $\phi_{x}\colon\Gamma_{F}\to G$ corresponding to $x$ is
surjective. From Lemma \ref{lem:Aut-twist}, 
\[
\pi_{0}(\ulAut_{F}(x))=\ulAut_{F}(x)(\overline{F})/\Gamma_{F}=\Conj(G).
\]
From the $c$-comprehensiveness, if a subgroup scheme $H$ of $\ulAut_{F}(x)$
contains a $c$-minimal element, then it is in fact the entire group
scheme $\ulAut_{F}(x)$. Thus, the first assertion of the proposition
holds. 

The second and the third assertions follow from the fact that the
set of non-connected $G$-torsors over $F$ is the union of the images
of the maps $(\B H_{F})\langle F\rangle\to(\B G_{F})\langle F\rangle$
associated to proper subgroups $H\subsetneq G$.
\end{proof}
\begin{rem}
For a $c$-comprehensive group $G$, the stack $\cX=\B G_{F}$ may
have a breaking thin morphism. For example, let $n\ge4$ and let $H:=\langle(1,2),(3,4)\rangle\subset S_{n}$.
The induced morphism $\B H_{F}\to\B(S_{n})_{F}$ is breaking thin.
However, every $F$-point in its image is a non-connected $S_{n}$-torsor. 
\end{rem}

\bibliographystyle{alpha}
\bibliography{manin-malle}

\end{document}